\documentclass[11pt]{article}

\usepackage[margin=1.17in]{geometry}

\usepackage{amssymb,amsmath}
\usepackage{amsthm}
\usepackage{hyperref}
\usepackage{mathrsfs}
\usepackage{textcomp}
\hypersetup{
    colorlinks,%
    citecolor=black,%
    filecolor=black,%
    linkcolor=black,%
    urlcolor=black
}

\usepackage{verbatim}
\usepackage{sectsty}

\usepackage[titletoc,toc,title]{appendix}

\makeatletter

\newdimen\bibspace
\renewenvironment{thebibliography}[1]{%
 \section*{\refname 
       \@mkboth{\MakeUppercase\refname}{\MakeUppercase\refname}}%
     \list{\@biblabel{\@arabic\c@enumiv}}%
          {\settowidth\labelwidth{\@biblabel{#1}}%
           \leftmargin\labelwidth
           \advance\leftmargin\labelsep
           \itemsep\bibspace
           \parsep\z@skip     %
           \@openbib@code
           \usecounter{enumiv}%
           \let\p@enumiv\@empty
           \renewcommand\theenumiv{\@arabic\c@enumiv}}%
     \sloppy\clubpenalty4000\widowpenalty4000%
     \sfcode`\.\@m}
    {\def\@noitemerr
      {\@latex@warning{Empty `thebibliography' environment}}%
     \endlist}

\makeatother

\makeatletter

\newtheorem{thm}{Theorem}[section]
\newtheorem{lem}[thm]{Lemma}
\newtheorem{prop}[thm]{Proposition}
\newtheorem{defn}[thm]{Definition}

\newtheorem{cor}[thm]{Corollary}
\newtheorem{rem}[thm]{Remark}

\newtheorem*{rem*}{Remark}

\def\Xint#1{\mathchoice
  {\XXint\displaystyle\textstyle{#1}}%
  {\XXint\textstyle\scriptstyle{#1}}%
  {\XXint\scriptstyle\scriptscriptstyle{#1}}%
  {\XXint\scriptscriptstyle\scriptscriptstyle{#1}}%
  \!\int}
\def\XXint#1#2#3{{\setbox0=\hbox{$#1{#2#3}{\int}$}
  \vcenter{\hbox{$#2#3$}}\kern-.5\wd0}}

\def\dashint{\Xint-}

\newcommand{\al}{\alpha}                \newcommand{\lda}{\lambda}
\newcommand{\om}{\Omega}                \newcommand{\pa}{\partial}
\newcommand{\va}{\varepsilon}           \newcommand{\ud}{\mathrm{d}}
\newcommand{\be}{\begin{equation}}      \newcommand{\ee}{\end{equation}}
                 
\newcommand{\Lda}{\Lambda}              
\newcommand{\R}{\mathbb{R}}

\begin{document}

\title{\textbf{Optimal boundary regularity for fast diffusion equations in bounded domains}
\bigskip}

\author{\medskip  Tianling Jin\footnote{T. Jin was partially supported by Hong Kong RGC grants ECS 26300716 and GRF 16302217.}\quad and \quad
Jingang Xiong\footnote{J. Xiong was partially supported by NSFC 11501034, a key project of NSFC 11631002 and NSFC 11571019.}}

\date{\today}

\maketitle

\begin{abstract} We prove optimal boundary regularity for bounded positive weak solutions of fast diffusion equations in smooth bounded  domains. This solves a problem raised by Berryman and Holland in 1980 for these equations in the subcritical and critical regimes. Our proof of the a priori estimates uses a geometric type structure of the fast diffusion equations, where an important ingredient is an evolution equation for a curvature-like quantity. 

\medskip


\medskip

\end{abstract}

\section{Introduction} 

Let $\om $ be a smooth bounded domain in $\R^n$ with $n\ge 1$. We consider the fast diffusion equation
\be \label{eq:main}
\pa_t u^p -\Delta u=0 \quad \mbox{in }\om \times (0,\infty)
\ee
with the Dirichlet condition
\be \label{eq:main-d}
u=0 \quad \mbox{on }\pa \om \times (0,\infty)
\ee
in the range of subcritical and critical Sobolev exponents, i.e., $1<p<\infty$ if $n=1,2$ and $1<p\le \frac{n+2}{n-2}$ if $n\ge 3$.  The fast diffusion equations arise in the modelling of gas-kinetics,  plasmas, thin liquid film dynamics driven by Van der Waals forces and etc.  In particular, the case $p=2$ corresponds to the scaling predicted for Okuda-Dawson diffusion; see Okuda-Dawson \cite{OD} and Drake-Greenwood-Navratil-Post \cite{DGNP}.  The critical case $p=\frac{n+2}{n-2}$ in dimension $n\ge 3$ describes the evolution of a conformal metric by the Yamabe flow.

It is well known that the solutions of \eqref{eq:main} will be extinct in finite time, which can date back to the work of Sabinina \cite{Sabinina1,Sabinina2}. We denote $T^*$ as the extinction time. In the classical paper \cite{BH}, Berryman-Holland established the asymptotic behavior of solutions to \eqref{eq:main} near the extinction time $T^*$ in the $H^1_0(\om)$ space, assuming that $\partial_t u, \nabla u, \nabla\partial_t u, \nabla^2 u\in C(\overline \om\times (0,T^*))$. However,  this \emph{a priori} smoothness condition was left as an assumption, and was listed as the first  unsolved problem in their paper. 

The theory of weak solutions of \eqref{eq:main} has been studied extensively; see B\'enilan-Crandall \cite{BC}, Brezis-Friedman \cite{BF},  Herrero-Pierre \cite{HP},   Pierre \cite{P}, Dahlberg-Kenig \cite{DKenig},  Chasseigne-V\'azquez \cite{CV, CV0} and see also the monographs of  Daskalopoulos-Kenig  \cite{DaK} and V\'azquez \cite{Vaz,Vaz1}. In this paper, we say that $u$ is a  weak solution of \eqref{eq:main} and \eqref{eq:main-d} if $u\in L^2_{loc}([0,\infty); H^1_0(\Omega))$, $u\ge 0$, $u^p\in C([0,\infty);L^{1}(\Omega))$, and satisfies 
\[
\int_{\Omega\times(0,\infty)} \Big(u^p\pa_t\varphi -\nabla u\cdot\nabla\varphi\Big)\,\ud x\ud t=0\quad\forall\,\varphi\in C_c^1(\Omega\times(0,\infty)).
\]

As for the regularity of nonnegative (nontrivial) weak solutions to \eqref{eq:main} and \eqref{eq:main-d},  the following results were proved: 

\begin{itemize}
\item If $u(\cdot,0)\in L^q(\Omega)$, with $q\ge p$ when $p\in (1,\frac{n}{n-2})$, and $q>\frac{n(p-1)}{2}$ when $p>\frac{n}{n-2}$, then the so-called smoothing effect holds true
\[
0\le u(\cdot, t)\le c \frac{\|u(\cdot,0)\|^{2q\nu}_{L^{q}(\Omega)}}{t^{n\nu}}\quad \mbox{for all }t>0,\ 
\]
where $\nu=\frac{1}{2q-n(p-1)}$. The above estimates can be found in \cite{BCV,BV,DaK, DGV,DK,DKV,Vaz} for instance. 

\item Moreover, there exists the so-called finite extinction time, i.e. a time $T^*$ such that $u(\cdot, t)\equiv 0$ for all $t>T^*$, and $u(\cdot, t)>0$ for $t<T^*$. For more precise local lower bounds and Harnack inequalities we refer to DiBenedetto-Gianazza-Vespri \cite{DGV,DGV-b} and Bonforte-V\'azquez \cite{BV}.

\item Bounded solutions are H\"older continuous up to the $\pa \om \times (0,T^*)$ and smooth in the interior; see Chen-DiBenedetto \cite{CDi}. Continuity of bounded solutions were proved earlier by, e.g., DiBenedetto \cite{DiBenedetto}, Sacks \cite{Sacks} and Kwong \cite{KwongY}.

\item  Suppose that $1<p<\frac{n+2}{n-2}$ if $n\ge 3$, and $1<p<\infty$ if $n=1,2$.  Then it has been proved in DiBenedetto-Kwong-Vespri \cite{DKV} that for $0<\delta\le t\le T^*$,   there hold
\begin{equation}\label{eq:dkvsubcritical}
\frac{1}{c_0} d (x)(T^*-t)^{\frac{1}{p-1}} \le u(x,t)  \le c_0 d (x)(T^*-t)^{\frac{1}{p-1}} , 
\end{equation}   
\[
|D_x^k  u(x,t)|  \le C_1 d (x)^{1-k}(T^*-t)^{\frac{1}{p-1}}, \quad k=1,2,\cdots,
\]
and
\[
|\pa_t^k  u(x,t)| \le C_1 d (x)^{1-k(p+1)}(T^*-t)^{\frac{1}{p-1}-k},
\]
where  $d(x)=dist(x,\pa \om)$, $c_0\ge 1$ depends only on  $\|u(\cdot,0)\|_{H^{1}_0(\om)}$, $\|u(\cdot,0)\|_{L^{p+1}(\om)}$, $n$, $\om$, $\delta$  and $p$; and the constant $C_1>0$ depends only on $\|u(\cdot,0)\|_{H^{1}_0(\om)}$, $\|u(\cdot,0)\|_{L^{p+1}(\om)}$, $n$, $\om$, $\delta$, $p$ and also $k$. The formula \eqref{eq:dkvsubcritical} is called global Harnack principle in \cite{DKV} and the subsequent literature. Furthermore, the estimate \eqref{eq:dkvsubcritical}  has been improved in Bonforte-Grillo-V\'azquez \cite{BCV}, where the estimate \eqref{eq:relativeerror} in the below is proved.

If $p=\frac{n+2}{n-2}$, then the proof in \cite{DKV} also implies that for $ (x,t) \in \om \times (\delta, T^*-\delta)$,   
\begin{equation}\label{eq:dkvcritical}
\frac{1}{\tilde c_0}d(x)\le u(x,t) \le \tilde c_0 d(x), 
\end{equation}   
where $\tilde c_0\ge 1$ depends only on $n,\om, \|u\|_{L^\infty(\Omega\times(0,T^*-\delta))}, \delta $ and $ p$.

\end{itemize}
 
From the above result of \cite{DKV} , we have  $|\pa_t u(x, t)|\le C d(x)^{-p}\notin L^1(\om)$ for any fixed $0<t<T^*$, due to $p>1$. And the derivation in Berryman-Holland \cite{BH}  of the monotonicity of the Dirichlet integral for bounded weak solutions 
\be \label{eq:energy-identity}
\frac{\ud }{\ud t} \int_\om |\nabla u(\cdot,t)|^2= -2 \int_{\om} \nabla u \nabla \pa_t u = -2p \int_{\om }u^{p-1} |\pa_t u|^2
\ee 
is in question. In fact, only the energy inequality
\be\label{eq:energy-ineq}
\begin{split}
&\int_\om |\nabla u(\cdot,t_1)|^2- \int_\om |\nabla u(\cdot,t_2)|^2 + 2p\int_{t_2}^{t_1} \int_{\om }u^{p-1} |\pa_t u|^2  \,\ud x\ud t\\
&\quad \le 0, \quad \forall~ 0<t_2<t_1<T^*
\end{split}
 \ee  
was verified without regularity assumptions; see Feireisl-Simondon \cite{FS}.    
 
In this paper, we  prove the a priori smoothness hypothesis in Berryman-Holland \cite{BH}. 

\begin{thm}\label{thm:main} Suppose $1<p<\infty$ if $n=1$ or $2$,  and $1<p\le \frac{n+2}{n-2}$ if $n\ge 3$.   Let $u$ be a bounded nonnegative weak solution of \eqref{eq:main} and \eqref{eq:main-d}, and $T^*>0$ be its extinction time.  If $p$ is an integer, then $u\in C^\infty(\overline \om\times (0,T^*))$. If $p$ is not an integer, then $u(x,\cdot)\in C^\infty(0,T^*)$ for every $x\in\overline\Omega$, and $ \pa_t^l u(\cdot,t) \in C^{2
+p}(\overline \om)$ for all $l\ge 0$ and all $t\in(0,T^*)$.  
\end{thm}

The sharpness of the regularity of the solutions in the $x$ variable when $p$ is not an integer can be seen from the exact solution $u(x,t)=[(p-1)(1-t)/p]^{\frac{1}{p-1}} S(x)$, where $S$ is a positive solution of \eqref{eq:steady}.  If $p$ is not an integer, then $S$ is precisely of $C^{2+p}(\overline\Omega)$. 

Now we can justify \eqref{eq:energy-identity} for weak solutions, and thus,  \eqref{eq:energy-ineq} becomes an identity.   

The structure of our proof of this regularity is a combination of a priori estimates and a short time existence theorem. The difficulties and our ideas of each of these two parts are explained in order as follows:

\smallskip

(i). First, we  note that neither the H\"older regularity of the solution of \eqref{eq:main}, nor the estimate \eqref{eq:dkvsubcritical} (or \eqref{eq:dkvcritical}), is sufficient to bootstrap its regularity. Possible bootstrap should require that $u(x,t)/d(x)$ is in some, say, H\"older spaces, which will require some global H\"older estimates for $\nabla u$. Secondly, we want to derive a priori estimates by differentiating the equation. For the equation \eqref{eq:main}, differentiating it in the time variable is more natural, because this will keep the structure of the equation as well as the zero boundary condition. If we let $v=u_t$, then
\[
pu^{p-1} \partial_t v+p(p-1)u^{p-2}  v^2=\Delta v.
\]
With the help of  the nonnegative term $p(p-1)u^{p-2}v^2$, using the De Giorgi-Nash-Moser iteration we can show that $\partial_t u=v\le C$. The difficulty is to show the lower bound of $\partial_t u$. Here, we overcome this difficulty by proving arbitrarily high integrability of $\partial_t u/u$, which is our main contribution in this part. To do this, we construct a curvature-like quantity, and derive its evolution equation. This evolution equation extends the scalar curvature evolution  equation along the Yamabe flow or the scalar curvature flow; see Schwetlick-Struwe  \cite{SS}, Brendle \cite{Br05} and Chen-Xu \cite{CX}. Then we prove the arbitrarily high integrability estimates by an ODE argument. These a priori estimates depend only on the constant $c_0$ in \eqref{eq:dkvsubcritical} or $\tilde c_0$ in \eqref{eq:dkvcritical}, and also the $H^1_0$ norm of the initial data. In this step, our proof requires the exponent $p$ to be subcritical or critical. 

\smallskip

(ii). For the short time existence of regular (up to the boundary $\partial\Omega$) solutions to \eqref{eq:main}, we need to construct a suitable set of initial data, which are dense in $H^1_0$, and establish a Schauder theory for a class of singular parabolic equations, so that the linearized operator for \eqref{eq:main} will be invertible in these spaces and we can apply the implicit function theorem.  The Schauder theory is proved in a well-designed Campanato space.   When applying the implicit function theorem, we need a second order approximating solution. The reason is that the equation after differentiating \eqref{eq:main} in $t$ is still nonlinear (semi-linear) and singular, and we would not be able to bootstrap the regularity if the approximation is of first order. We need the regularity of the solutions in short time to be high enough to carry out our proof of a priori estimates. The short time existence theorem holds for all $p\in (1,\infty)$. A Schauder theory for a class of degenerate parabolic equations and short time existence of smooth solutions of porous medium equations $(0<p<1)$ with compact support smooth initial data were established by Daskalopoulos-Hamilton \cite{DH}. Since our equations are singular parabolic,  our proofs are different from theirs.  

\smallskip

The equation \eqref{eq:main} with $p\in (1,\infty)$ has infinite speed of propagation. That is, if $u(\cdot,0)\not\equiv 0$ is nonnegative and compactly supported in $\Omega$, then the solution $u(\cdot,t)$ will  immediately become positive everywhere in $\Omega$. However, if $p\in (0,1)$, for which the equation \eqref{eq:main} becomes the so-called  porous medium equation,  the solution $u(\cdot,t)$ with such initial data will still be compactly supported in $\Omega$ at least for a short time, and thus, $\partial\{u>0\}$ is a free boundary. H\"older continuity of the solutions to porous medium equations and their free boundaries were proved by Caffarelli-Friedman \cite{CF}. Their higher regularities were obtained under extra assumptions, such as a non-degeneracy condition of the initial data by Caffarelli-V\'azquez-Wolanski \cite{CVW}, Caffarelli-Wolanski \cite{CW}, Koch \cite{Koch}, Daskalopoulos-Hamilton-Lee \cite{DHL}, or a flatness assumption of the solution by Kienzler-Koch-V\'azquez \cite{KKV}.

The solution in Theorem \ref{thm:main} also satisfies the following quantitative estimates. When $p$ is a subcritical Sobolev exponent, we have the local boundedness estimate of the solution, and can use the  explicit bound \eqref{eq:dkvsubcritical} to obtain the estimate up to the extinction time.

\begin{thm}[Subcritical case]\label{thm:main2} 
Let $1<p<\infty$ for $n=1, 2$,  or $1<p< \frac{n+2}{n-2}$ for $n\ge 3$.  Let $u$ be a bounded nonnegative weak solution of \eqref{eq:main} and \eqref{eq:main-d}, and $T^*>0$ be its extinction time.  

 If $p$ is an integer, then we have, for $ (x,t) \in \overline\om \times (\delta, T^*)$ with $0<\delta<T^*/4$ that
 \[
d(x)^{-1}|\pa_t^l  u(x,t)| +\|D_x^k\pa_t^l   u(\cdot,t)\|_{L^{\infty}(\overline\Omega)}  \le C(T^*-t)^{\frac{1}{p-1}-l},
\] 
where $C$ depends only on $n,\om, l, p,\delta, k$, $\|u(\cdot,0)\|_{H^{1}_0(\om)}$ and $\|u(\cdot,0)\|_{L^{p+1}(\om)}$.

If $p$ is not an integer, then we have, for $ (x,t) \in \overline\om \times (\delta, T^*)$ with $0<\delta<T^*/4$ that
 \[
d(x)^{-1}|\pa_t^l  u(x,t)| +\|\pa_t^l  u (\cdot,t)\|_{C^{2+p}(\overline\Omega)}  \le \widetilde{C} (T^*-t)^{\frac{1}{p-1}-l}, 
\]   
where $\widetilde{C}$ depends only on $n,\om, l, p,\delta$, $\|u(\cdot,0)\|_{H^{1}_0(\om)}$ and $\|u(\cdot,0)\|_{L^{p+1}(\om)}$.
\end{thm}  

When $p$ is the critical Sobolev exponent, the rescaled solution in the time variable (that is, the $v$ in the equation \eqref{eq:rescaling-main} in the below) may blow up as the time approaches to infinity (see, e.g., Galaktionov-King \cite{GKing} and Sire-Wei-Zheng \cite{SWZ}), and therefore, we will not have the same uniform estimate of $u$ up to the extinction time as that for the subcritical Sobolev case.

\begin{thm}[Critical case]\label{thm:main3}  
Let $n\ge 3$, $p= \frac{n+2}{n-2}$, $u$ be a bounded nonnegative weak solution of \eqref{eq:main} and \eqref{eq:main-d}, and $T^*>0$  be its extinction time.  

 If $p$ is an integer, then we have, for $ (x,t) \in \overline\om \times (\delta, T^*-\delta)$ with $0<\delta<T^*/4$ that
 \[
d(x)^{-1}|\pa_t^l  u(x,t)| +\|D_x^k\pa_t^l   u(\cdot,t)\|_{L^{\infty}(\overline\Omega)}  \le C,
\] 
where $C>0$ depends only on $n,\om, l, p,\delta, k$, $\|u(\cdot,0)\|_{H^{1}_0(\om)}$, $\|u\|_{L^\infty(\Omega\times(0,T^*-\delta))}$ and  $\|u(\cdot,0)\|_{L^{p+1}(\om)}$.

If $p$ is not an integer, then we have, for $ (x,t) \in \overline\om \times (\delta, T^*-\delta)$ with $0<\delta<T^*/4$ that
 \[
d(x)^{-1}|\pa_t^l u (x,t)| +\|\pa_t^l u  (\cdot,t)\|_{C^{2+p}(\overline\Omega)}  \le \widetilde{C} 
\]   
where $\widetilde{C}>0$ depends only on $n,\om, l, p,\delta$, $\|u(\cdot,0)\|_{H^{1}_0(\om)}$,  $\|u\|_{L^\infty(\Omega\times(0,T^*-\delta))}$ and $\|u(\cdot,0)\|_{L^{p+1}(\om)}$.      
\end{thm}

In the end, we give an application of  Theorems \ref{thm:main} and \ref{thm:main2}.   
Let  
\[
v(x,t)=\Big (\frac{p}{(p-1)(T^*-\tau)}\Big)^{\frac{1}{p-1}} u(x,\tau), \quad t =\frac{(p-1)}{p}   \ln \Big(\frac{T^*}{T^*-\tau}\Big). 
\]
Then
\be \label{eq:rescaling-main}
\pa_t v^p= \Delta v + v^{p} \quad \mbox{in }\om \times (0,\infty).
\ee
If $p$ is a subcritical exponent, then
\be  \label{eq:rescaling-main-b}
\frac{1}{c_0} d(x) \le v(x,t) \le  c_0d(x)  \quad \mbox{in }\om \times (1,\infty), 
\ee
where $c_0>1$ is independent of $t$.   Assuming that $\partial_t u, \nabla u, \nabla\partial_t u, \nabla^2 u\in C(\overline \om\times (0,T^*))$,  Berryman-Holland \cite{BH} proved that $v(\cdot,t)\to S$ in $H_0^1(\om)$ along a subsequence, where $S$ is a positive solution of 
\be \label{eq:steady}
-\Delta S=S^{p} \quad \mbox{in }\om\quad\mbox{and}\quad
S=0 \quad \mbox{on }\pa \om.
\ee   
Feireisl-Simondon \cite{FS} removed the regularity assumption and proved  that $v(\cdot,t)\to S$ in $C^0(\overline \om)$ as $t\to \infty$.  Bonforte-Grillo-V\'azquez \cite{BCV} proved that 
 \begin{equation}\label{eq:relativeerror}
\lim_{t\to \infty}\Big\| \frac{v(\cdot,t)}{S}-1\Big\|_{L^\infty(\om)}  = 0.
 \end{equation}
  Recently,  Bonforte-Figalli \cite{BFig} proved that for generic bounded domains $\om$  there exists   $\gamma_p>0, C>0 $ such that 
  \begin{equation}\label{eq:sharpgammap}
  \int_{\Omega}\left|\frac{v(x,t)}{S(x)}-1\right|^2 S^{p+1}(x)\,\ud x\le C e^{-2\gamma_pt},
  \end{equation}
  and the decay rate $\gamma_p>0$ is sharp. They also prove that 
  \[
\Big\| \frac{v(\cdot,t)}{S}-1\Big\|_{L^\infty(\om)}  \le Ce^{-\frac{\gamma_p}{4n} t}. 
 \]
 See \cite{BFig} for the meaning of  generic domains. As mentioned in \cite{BFig}, examples of such domains includes balls and most $C^{2,\alpha}$ perturbations of annuli, but do not include certain annuli themselves. Combing with our main theorem, we can show the sharp exponential  convergence rate in regular norms.

\begin{cor} \label{cor:main} 
Suppose $1<p<\infty$ if $n=1$ or $2$,  and $1<p< \frac{n+2}{n-2}$ if $n\ge 3$. 
Suppose $v$ is a solution of \eqref{eq:rescaling-main} satisfying  \eqref{eq:rescaling-main-b}  and \eqref{eq:relativeerror}.  We have that if $p$ is an integer, then
\[
\lim_{t\to \infty} \left\|\frac{v(\cdot,t)}{S}-1\right\|_{C^{k}(\overline \om)}=0
\] 
for any positive integer $k$; and if $p$ is not an integer, then
\[
\lim_{t\to \infty} \left\|\frac{v(\cdot,t)}{S}-1\right\|_{C^{1+p}(\overline \om)}=0.
\] 
And, for generic smooth domains $\om$, if $p$ is an integer, then for any positive integer $k$, there exists $C(k)>0$ such that
 \begin{equation}\label{eq:generalconvergesto0}
\left\|\frac{v(\cdot,t)}{S}-1\right\|_{C^{k}(\overline \om)}  \le C(k) e^{-\gamma_p t}\ \mbox{for all large }t,
 \end{equation}
 where $\gamma_p>0$ is the one in \eqref{eq:sharpgammap}.  If $p$ is not an integer, then
 \begin{equation}\label{eq:genericexp}
\left\|\frac{v(\cdot,t)}{S}-1\right\|_{C^{1+p}(\overline \om)} \le C e^{-\gamma_p t}\ \mbox{for all large }t.
 \end{equation}
 \end{cor}

 The Sobolev critical case $p=\frac{n+2}{n-2}$ when $n\ge 3$ is very intriguing. Pohozaev \cite{Poh} proved the non-existence of positive solutions of \eqref{eq:steady} if $p\ge \frac{n+2}{n-2}$ and $\om$ is star-shaped.   On the other hand,  if the topology of $\om$ is non-trivial, the existence of positive solutions of \eqref{eq:steady} was obtained by Bahri-Coron \cite{Bahri-C}. Regardless the topology of $\om$, Brezis-Nirenberg \cite{BN} proved the existence of solutions to
\be 
 -\Delta S-bS=S^{\frac{n+2}{n-2}} \quad \mbox{in }\om \quad \mbox{and} \quad S=0 \quad \mbox{on }\pa \om
 \ee
in dimension $n\ge 4$,  where $0<b<\lda_1$ and $\lda_1$ is the first Dirichlet eigenvalue of $-\Delta$ in $\om$. Druet \cite{D} proved the existence of solution in dimension $3$, under a positive mass type  assumption on the Green function of the operator  $-\Delta -b$.  In our preprint \cite{JX20}, we  studied the extinction behavior of positive solutions of the fast diffusion equation  $\pa_t u^{\frac{n+2}{n-2}}=\Delta u +bu$ with the Dirichlet condition  \eqref{eq:main-d}, where the curvature-like quantity we defined is crucial in carrying out the concentration compactness.

 The paper is organized as follows. In Section \ref{sec:linear}, we will establish regularity results for a class of singular parabolic equations, which fit for the linearized equations of the fast diffusion equations. There is no restriction on $p$ for all dimensions. 
To prove a Schauder type estimates, we introduce a Campanato space matching  the scaling of the singular parabolic equation.   In Section \ref{sec:short-time}, we construct local in time  smooth solutions of the fast diffusion equation for some initial data, the set of which is dense in $H^1_0(\om)$.  
In Section \ref{sec:time-derivative}, we establish the crucial a priori estimates for the time derivatives, where we use a curvature-like evolution equation.   In Section \ref{sec:last}, we prove our main theorems and Corollary \ref{cor:main}.

\bigskip

\noindent \textbf{Acknowledgement:} Part of this work was completed while the second named author was visiting the Hong Kong University of Science and Technology and University of Rome Tor Vergata, to which he is grateful for providing  the very stimulating research environments and supports. Both authors would like to thank Professor YanYan Li for his interests and constant encouragement. Finally, we would like to thank the anonymous referee for his/her careful reading of the paper and invaluable suggestions that greatly improved the presentation and the quality of the paper.

\section{Linear singular parabolic equations} 
\label{sec:linear}

This section is devoted to the well-posedness  of the linearized equation of \eqref{eq:main}. Let $\omega$ be a smooth function in $\overline\Omega$ comparable to the distance function $d(x)$, that is, $0<\inf_{\Omega}\frac{w}{d}\le \sup_{\Omega}\frac{w}{d}<\infty$. For example, $w$ can be taken as the normalized first eigenfunction of $-\Delta$ in $\Omega$ with Dirichlet zero boundary condition. Because of \eqref{eq:dkvsubcritical} and \eqref{eq:dkvcritical}, the linearized equation of \eqref{eq:main} will be of the form:
\begin{align} \label{eq:general}
a(x,t) \omega(x)^{p-1} \pa_t u-\mathrm{div}(A(x,t) \nabla u)&+b_1(x,t)u+b_2 (x,t) \omega(x)^{p-1}u  \nonumber\\
&= \omega(x)^{p-1}f \quad \mbox{in }\Omega \times(-1,0],
\end{align}
where $A(x,t)=(a_{ij}(x,t))_{n\times n}$ is a symmetric matrix. We suppose that $p>1$ and 
\be \label{eq:ellip}
\lda\le a(x,t)\le \Lda,  \quad \lda |\xi|^2 \le \sum_{i,j=1}^na_{ij}(x,t)\xi_i\xi_j\le \Lda |\xi|^2 \quad\forall\ \xi\in\R^n,
\ee
where $0<\lda\le \Lda<\infty$. We also assume that there exists a constant $\bar \lda>0$ such that
\be \label{eq:coeromega}
\int_{\Omega} [(A\nabla \phi)\cdot \nabla \phi +b_1 \phi^2]\,\ud x \ge \bar \lda \int_{\Omega} \phi^2 \quad \quad \forall \phi\in H^1_0(\Omega), \forall \ t\in(-1,0),
\ee
and
\begin{equation}\label{eq:positivelower}
b_2\ge \lambda \quad\mbox{in } \Omega \times(-1,0].
\end{equation}
Under the above conditions, the equation \eqref{eq:general} is uniformly parabolic when $x$ stays away from the boundary $\partial\Omega$. Therefore, to obtain global estimates for \eqref{eq:general}, we only need to establish estimates near $\partial\Omega$, that is in $(B_r(x_0)\cap\Omega)\times(-1,0]$, where $x_0\in\partial\Omega,r>0$ and $B_r(x_0)$ is the open ball in $\R^n$ centered at $x_0$ with radius $r$. By the standard flattening the boundary techniques for studying boundary estimates, we only need to consider the equation in the half ball case. We can choose $r$ small enough to ensure that the coercivity \eqref{eq:coeromega} (possibly with $\bar\lambda/2$ instead of $\bar\lambda$) still holds after flattening the boundary.

Now we would like to consider the half ball case, i.e., $\Omega$ is a half ball.  For $\bar x=(\bar x', 0)$,  denote $B_R^+(\bar x) =B_R(\bar x)\cap \{(x',x_n):x_n>0\}$,
\[
Q_R^+(\bar x, \bar t)= B_R^+(\bar x)  \times [\bar t-R^2, \bar t], \quad \mathcal{Q}_R^+(\bar x,\bar t)= B_R^+(\bar x) \times [\bar t-R^{p+1}, \bar t].
\] 
For brevity, we drop $(\bar x)$ and $(\bar x,\bar t)$ in the above notations if $\bar x=0$ or $(\bar x,\bar t)=(0,0)$.  

Consider the equation
\be \label{eq:linear-eq}
ax_n^{p-1} \pa_t u-\mathrm{div}(A\nabla u )+b_1u+b_2  x_n^{p-1}  u= x_n^{p-1} f \quad \mbox{in }Q_1^+
\ee
with partial Dirichlet condition
\be \label{eq:linear-eq-D}
u=0 \quad \mbox{on }\pa' B_1^+\times[-1,0],
\ee
where $\pa' B_R^+= B_R \cap \{x_n=0\}$.  The equation \eqref{eq:linear-eq} is uniformly parabolic away from $\{x_n=0\}$, where classical estimates would apply directly. Here, we would like to establish  regularity estimates for solutions of \eqref{eq:linear-eq} and \eqref{eq:linear-eq-D} up to the boundary $\{x_n=0\}$, that is, in $B_{1/2}^+\times[-1,0]$.

We also denote $\pa'' B_R^+=\pa B_R^+\setminus \pa' B_R^+$,  and $\partial_{pa} Q_R^+ $ as the standard parabolic boundary of $Q_R^+ $.

\subsection{Local boundedness}

We need some Sobolev type inequalities first. 

\begin{lem}\label{lem:Hardy-Sobolev} Let $s\in (0,2)$ and $R\ge 1$. Then
\begin{equation}\label{eq:hardysobolev}
\Big(\int_{B_R^+}\frac{|u(x)|^{\frac{2(n-s)}{n-2}}}{x_n^s}\Big)^{\frac{n-2}{n-s}}\le C(n)^{\frac{n-2}{n-s}}\int_{B_R^+}|\nabla u|^2\quad\forall\ u\in H_0^1(B_R^+),
\end{equation}
when $n\ge 3$, and for $s\le r<\infty$,
\begin{equation}\label{eq:hardysobolev-2d}
\Big(\int_{B_R^+}\frac{|u(x)|^{r}}{x_n^s}\Big)^{\frac{2}{r}}\le C(R, r,s)\int_{B_R^+}|\nabla u|^2\quad\forall\ u\in H_0^1(B_R^+),
\end{equation}
when $n=1,2$.
\end{lem} 

\begin{proof}  If $n\ge 3$, using the H\"older inequality,  Hardy inequality and Sobolev inequality,  we have 
\begin{align*}
\int_{B_R^+}\frac{|u(x)|^{\frac{2(n-s)}{n-2}}}{x_n^s}& = \int_{B_R^+}\frac{|u(x)|^{s}}{x_n^s} |u(x)|^{\frac{2n-sn}{n-2}} 
\\&  \le \Big( \int_{B_R^+}\frac{|u(x)|^{2}}{x_n^2} \Big)^{\frac{s}{2}}  \Big( \int_{B_R^+}|u|^{\frac{2n}{n-2}} \Big)^{\frac{2-s}{2}} \\&
\le C(n)\Big( \int_{B_R^+} |\nabla u|^2 \Big)^{\frac{s}{2}} \Big( \int_{B_R^+} |\nabla u|^2 \Big)^{\frac{n}{n-2}\frac{2-s}{2}} \\&= C(n)\Big( \int_{B_R^+} |\nabla u|^2 \Big)^{\frac{n-s}{n-2}}  .
\end{align*}
If $n=1,2$, we have 
\begin{align*}
\int_{B_R^+}\frac{|u(x)|^{r}}{x_n^s} &= \int_{B_R^+}\frac{|u(x)|^{s}}{x_n^s} |u|^{r-s}  \\&
\le \Big( \int_{B_R^+}\frac{|u(x)|^{2}}{x_n^2} \Big)^{\frac{s}{2}}  \Big( \int_{B_R^+}|u|^{\frac{2(r-s)}{2-s}} \Big)^{\frac{2-s}{2}}  \\&
\le C(R,r,s)  \Big( \int_{B_R^+} |\nabla u|^2 \Big)^{\frac{r}{2}} . 
\end{align*}
Therefore, we complete the proof. 
\end{proof}

Let 
\[
V_0^1(B_R^+\times (-T,0]) =L^\infty ((-T,0]; L^2(B_R^+,x_n^{p-1}\ud x)) \cap  L^2((-T,0];H_0^1(B_R^+)) 
\]
and 
 \begin{equation}\label{eq:V10norm}
  \|u\|_{V_0^1(B_R^+\times (-T,0])}^2= \sup_{-T<t<0} \int_{B_R^+ }u^2 x_n^{p-1}\,\ud x +  \|\nabla u\|_{L^2(B_R^+ \times(-T,0])} ^2 
  \end{equation}
    for $u\in V_0^1(B_R^+\times (-T,0])  $.

\begin{lem}\label{lem:weightedsobolev} Let $R\ge 1$ and $T>0$.
For every $u\in V_0^1(B_R^+\times (-T,0])$, we have
 \[
 \Big(\int_{B_R^+ \times (-T,0]} |u|^{2\chi}\ud x \ud t \Big)^{\frac{1}{\chi}} \le C \|u\|_{V_0^1(B_R^+\times (-T,0])}^2,
 \]
 where $\chi =\frac{(n+2-2s)}{n-s} >1$ with $s=\frac{2(p-1)}{n+p-3}$ and $C$ depends only on $n$ and $p$  if $n\ge 3$;  while $\chi=\frac{p+1}{p}$ and $C$  depends only on $R$ and $p$ if $n=1,2$.
\end{lem}
\begin{proof} We prove the case $n\ge 3$ first.
Note that  $\frac{s(n-2)}{2-s}=p-1$. By \eqref{eq:hardysobolev} and the H\"older inequality, we have
\begin{align*}
\int_{B_R^+  }|u|^{\frac{2(n+2-2s)}{n-s} } \,\ud x &=\int_{B_R^+  }|u|^{2} x_n^{-\frac{s(n-2)}{n-s}}|u|^{\frac{2(2-s)}{n-s} } x_n^{\frac{s(n-2)}{n-s}}\,\ud x  \\
&\le \Big(\int_{B_R^+}\frac{|u|^{\frac{2(n-s)}{n-2}}}{x_n^s}\,\ud x\Big)^{\frac{n-2}{n-s}} \Big( \int_{B_R^+ } u^2 x_n^{\frac{s(n-2)}{2-s}}\,\ud x\Big)^{\frac{2-s}{n-s}} \\&
\le C \Big(\int_{B_R^+  } |\nabla u|^2\,\ud x\Big)   \Big( \int_{B_R^+ } u^2 x_n^{p-1}\,\ud x\Big)^{\frac{2-s}{n-s}}.
\end{align*}
Integrating the above inequality in $t$, we have
\begin{align*}
&\Big(\int_{-T}^0\int_{B_R^+  }|u(x,t)|^{\frac{2(n+2-2s)}{n-s}} \,\ud x \ud t\Big)^{\frac{n-s}{n+2-2s}}\\
& \le C  \sup_{-T<t<0}\Big( \int_{B_R^+  }u^2 x_n^{p-1}\,\ud x\Big)^{\frac{2-s}{n+2-2s}} \Big(\int_{B_R^+ \times[-T,0]} |\nabla u|^2 \,\ud x \ud t  \Big)^{\frac{n-s}{n+2-2s}}\\&
\le C \Big( \|\nabla u\|_{L^2(B_R^+ \times(-T,0])} ^2 + \sup_{-T<t<0} \int_{B_R^+  }u^2 x_n^{p-1}\,\ud x \Big),
\end{align*}
where we have used the Young inequality in the last inequality.

If $n=1,2$, using \eqref{eq:hardysobolev-2d} and the H\"older inequality, we have 
\begin{align*}
\int_{B_R^+  }|u|^{2+\frac{2}{p}} \,\ud x&= \int_{B_R^+  }|u|^{2} x_n^{-\frac{p-1}{p}} |u|^{\frac{2}{p}} x_n^{\frac{p-1}{p}}  \,\ud x \\&
\le \Big( \int_{B_R^+  }  \frac{|u|^{\frac{2p}{p-1}}}{x_n}\,\ud x \Big)^{\frac{p-1}{p }} \Big( \int_{B_R^+ } u^2 x_n^{p-1}\,\ud x\Big)^{\frac{1}{p}}  \\&
\le C\Big( \int_{B_R^+  }  |\nabla u|^2\,\ud x \Big) \Big( \int_{B_R^+ } u^2 x_n^{p-1}\,\ud x\Big)^{\frac{1}{p}} .
\end{align*}
Integrating the above inequality in $t$ and using the Young inequality, we then complete the proof.
\end{proof}

To accommodate the partial boundary condition \eqref{eq:linear-eq-D}, we define the following space:
\[
H^1_{0,L}(B_R^+)=\{u\in H^1(B_R^+): u\equiv 0\ \mbox{on}\ \partial'B_R^+\}.
\]
By the  Poincar\'e inequality, $\|\nabla w\|_{L^2(B^+_R)}$ is  a norm for $w\in H^1_{0,L}(B_R^+)$.
Now we can prove local boundedness of weak solutions to
\be\label{eq:degiorgilinear}
a(x,t) x_n^{p-1} \pa_t u -\mathrm{div}(A(x,t) \nabla u)+b(x,t) u=f(x,t) \quad \mbox{in }Q_1^+
\ee
with the partial boundary condition \eqref{eq:linear-eq-D}, where the weak solutions are defined in a standard distribution sense in Definition \ref{defn:weaksolutionp}.

\begin{prop}\label{prop:local-bound} Suppose that $a$ and $A$ satisfy  \eqref{eq:ellip}, and $\pa_ta ,b, f\in L^q(Q_1^+)$ for some constant $q>\max\{\frac{\chi}{\chi-1},\frac{n+p+1}{2}\}$, where $\chi>1$ is the one in Lemma \ref{lem:weightedsobolev}. Let $u\in C ((-1,0]; L^2(B_1^+,x_n^{p-1}\ud x)) \cap  L^2((-1,0];H_{0,L}^1(B_1^+)) $ be a weak solution of \eqref{eq:degiorgilinear}.  Then we have, for any $\gamma>0$,
\[
\|u\|_{L^\infty(Q_{1/2}^+)} \le C\Big(\|u\|_{L^\gamma(Q_{1}^+)} +\|f\|_{L^q(Q_1^+)}\Big),
\]
where $C>0$ depends only on $n$, $\lda, \Lda$, $\gamma$, $p$, $\|\pa_t a\|_{L^q(Q_1^+)}$ and  $\|b\|_{L^q(Q_1^+)}$.
\end{prop}

\begin{proof} 
We provide a proof using the De Giorgi iteration with the help of the Sobolev inequality in Lemma \ref{lem:weightedsobolev}.

For $\theta\in (0,1)$, we first show that
\be \label{eq:l2-l_infty}
 \|u\|_{L^\infty(Q_{\theta}^+)} \le C\Big(\frac{1}{(1-\theta)^{\frac{1}{\delta}}}\|u\|_{L^2(Q_{1}^+)} +\|f\|_{L^q(Q_1^+)}\Big),
\ee
where $\delta=1-\frac{1}{q}-\frac{1}{\chi} $, and $C>0$ depends only on $n$, $\lda, \Lda$, $\gamma$, $p$,  $\|\pa_t a\|_{L^q(Q_1^+)}$ and  $\|b\|_{L^q(Q_1^+)}$. Clearly, we only need to consider $\theta\ge 1/2$.

Let $-1<t_1<t_2\le -\frac12$, and $0\le \eta_1(t)\le 1$ be a smooth function in $\R$ satisfying that $\eta_1(t)=0$ for $t\le t_1$, $\eta_1(t)=1$ for $t\ge t_2$ and $|\eta_1'(t)|\le \frac{C}{t_2-t_1}$, where $C>0$ is independent of $t_1$ and $t_2$. Let $\frac12\le r_2\le r_1 <1$, and $ 0\le \eta_2(x)\le 1$ be a smooth function satisfying that $\eta_2(x)=1$ if $|x|\le r_2$, $\eta_2(x)=0$ for $|x|\ge r_1$ and $|\nabla \eta |\le \frac{C}{r_1-r_2}$, where $C>0$ is independent of $r_1 $ and $r_2$. Let $\eta(x,t)=\eta_1(t)\eta_2(x)$.

For $k\ge 0$, let $v= (u-k)^+$. Multiplying $v\eta^2$ to the equation and integrating by parts (in this step, we may assume that $\eta^2v\in V_0^{1,1}(Q_1^+)$ (defined in \eqref{eq:V110}), since otherwise we can replace $u$ by its Steklov average, and take a limit in the end), we have
\begin{align*}
&\sup_{-1<t<0} \int_{B_1^+} (v\eta)^2 x_n^{p-1} \,\ud x  +\int_{Q_1^+} |\nabla (v\eta)|^2 \,\ud x \ud t\\& \le C \int_{Q_{1}^+} \Big \{v^2 ((|b|+|\pa_t a|)\eta^2+|\pa_t\eta|\eta +|\nabla \eta|^2)  +  f^{+} v\eta^2\Big\} \\
&\quad+C k^2\int_{\{v\eta\neq 0 \}}  |b| \eta^2.
\end{align*}

By the H\"older inequality and Lemma \ref{lem:weightedsobolev}, we have
\begin{align*}
\int_{Q_1^+} f^+ v \eta^2 & \le \Big (\int_{Q_1^+} (f^+)^q \Big)^{\frac{1}{q}}  \Big(\int_{Q_1^+} (v\eta)^{2\chi} \Big)^{\frac{1}{2\chi}} |\{v\eta \neq 0\}|^{1-\frac{1}{q}- \frac{1}{2\chi}}\\&
\le C(n,p)\Big (\int_{Q_1^+} (f^+)^q \Big)^{\frac{1}{q}}  \|v\eta\|_{V_0^1 (Q_1^+)} |\{v\eta \neq 0\}|^{1-\frac{1}{q}- \frac{1}{2\chi}} \\&
\le \va \|v\eta\|_{V_0^1 (Q_1^+)} ^2 + C \|f^+\|_{L^q(Q_1^+)}^2    |\{v\eta \neq 0\}|^{2-\frac{2}{q}- \frac{1}{\chi}}, 
\end{align*}
\begin{align*}
&\int_{Q_1^+}  (|\pa_t a| +|b|) v^2 \eta^2 \\
& \le \Big (\int_{Q_1^+} (|\pa_t a| +|b|)^q \Big)^{\frac{1}{q}}  \Big(\int_{Q_1^+} (v\eta)^{2\chi} \Big)^{\frac{1}{\chi}} |\{v\eta \neq 0\}|^{1-\frac{1}{q}- \frac{1}{\chi}}\\&
\le C(n,p)^2\Big (\int_{Q_1^+} (|\pa_t a| +|b|)^q \Big)^{\frac{1}{q}}  \|v\eta\|_{V_0^1 (Q_1^+)}^2 |\{v\eta \neq 0\}|^{1-\frac{1}{q}- \frac{1}{\chi}}
\end{align*}
and 
 \begin{align*}
\int_{\{v\eta\neq 0 \}}  |b| \eta^2 & \le \Big (\int_{\{v\eta\neq 0\} } |b|^q\Big)^{\frac{1}{q}}   |\{v\eta \neq 0\}|^{1-\frac{1}{q}}.
\end{align*}
 It follows that
 \[
 \|v \eta\|_{V_0^1 (Q_1^+)}^2 \le C \Big( \int_{Q_1^+} v^2 (\eta |\pa_t \eta| +|\nabla \eta|^2) \ud x\ud t+  (k^2 +F^2 ) |\{v\eta \neq 0\}|^{1-\frac{1}{q}} \Big)
 \]
provided that $|\{v\eta \neq 0\}|$ is small, where $F^2= \|f^+\|_{L^q(Q_1^+)}^2$. By the H\"older inequality and Lemma \ref{lem:weightedsobolev}, we have
 \begin{align*}
 \int_{Q_1^+} (v\eta)^2 &\le  \Big( \int_{Q_1^+} (v\eta)^{2\chi}  \Big)^{\frac{1}{\chi}} |\{v\eta\neq 0\}|^{1-\frac{1}{\chi}}\\&
 \le C  \Big( \int_{Q_1^+} v^2 (\eta |\pa_t \eta| +|\nabla \eta|^2) \ud x\ud t |\{v\eta\neq 0\}|^{1-\frac{1}{\chi}} \\& \qquad +  (k^2 +F^2 ) |\{v\eta \neq 0\}|^{2-\frac{1}{q}-\frac{1}{\chi}} \Big) .
 \end{align*}
 Let $A(k,R)= \{(x,t)\in Q_{R}^+: u\ge k\}$. Then
 \begin{align}
 \int_{A(k,r_2)} (u-k)^2 \le& \frac{C}{(r_1-r_2)^2} |A(k,r_1)|^\delta \int_{A(k,r_1)} (u-k)^2 \nonumber \\&+C(k^2+ F^2 ) |A(k,r_1)|^{1+\delta} \label{eq:degiorgi}
 \end{align}
 if $|A(k,r_1)|$ is small, where $\delta=1-\frac{1}{q}-\frac{1}{\chi}$. Note that
\[
|A(k,R)| \le \frac{1}{k} \int_{A(k,R)} u^+ \,\ud x \le  \frac{1}{k} \|u\|_{L^2(Q_1^+)} |A(k,R)|^{1/2}.
\]
Hence, $|A(k,R)| \le \frac{1}{k^2} \|u\|_{L^2(Q_1^+)}^2 $. Hence, \eqref{eq:degiorgi} holds when $k^2\ge k_0^2=C \|u\|_{L^2}^2$ for some constant $C$ depending only on $n, \lda, \Lda$, $\|\pa_t a\|_{L^q(Q_1)}$ and $\|b\|_{L^q(Q_1)}$.

For any $m>k\ge k_0$ and $0<r<1$, we have
\[
\int_{A(m,r)} (u-m)^2 \le \int_{A(k,r)} (u-k)^2
\]
and
\[
|A(m,r) | = |Q_r^+ \cap \{u-k>m-k\}| \le \frac{1}{(m-k)^2} \int_{A(k,r)} (u-k)^2.
\]
By \eqref{eq:degiorgi}, we have
 \begin{align*}
&\int_{A(m,r_2)} (u-m)^2 \\
&\le C  \left(\frac{1}{(r_1-r_2)^2} +\frac{m^2+F^2}{(m-k)^2} \right) \frac{1}{(m-k)^{2\delta}} \left(\int_{A(k,r_1)} (u-k)^2\right)^{1+\delta}
 \end{align*}
or
\be\label{eq:degiorgi-1}
\|(u-m)^+\|_{L^2(Q_{r_2}^+)} \le C  \left(\frac{1}{r_1-r_2} +\frac{m+F}{m-k} \right) \frac{1}{(m-k)^{\delta}} \|(u-k)^+\|_{L^2(Q_{r_1}^+)}^{1+\delta}.
\ee
Set $\phi(k,r)=\|(u-k)^+\|_{L^2(Q_{r}^+)}  $. Let $k>0$ to be fixed,
\[
k_\ell= k_0 + k(1-2^{-\ell}) \quad \mbox{and}\quad  r_\ell= \theta+ 2^{-\ell} (1-\theta) \quad \mbox{for } \quad \ell=0,1,\dots.
\]
By \eqref{eq:degiorgi-1}, we have
\begin{align*}
\phi(k_\ell, r_\ell) &\le C \left( \frac{2^\ell}{1-\theta} + \frac{2^{\ell}(k_0+F+k)}{k} \right) \frac{2^{\delta \ell}}{k^{\delta}} \phi(k_{\ell-1}, r_{\ell-1})^{1+\delta} \\&
\le   \frac{C}{1-\theta} \cdot \frac{(k_0+F+k)}{k}  \frac{2^{(1+\delta) \ell}}{k^{\delta}} \phi(k_{\ell-1}, r_{\ell-1})^{1+\delta}.
\end{align*}
By induction, one can show that
\be\label{eq:degiorgiphi}
\phi(k_\ell, r_\ell) \le \frac{\phi(k_0,r_0)}{2^{\ell \frac{1+\delta}{\delta}}}
\ee
provided that
\be\label{eq:choiceofk}
  \frac{C 2^{\frac{(1+\delta)^2}{\delta}}}{1-\theta} \frac{(k_0+F+k)}{k}  \left(\frac{\phi(k_0, r_0)}{k}\right)^{\delta} \le 1.
\ee
Indeed, given \eqref{eq:choiceofk}, for $\ell=1$, we have
\[
\phi(k_1, r_1)\le   \frac{C}{1-\theta} \cdot \frac{(k_0+F+k)}{k}  \frac{2^{1+\delta}}{k^{\delta}} \phi(k_{0}, r_{0})^{1+\delta}\le \frac{\phi(k_0,r_0)}{2^{\frac{1+\delta}{\delta}}}.
\]
Now, suppose that \eqref{eq:degiorgiphi} holds for $\ell$. Then for $\ell+1$, we have
\begin{align*}
\phi(k_{\ell+1}, r_{\ell+1}) & \le   \frac{C}{1-\theta}  \frac{(k_0+F+k)}{k}  \frac{2^{(1+\delta) (\ell+1)}}{k^{\delta}} \phi(k_{\ell}, r_{\ell})^{1+\delta}\\
&\le  \frac{C}{1-\theta}  \frac{(k_0+F+k)}{k}  \frac{2^{(1+\delta) (\ell+1)}}{k^{\delta}} \left[\frac{\phi(k_0,r_0)}{2^{\ell \frac{1+\delta}{\delta}}}\right]^{1+\delta}\\
&= \frac{C2^{\frac{(1+\delta)^2}{\delta}}}{1-\theta} \frac{(k_0+F+k)}{k}  \frac{\phi(k_0,r_0)^\delta}{k^{\delta}} \frac{\phi(k_0,r_0)}{2^{(\ell+1) \frac{1+\delta}{\delta}}}\\
&\le \frac{\phi(k_0,r_0)}{2^{(\ell+1) \frac{1+\delta}{\delta}}}.
\end{align*}
This proves \eqref{eq:degiorgiphi} assuming \eqref{eq:choiceofk}. By choosing
\[
k= C^{\frac{1}{\delta}} 2^{\frac{1}{\delta}+\frac{(1+\delta)^2}{\delta^2}} (k_0+F +(1-\theta)^{-\frac{1}{\delta}} \phi(k_0, r_0) ),
\]
we  have $\frac{(k_0+F+k)}{k}\le 2$ and $  \frac{C 2^{\frac{(1+\delta)^2}{\delta}}}{1-\theta}  \left(\frac{\phi(k_0, r_0)}{k}\right)^{\delta} \le \frac 12$, so that  \eqref{eq:choiceofk} holds.

Sending $\ell\to \infty$, we conclude that
\[
\phi(k_0+k,\theta)=0.
\]
It yields \eqref{eq:l2-l_infty}, since $k_0 =C\|u^+\|_{L^2(Q_1^+)}  $ and $\phi(k_0, r_0) \le \|u^+\|_{L^2(Q_1^+)} $.

For any $R\le 1$, define
\begin{align*}
&\tilde u(x,t)= u(Rx, R^{p+1}t),\quad \tilde a(x,t)= a(Rx, R^{p+1}t), \quad \tilde A(x,t)= A(Rx, R^{p+1}t),\\
& \tilde b(x,t)= b(Rx, R^{p+1}t), \quad\mbox{and}\quad \tilde f(x,t)= f(Rx, R^{p+1}t).
\end{align*}
Then
\[
\tilde a x_n^{p-1}  \pa_t \tilde u- \mathrm{div}(\tilde A \nabla \tilde u) +R^2 \tilde b \tilde u = R^2 \tilde f \quad \mbox{in }Q_1^+
\]
and $\tilde u=0$ on $\pa' B_1\times (-1,0]$. Moreover, $\tilde a$ and $\tilde A$ also satisfy \eqref{eq:ellip}, and $$\|R^2\tilde b\|_{L^q(Q_1^+)}= R^{2-\frac{n+p+1}{q}}\|b\|_{L^q(\mathcal{Q}_R^+)}\le \|b\|_{L^q(Q_1^+)}$$ (since $q>\frac{n+p+1}{2}$), and $\|\partial_t \tilde a\|_{L^q(Q_1^+)}= R^{p+1-\frac{n+p+1}{q}}\|\partial_t a\|_{L^q(\mathcal{Q}_R^+)}\le \|\partial_t a\|_{L^q(Q_1^+)}.$
By the estimates \eqref{eq:l2-l_infty}, we have
\[
\|\tilde u\|_{L^\infty (Q_{\theta}^+)} \le C\Big(\frac{1}{(1-\theta)^{\frac{1}{\delta}}}\|\tilde u\|_{L^2(Q_1^+)} +R^2 \|\tilde f\|_{L^q(Q_1^+)}\Big).
\]
Scaling back to $u$,  we obtain
\begin{align*}
\| u\|_{L^\infty (\mathcal{Q}_{\theta R}^+)} &\le C\Big(\frac{1}{(1-\theta)^{\frac{1}{\delta}}R^{\frac{n+p+1}{2}}}\| u\|_{L^2(\mathcal{Q}_R^+)} +R^{2-\frac{n+p+1}{q}} \|f\|_{L^q(\mathcal{Q}_R^+)}\Big)\\
&\le C\Big(\frac{1}{(1-\theta)^{\frac{1}{\delta}}R^{\frac{n+p+1}{2}}}\| u\|_{L^\infty(\mathcal{Q}_R^+)}^{\frac{2-\gamma}{2}}\| u\|_{L^\gamma(\mathcal{Q}_R^+)}^{\frac{\gamma}{2}} +\|f\|_{L^q(\mathcal{Q}_1^+)}\Big)\\
&\le \frac{1}{2}\| u\|_{L^\infty (\mathcal{Q}_{R}^+)} +\frac{C}{(1-\theta)^{\frac{2}{\delta\gamma}}R^{\frac{n+p+1}{\gamma}}}\| u\|_{L^\gamma(\mathcal{Q}_R^+)}+C\|f\|_{L^q(\mathcal{Q}_1^+)},
\end{align*}
where we have used $\theta^2 R^{p+1} \ge (\theta R)^{p+1}$ in the first inequality, and H\"older's inequality in the last inequality. By the De Giorgi iteration lemma, that is Lemma \ref{lem:lemma4.3inHL},   the proposition follows immediately.
\end{proof}

If we assume in addition that $u(\cdot,-1)\equiv 0$, then one can obtain the boundedness in $B_{1/2}^+\times[-1,0])$, i.e., up to the initial time. The definition of such weak solutions with initial condition is given in Definition \ref{defn:weaksolutionwithinitialtime}.

\begin{prop}\label{prop:global-bound}
Assume the assumptions in Proposition \ref{prop:local-bound}. Let $u\in C ([-1,0]; L^2(B_1^+,x_n^{p-1}\ud x)) \cap  L^2((-1,0];H_{0,L}^1(B_1^+)) $ be a weak solution of \eqref{eq:degiorgilinear} with the partial boundary condition \eqref{eq:linear-eq-D} and also the initial condition $u(\cdot,-1)\equiv 0$.  Then we have, for any $\gamma>0$,
\[
\|u\|_{L^\infty(B_{1/2}^+\times[-1,0]))} \le C\Big(\|u\|_{L^\gamma(Q_{1}^+)} +\|f\|_{L^q(Q_1^+)}\Big),
\]
where $C>0$ depends only on $n$, $\lda, \Lda$, $\gamma$, $p$, $\|\pa_t a\|_{L^q(Q_1^+)}$ and  $\|b\|_{L^q(Q_1^+)}$.
\end{prop}

The proof of Proposition \ref{prop:global-bound} is almost identical to that of Proposition \ref{prop:local-bound}. Actually, it will be slightly easier since one does not need the cut-off in $t$. We omit the details.

\begin{rem}\label{rem:Linfinitywithepsilon} 
Proposition \ref{prop:local-bound} and Proposition \ref{prop:global-bound} still hold if one replaces the weight $x_n^{p-1}$ by $(x_n+\va)^{p-1}$ in the equation \eqref{eq:degiorgilinear} for all $\va\in [0,1)$, with the constant $C$ in the estimates independent of $\va$. Indeed, the only change in the proof is using the scaling 
\[
\tilde u(x,t)= u(Rx, R^2 (R\vee\va)^{p-1} t)
\]
with adapted mollifications of the cylinders, where $R\vee\va=\max(R,\va)$.
\end{rem}

\subsection{Smooth coefficients}

Next, we would like to study the regularity for weak solutions of \eqref{eq:linear-eq} when its coefficients are regular. We suppose that $-\mathrm{div} (A \nabla ) +b_1$ is coercive, i.e., there exists a constant $\bar \lda>0$ such that
\be \label{eq:coer-app-smooth}
\int_{B_1^+} A\nabla \phi \nabla \phi +b_1 \phi^2 \ge \bar \lda \int_{B_1^+} \phi^2 \quad \quad \forall \phi\in H^1_0(B_1^+), a.e.\ t\in[-1,1].
\ee
We use the notation that
\[
\|u\|_{L^2(Q_1; x_n^{p-1})}=\left(\int_{Q_1^+}u(x,t)^2 x_n^{p-1}\,\ud x\ud t\right)^{\frac 12}.
\]
Let $W^{1,1}_2(Q_1^+)=\{g\in L^2(Q_1^+): \pa_tg\in L^2(Q_1^+), D_x g \in L^2(Q_1^+)\}$ be the standard Sobolev space.
\begin{prop} \label{prop:smooth-co}  
Suppose  $a, A, b_1, b_2, f\in C^3(\overline{Q_1^+})$, $a$ and $A$ satisfy  \eqref{eq:ellip}. Suppose also that $A$ and $b_1$ are independent of $t$ and satisfy \eqref{eq:coer-app-smooth}.  Let $g\in W^{1,1}_2(Q_1^+)$ satisfying $g\equiv 0$ on $\pa'B_1\times[-1,0]$. Then there exists a unique weak solution of \eqref{eq:linear-eq} with the boundary condition $u=g$ on $\pa_{pa}Q_1^+$. Moreover, $u\in C^{2} (\overline{Q_{1/2}^+})$, all the second order derivatives of $u$ are Lipschitz continuous in $\overline{Q_{1/2}^+}$, and there holds
\[
\sum_{i+j\le 3,i\ge 0,j\ge 0}\|D_x^j\pa_t^ju\|_{L^\infty(\overline{Q_{1/2}^+})} \le C\left(\|u\|_{L^2(Q_1^+; x_n^{p-1})}+ \|f\|_{C^3(Q_1^+)}\right),
\]
where $C>0$ depends only on $n,p,\lda, \Lda$, and the $C^3(\overline{Q_1^+})$ norms  of $a, A, b_1$ and $b_2$. 
\end{prop}

The definition of weak solution $u$ to \eqref{eq:linear-eq} with the boundary condition $u=g$ on $\pa_{pa}Q_1^+$ is given in Definition \ref{defn:weaksolutionglobalin}.

\begin{proof} 
The existence and uniqueness of the desired weak solution follow from Theorem \ref{thm:existenceofweaksolution}. In the below, we shall prove the regularity.

Let
$u_\va\in C ([-1,0]; L^2(B_1^+)) \cap  L^2((-1,0];H^1(B_1^+))$ be the unique weak solution of
\be\label{eq:degiorgilinearappsmooth}
\begin{split}
&a\cdot  (x_n+\va)^{p-1} \pa_t u_\va -\mathrm{div}(A\nabla u_\va)+b_1 u_\va+b_2\cdot (x_n+\va)^{p-1}u_\va\\
&\quad=(x_n+\va)^{p-1}f \quad \mbox{in }Q_1^+.
\end{split}
\ee
satisfying $u_\va=g$ on $\pa_{pa}Q_1^+$. Since \eqref{eq:degiorgilinearappsmooth} is uniformly parabolic, it follows from classical regularity theory for uniformly parabolic equations that $u_\va\in C^4(Q_{1-\delta}^+)$ for all $\delta>0$. We will derive some uniform estimates of $u_\va$ independent on $\va.$

Applying Proposition \ref{prop:local-bound} and Remark \ref{rem:Linfinitywithepsilon} with a small $\gamma$, and using the H\"older inequality, we obtain 
\be \label{eq:prop24-1}
\|u_\va\|_{L^\infty(Q_{15/16}^+)}\le C\Big(\|u_\va\|_{L^2(Q_1^+;x_n^{p-1})}+\|f\|_{L^\infty(Q_1^+)}\Big).
\ee
Using $u_\va \eta^2$ as a test function with $\eta\in C_c^{2}(Q_{15/16})$ being some cutoff function such that $\eta\equiv 1$ in $Q_{14/16}$, the standard energy estimates argument with the estimate \eqref{eq:prop24-1} yields 
\begin{align}
\|u_\va\|_{V_0^1(Q_{14/16}^+)}& \le C \Big (\|u_\va\|_{L^2(Q_{15/16})}+\|f\|_{L^\infty(Q_{15/16}^+)}\Big) \nonumber \\&\le C\Big(\|u_\va\|_{L^2(Q_1^+;x_n^{p-1})}+\|f\|_{L^\infty(Q_1^+)}\Big). 
 \label{eq:prop24-2}
\end{align}
Using $\tilde\eta^2 (x_n+\va)^{p+1}\pa_tu_\va $ as a test function with $\tilde\eta\in C_c^{2}(Q_{14/16})$ being a cutoff function such that $\tilde\eta\equiv 1$ in $Q_{14/16}$, we  obtain  
\begin{align*}
&\int_{B_{14/16}^+} (x_n+\va)^{2p} |\pa_t u_\va|^2 \tilde\eta^2 + \frac{\ud }{\ud t} \int_{B_{14/16}^+}\tilde\eta^2 (x_n+\va)^{p+1} (A\nabla u_\va)\nabla u_\va \\&
\le C \Big(\int_{B_{14/16}^+} |\nabla u_\va|^2 +\tilde\eta|\nabla u_\va||\partial_t u| (x_n+\va)^{p}+\tilde\eta |u||\pa_t u_\va|(x_n+\va)^{p+1} \\
&\quad\quad\quad+ (x_n+\va)^{2p}f^2\Big). 
\end{align*}
Using the H\"older inequality
\begin{align*}
&\int_{B_{14/16}^+} (x_n+\va)^{2p} |\pa_t u_\va|^2 \tilde\eta^2 + \frac{\ud }{\ud t} \int_{B_{14/16}^+} \tilde\eta^2 (x_n+\va)^{p+1}(A\nabla u_\va)\nabla u_\va \\
&\le C \Big(\int_{B_{14/16}^+} |\nabla u_\va|^2 +u_\va^2 +  (x_n+\va)^{2p}f^2\Big). 
\end{align*}
Integrating both sides in $t$ and making use of \eqref{eq:prop24-1} and \eqref{eq:prop24-2}, we have 
\be \label{eq:t-derive}
\Big(\int_{Q_{13/16}^+} x_n^{2p} |\pa_t u_\va|^2\,\ud x\ud t\Big)^{1/2} \le C \Big(\|u_\va\|_{L^2(Q_1;x_n^{p-1})}+\|f\|_{L^\infty(Q_1^+)}\Big). 
\ee
Let $v_\va=\pa_t u_\va$. We find 
\begin{align*}
&a (x_n+\va)^{p-1} \pa_t v_\va-\mathrm{div}(A \nabla v_\va)+b_1v_\va+(\pa_ta+b_2) (x_n+\va)^{p-1} v_\va\\
&=(x_n+\va)^{p-1} (\pa_t f-\pa_t b_2 u_\va) \quad \mbox{in }Q_{1}^+.
\end{align*}
and $v_\va=0$ on $\pa' B_1^+\times (-1,0]$. By Proposition \ref{prop:local-bound}, we have 
\[
\|v_\va\|_{L^\infty(Q_{12/16}^+)} \le C\Big (\|v_\va\|_{L^{\gamma}(Q_{13/16}^+)} +\|\partial_t f\|_{L^\infty(Q_{13/16}^+)} +\|u_\va\|_{L^\infty(Q_{13/16}^+)} \Big),
\]
where $\gamma>0$. Choose $\frac{2\gamma p}{2-\gamma}<1$.   Using the H\"older inequality,  \eqref{eq:t-derive} and \eqref{eq:prop24-1}, we have 
\[
\|\pa_t u_\va \|_{L^\infty(Q_{12/16}^+)} \le C\Big(\|u_\va\|_{L^2(Q_1;x_n^{p-1})}+\|f\|_{C^1(Q_1^+)}\Big). 
\]
By considering the equations of $\pa_t^2 u_\va$ and $\pa_t^3 u_\va$ and repeating the above arguments, we have
\begin{align*}
\|\pa_t^2 u_\va \|_{L^\infty(Q_{10/16}^+)}+\|\pa_t^3 u_\va \|_{L^\infty(Q_{10/16}^+)} \le C \Big(\|u_\va\|_{L^2(Q_1;x_n^{p-1})}+\|f\|_{C^3(Q_1^+)} \Big).
\end{align*}
Applying the $W^{2,p}$ estimates from linear uniformly elliptic equations with bounded right hand sides to the equations of $\pa_t u_\va$ and $\pa_t^2 u_\va$, we have 
\begin{align*}
&\sup_{-\frac{81}{256}<t<0}\left(|D_x \pa_t u_\va(\cdot,t)|_{C^{1/2}(B_{9/16}^+)}+|D_x \pa_t^2 u_\va(\cdot,t)|_{C^{1/2}(B_{9/16}^+)}\right)\\
&\le C \Big(\|u_\va\|_{L^2(Q_1;x_n^{p-1})}+\|f\|_{C^3(Q_1^+)} \Big).
\end{align*}
Applying the Schauder estimates from linear uniformly elliptic equations to the equation of $\pa_t u_\va$  on each time slice, we have 
\[
|D_x^2 \pa_t u_\va|_{L^\infty(Q_{9/16}^+)}\le C \Big(\|u\|_{L^2(Q_1;x_n^{p-1})}+\|f\|_{C^3(Q_1^+)} \Big).
\]
Differentiating the equation in $x'=(x_1,\cdots,x_n)$ and using the Schauder estimates, we have
\[
|D_x^2 D_{x'}u_\va|_{L^\infty(Q_{9/16}^+)}\le C \Big(\|u_\va\|_{L^2(Q_1;x_n^{p-1})}+\|f\|_{C^3(Q_1^+)} \Big).
\]
Differentiating the equation in $x_n$, using the equation and facts that
\[
|u(x,t)|+|\partial_t u(x,t)|\le C \Big(\|u_\va\|_{L^2(Q_1;x_n^{p-1})}+\|f\|_{C^3(Q_1^+)} \Big)\quad\forall\ (x,t)\in Q_{9/16}^+,
\]
we obtain
\[
|D_{x_n}^3u_\va|_{L^\infty(Q_{9/16}^+)}\le C \Big(\|u_\va\|_{L^2(Q_1;x_n^{p-1})}+\|f\|_{C^3(Q_1^+)} \Big).
\]
Hence,
\begin{equation}\label{eq:thirdorderestimateuep}
\sum_{i+j\le 3,i\ge 0,j\ge 0}\|D_x^j\pa_t^ju_\va\|_{L^\infty(\overline{Q_{1/2}^+})} \le C\left(\|u_\va\|_{L^2(Q_1^+; x_n^{p-1})}+ \|f\|_{C^3(Q_1^+)}\right),
\end{equation}
Since 
\[
\|u_\va\|_{V^1_0(Q_1^+)}\le C (\|f\|_{L^2}+\|g\|_{W^{1,1}_2(Q_1^+)}),
\] 
we can pass along a subsequence that $u_{\va_j}\to u$ in $C^2(\overline Q_{1/2}^+)$, and inherit the estimate \eqref{eq:thirdorderestimateuep}. This finishes the proof.
\end{proof}

If we in addition assume $u(\cdot,-1)\equiv 0$ and further required compatibility conditions on $\pa'B\times\{t=-1\}$ (such as $f=0$ near $\{x_n=0,t=-1\}$), then one can show the above estimate in $B_{1/2}^+\times[-1,0])$, i.e., up to the initial time,  in a similar way. 

\begin{prop} \label{prop:smooth-co-boundary}  
Assume the assumptions in Proposition \ref{prop:smooth-co}. Suppose in addition that $f\in C^6(\overline Q_1^+)$,  $f$ vanishes in $\{0<x_n<\delta\}\times[-1,-1+\delta]$ for some $\delta>0$, and $g(\cdot,-1)\equiv 0$. Then there exists a unique weak solution of \eqref{eq:linear-eq} with the boundary condition $u=g$ on $\pa_{pa}Q_1^+$. Moreover, $u\in C^{2} (\overline{B_{1/2}^+}\times[-1,0])$,  all the second order derivatives of $u$ are Lipschitz continuous in $\overline{B_{1/2}^+}\times[-1,0]$, and there holds
\[
\sum_{i+j\le 3,i\ge 0,j\ge 0}\|D_x^j\pa_t^ju\|_{L^\infty(\overline{B_{1/2}^+}\times[-1,0])} \le C\left(\|u\|_{L^2(Q_1^+; x_n^{p-1})}+ \|f\|_{C^6(Q_1^+)}\right),
\]
where $C>0$ depends only on $\delta, n,p,\lda, \Lda$, and the $C^3(\overline{Q_1^+})$ norms  of $a, A, b_1$ and $b_2$. 
\end{prop}  

The additional assumption on $f$ in Proposition \ref{prop:smooth-co-boundary} ensures the compatibility condition for the equation of $\pa_t^k u_\va$, $k=0,1,2,3$, where $u_\va$ is the solution to \eqref{eq:degiorgilinearappsmooth} with  $u_\va=g$ on $\pa_{pa}Q_1^+$. Then one can obtain the uniform $C^3(\overline{B_{1/2}^+}\times[-1,0])$ estimates for $u_\va$ independent on $\va$. We omit the details. 

For $\delta>0$ small, let $\Omega_\delta=\{x\in\Omega:d(x)<\delta\}$.

\begin{thm}\label{thm:global-smooth} 
Assume $A$ and $b_1$ are independent of $t$, belong to  $C^3(\overline\Omega)$ and satisfy \eqref{eq:coeromega},  while $a , b_2$  belong to $C^{3}(\overline\Omega \times[-1,0])$, \eqref{eq:ellip} holds, and $f\in C^6(\overline\Omega \times[-1,0])$  vanishes in $\Omega_\delta\times[-1,-1+\delta]$ for some $\delta>0$. Then there exists a unique classical solution of \eqref{eq:general} satisfying $u=0$ on $\partial_{pa}(\Omega \times(-1,0])$. Moreover,
\begin{align*}
&\|u\|_{C^{2+\gamma} (\overline \Omega \times [-1,0])}\le C\|f\|_{C^{6} (\overline \Omega \times [-1,0])}
\end{align*}
for every $\gamma\in (0,1)$, where $C>0$ depends only on $n,p,\bar \lda, \lda, \Lda,\gamma,\delta,\Omega$, and the $C^3(\overline \Omega \times [-1,0]))$ norms  of $a, A, b_1$ and $b_2$. 
\end{thm}

Here, $u$ is classical solution of \eqref{eq:general} means that $u, \pa_t u, D_x u, D^2_x u\in C(\overline\Omega\times[-1,0])$ and $u$ satisfies \eqref{eq:general}.

\begin{proof}
For $\va>0$, we consider the regularized equation
\begin{align} \label{eq:generalregularized}
a\cdot (\omega(x)+\va)^{p-1} \pa_t u_\va-\mathrm{div}(A \nabla u_\va)&+b_1u_\va+b_2 \cdot (\omega(x)+\va)^{p-1}u_\va  \nonumber\\
&= (\omega(x)+\va)^{p-1}f \quad \mbox{in }\Omega \times(-1,0].
\end{align}
Then there exists a unique classical solution of \eqref{eq:generalregularized} satisfying $u_\va=0$ on $\partial_{pa}(\Omega \times(-1,0])$. Using Propositions \ref{prop:smooth-co} and  \ref{prop:smooth-co-boundary}, we have that 
\begin{align*}
\|u_\va\|_{C^{2+\gamma} (\overline\Omega \times [-1,0])}&\le C(\|u_\va\|_{L^2(\overline\Omega \times [-1,0];x_n^{p-1})} +\|f\|_{C^{6} (\overline\Omega \times [-1,0])})\\
&\le C\|f\|_{C^{6} (\overline\Omega \times [-1,0])},
\end{align*}
where $C$ is independent of $\va$ and we used the energy estimates of $u_\va$ in the second inequality. The conclusion follows by taking $\va\to 0^+$.
\end{proof}

\subsection{A Campanato space}

We are going to establish Schauder type estimates for solutions of \eqref{eq:linear-eq} by the Campanato method.  We first define a Campanato type space for our purpose.

Denote
$$\ud \mu=|x_n|^{p-1} \ud x\ud t,$$ and for every measurable set $E\subset\R^{n+1}$, we use the notations of integral averages: 
\[
\dashint_E u(x,t) \ud x\ud t= \frac{1}{|E|}\int_E u(x,t) \ud x\ud t,\quad \dashint_E u(x,t) \ud \mu= \frac{1}{\mu(E)}\int_E u(x,t) \ud \mu,
\]
where $|E|$ is the Lebesgue measure of $E$, and $\mu(E)$ is the measure of $E$ with respect to $\mu$. For short, we denote $(u)_{E}$ and $(u)^\mu_{E}$ as the integral averages of $u$ over the set $E$ with respect to $\ud x\ud t$ and  $\ud\mu$, respectively. 

For $\bar x\in B_1^+$ and $\rho<1/2$,  denote 
\begin{align*}
\mathcal{D}_\rho(\bar x,\bar t)&=B_{\rho \bar x_n}(\bar x)\times [\bar t-\bar x_n^{p+1} \rho^2,\bar t],\\
\mathcal{\widetilde D}_\rho(\bar x,\bar t)&=B_{\rho \bar x_n}(\bar x)\times [\bar t-\bar x_n^{p+1} \rho^2,\bar t+\bar x_n^{p+1} \rho^2],\\
\mathcal{Q}_R(\bar x,\bar t)&= B_R(\bar x) \times [\bar t-R^{p+1}, \bar t],\\
\mathcal{\widetilde Q}_R(\bar x,\bar t)&= B_R(\bar x) \times [\bar t-R^{p+1}, \bar t+R^{p+1}].
\end{align*}
For a cylinder $Q=B_r^+\times I$ with an interval $I\subset[0,1]$ of length $r^2$ for some $0<r<1$, and for $\al\in (0,1)$ and $u\in L^2(Q)$, define   
\begin{align*}
[u]_{\mathscr{C}^{\al}(Q)}= \Big(& \sup_{(\bar x,\bar t)\in Q, 0<\rho<1 /2} \frac{1}{|\bar x_n\rho|^{2\al}}\dashint_{\mathcal{\widetilde D}_{\rho}(\bar x,\bar t)\cap Q} |u-(u)_{\mathcal{D}_{\rho}(\bar x,\bar t)\cap Q }|^2 \ud x\ud t\Big)^{1/2} \\&+\Big(\sup_{(\bar x,\bar t)\in\overline Q, \bar x_n=0, R>0 } \frac{1}{R^{2\al}}\dashint_{\mathcal{\widetilde Q}_R (\bar x,\bar t)\cap Q} |u-(u)^\mu_{\mathcal{\widetilde Q}_R (\bar x,\bar t)\cap Q}|^2 \ud \mu \Big)^{1/2}.
\end{align*}
 Denote $$\|u\|_{\mathscr{C}^{\al}(Q)}= \|u\|_{L^2(Q)}+[u]_{\mathscr{C}^{\al}(Q)}.$$  
The interior oscillation integral and the boundary oscillation integral in the definition of $[u]_{\mathscr{C}^{\al}(Q)}$ can be bridged as follows: $\forall \ (\bar x,\bar t)\in Q$,
\begin{equation}  \label{eq:connection}
\begin{split}
& \frac{1}{|\bar x_n|^{2\al}}\dashint_{\mathcal{\widetilde D}_{1/2} (\bar x,\bar t)\cap Q} |u-(u)_{\mathcal{\widetilde D}_{1/2} (\bar x,\bar t)\cap Q }|^2 \ud x\ud t  \\&\le  \frac{1}{|\bar x_n|^{2\al}}\dashint_{\mathcal{\widetilde D}_{1/2} (\bar x,\bar t)\cap Q} |u-(u)^\mu_{\mathcal{\widetilde D}_{1/2} (\bar x,\bar t)\cap Q } |^2 \ud x\ud t   \\& 
 \le \frac{C(n,p)}{|\bar x_n|^{2\al}}\dashint_{\mathcal{\widetilde D}_{1/2} (\bar x,\bar t)\cap Q} |u-(u)^\mu_{\mathcal{\widetilde D}_{1/2} (\bar x,\bar t)\cap Q }|^2 \ud \mu  \\& \le
 \frac{C(n,p)}{|\bar x_n|^{2\al}}\dashint_{\mathcal{\widetilde D}_{1/2} (\bar x,\bar t)\cap Q} |u-(u)^\mu_{\mathcal{\widetilde Q}_{2\bar x_n} (\bar x',0,\bar t)\cap Q} |^2 \ud \mu  \\&
 \le  \frac{C(n,p)}{|2 \bar x_n|^{2\al}}\dashint_{\mathcal{\widetilde Q}_{2\bar x_n} (\bar x',0,\bar t)\cap Q } |u-(u)^\mu_{\mathcal{\widetilde Q}_{2\bar x_n} (\bar x',0,\bar t)\cap Q} |^2 \ud \mu, 
 \end{split}
\end{equation}
where we used $\mathcal{\widetilde D}_{1/2} (\bar x,\bar t)\subset \mathcal{\widetilde Q}_{2\bar x_n} (\bar x',0,\bar t)$ in the last inequality. 

This Campanato space has the following two properties.

\begin{lem}\label{lem:embed}
$\mathscr{C}^{\al}(Q)$ is a Banach space. Moreover,
\[ 
C^{\al, \al/2}_{x,t}(Q)\subset \mathscr{C}^{\al}(Q)\subset C^{\al, \al/(p+1)}_{x,t}(Q).
\]
\end{lem}
\begin{proof}
It is clear that $C^{\al, \al/2}_{x,t}(Q)\subset \mathscr{C}^{\al}(Q)$. We will prove the second claim. 

Let $u\in \mathscr{C}^{\al}(Q)$.  For any $(\bar x,\bar t)\in Q, R =\rho \bar x_n<\bar x_n/5$,  we have
\begin{align*}
\frac{1}{R^{2\al}} \dashint_{\mathcal{\widetilde Q}_{R} (\bar x,\bar t)\cap Q} |u-(u)^\mu_{\mathcal{\widetilde Q}_{R}(\bar x,\bar t) \cap Q }|^2 \ud \mu 
&\le  [u]^2_{C^{\al,\frac{\al}{p+1}}(\mathcal{\widetilde Q}_{R} (\bar x,\bar t)\cap Q)} \\
&\le [u]^2_{C^{\al,\frac{\al}{p+1}}(\mathcal{\widetilde D}_{\rho} (\bar x,\bar t)\cap Q)}.
\end{align*}
Let $\tilde u(x,t)=u(\ell x+\bar x, \ell^{p+1}t+\bar t)$ with $\ell=|\bar x_n|$, and $A=\{(\frac{y-\bar x}{\ell}, \frac{s-\bar t}{\ell^{p+1}})\ |\ (y,s)\in Q\}$, and denote $\widetilde Q_\rho(y,s)=B_\rho(y)\times (s-\rho^2,s+\rho^2)$ (writing $\widetilde Q_\rho=\widetilde Q_\rho(0,0)$). Then by using the usual Campanato norm, we have
\be \label{eq:connection-1}
\begin{split}
&[u]^2_{C^{\al,\frac{\al}{p+1}}(\mathcal{\widetilde D}_{\rho} (\bar x,\bar t)\cap Q)}\\
&\le C \ell^{-2\alpha} [\tilde u]^2_{C^{\al,\frac{\al}{2}}(\widetilde Q_{\rho}\cap A)}\\
&\le C\ell^{-2\alpha} \sup_{(\bar y,\bar s)\in \widetilde Q_{\rho}\cap A, 0<\delta\le 2\rho} \frac{1}{\delta^{2\alpha}} \dashint_{\widetilde Q_{\delta}(\bar y,\bar s)\cap \widetilde Q_{\rho} \cap A} |\tilde u-(\tilde u)_{\widetilde Q_{\delta}(\bar y,\bar s)\cap \widetilde Q_{\rho}\cap A}|^2\\
&\le C\ell^{-2\alpha} \sup_{(\bar y,\bar s)\in \widetilde Q_{\rho}\cap A, 0<\delta\le 2\rho} \frac{1}{\delta^{2\alpha}} \dashint_{\widetilde Q_{\delta}(\bar y,\bar s) \cap A} |\tilde u-(\tilde u)_{\widetilde Q_{\delta}(\bar y,\bar s)\cap A}|^2.
\end{split}
\ee
Noticing that under the transform $(x,t)\to (\ell x+\bar x, \ell ^{p+1}t+\bar t)$, 
\[
\widetilde Q_{\delta}(\bar y,\bar s) \to B_{\delta \ell }(\ell \bar y+\bar x)\times (\ell^{p+1}\bar s+\bar t- \ell^{p+1}\delta^2, \ell^{p+1}\bar s+\bar t+ \ell^{p+1}\delta^2).
\]
Since $\bar y\in \widetilde Q_\rho$ and $\rho<1/5$, we have, by denoting that  $(\bar z,\bar \tau) := (\ell \bar y+\bar x, \ell^{p+1}\bar s+\bar t)$,
\[
\mathcal{\widetilde D}_{5\delta/6}(\bar z,\bar\tau)\subset B_{\delta \ell }(\ell \bar y+\bar x)\times (\ell^{p+1}\bar s+\bar t- \ell^{p+1}\delta^2, \ell^{p+1}\bar s+\bar t+ \ell^{p+1}\delta^2)\subset \mathcal{\widetilde D}_{5\delta/4}(\bar z,\bar\tau).
\]
Therefore, by changing variables on the right hand side of \eqref{eq:connection-1} back to $u$, we obtain 
\[
[u]^2_{C^{\al,\frac{\al}{2}}(\mathcal{\widetilde D}_{\rho} (\bar z,\bar \tau)\cap Q)} 
\le \sup_{(\bar z,\bar \tau)\in Q, 0<\delta<1 /2} \frac{1}{|\bar z_n\delta|^{2\al}}\dashint_{\mathcal{\widetilde D}_{\delta}(\bar z,\bar \tau)\cap Q} |u-(u)_{\mathcal{\widetilde D}_{\delta}(\bar z,\bar \tau)\cap Q }|^2 \ud x\ud t.
\]
Hence, for any $(\bar x,\bar t)\in Q, R =\rho \bar x_n<\bar x_n/5$,  we have
\[
\frac{1}{R^{2\al}} \dashint_{\mathcal{\widetilde Q}_{R} (\bar x,\bar t)\cap Q} |u-(u)^\mu_{\mathcal{\widetilde Q}_{R}(\bar x,\bar t) \cap Q }|^2 \ud \mu \le C \|u\|_{\mathscr{C}^{\al}(Q)}.
\]

For $\bar x_n/5\le R$, 
\[
\begin{split}
&\frac{1}{R^{2\al}} \dashint_{\mathcal{Q}_{R} (\bar x,\bar t)\cap Q} |u-(u)^\mu_{\mathcal{Q}_{R}(\bar x,\bar t) \cap Q }|^2 \ud \mu \\
&\le \frac{1}{R^{2\al}} \dashint_{\mathcal{Q}_{R} (\bar x,\bar t)\cap Q} |u-(u)^\mu_{\mathcal{Q}_{10R} (\bar x',0,\bar t)\cap Q}|^2 \ud \mu   \\&
\le \frac{C(n,p)}{|10R|^{2\al}} \dashint_{\mathcal{Q}_{10R} (\bar x',0,\bar t)\cap Q} |u-(u)^\mu_{\mathcal{Q}_{10R} (\bar x',0,\bar t)\cap Q}|^2 \ud \mu. 
\end{split}
\]
Therefore, if $u\in \mathscr{C}^{\al}(Q)$, then  
\begin{align*}
&\int_{Q} |u|^2\ud \mu+\sup_{(\bar x,\bar t)\in Q, 0<R<1} \frac{1}{R^{2\al}}  \dashint_{\mathcal{Q}_R(\bar x,\bar t) \cap Q} |u-(u)^\mu_{\mathcal{Q}_R(\bar x,\bar t) \cap Q}|^2 \,\ud \mu  \\
&\quad \le C(n,p) \|u\|_{\mathscr{C}^{\al}(Q)}^2. 
\end{align*}
Consequently, it follows from G\'orka \cite{G} that $$u\in C^{\al, \al/(p+1)}_{x,t}(Q).$$ Then one can show that  $\mathscr{C}^{\al}(Q)$ is a Banach space, whose proof is very similar to those for the standard Campanato spaces. 
\end{proof} 

\begin{lem}\label{lem:equivalentnorms}
Let $\alpha\in (0,1)$ and 
\begin{align*}
[u]_{\mathscr{\widetilde C}^\alpha(Q)}&:=\sup_{(x,t),(y,t)\in Q}\frac{|u(x,t)-u(y,t)|}{|x-y|^\alpha}+ \sup_{(x,t),(x,s)\in Q}\frac{|u(x,t)-u(x,s)|}{|t-s|^\frac{\alpha}{p+1}}\\
&\quad+\sup_{(x,t),(x,s)\in Q}x_n^{\frac{(p-1)\alpha}{2}}\frac{|u(x,t)-u(x,s)|}{|t-s|^\frac{\alpha}{2}}.
\end{align*}
There exists $C>0$ depending only on $n,p,\alpha$ such that
\[
\frac{1}{C}[u]_{\mathscr{C}^\alpha(Q)}\le [u]_{\mathscr{\widetilde C}^\alpha(Q)}\le C[u]_{\mathscr{C}^\alpha(Q)}.
\]
\end{lem}
\begin{proof}
The first inequality is clear. For the second inequality, using Lemma \ref{lem:embed},  we only need to prove that 
\[
\sup_{(x,t),(x,s)\in Q}x_n^{\frac{(p-1)\alpha}{2}}\frac{|u(x,t)-u(x,s)|}{|t-s|^\frac{\alpha}{2}}\le C[u]_{\mathscr{C}^\alpha(Q)}
\]

For every $(x,t),(x,s)\in Q$, if $|t-s|\ge (x_n/5)^{p+1}$, then
\[
x_n^{\frac{(p-1)\alpha}{2}}\frac{|u(x,t)-u(x,s)|}{|t-s|^\frac{\alpha}{2}}\le 5^{\frac{(p-1)\alpha}{2}} \frac{|u(x,t)-u(x,s)|}{|t-s|^\frac{\alpha}{p+1}}\le C[u]_{\mathscr{C}^\alpha(Q)}.
\]
If $|t-s|\le (x_n/5)^{p+1}$, then let $\ell=x_n$ and $\tilde u(y,\tau)=u(\ell y+x,\ell^{p+1}\tau+t)$, and then using a similar proof in that of Lemma \ref{lem:embed}, we have
\begin{align*}
x_n^{\frac{(p-1)\alpha}{2}}\frac{|u(x,t)-u(x,s)|}{|t-s|^\frac{\alpha}{2}}\le C\ell^{-\alpha}[\tilde u]_{C^{\alpha,\alpha/2}(\widetilde Q_{1/5}\cap A)}\le C[u]_{\mathscr{C}^\alpha(Q)},
\end{align*}
where $A$ is the image of $Q$ under the scaling $(y,\tau)\to ((y-x)/\ell,(\tau-t)/(\ell^{p+1}))$.
\end{proof}

\begin{rem}\label{rem:dividecheck}
If $D_xu\in\mathscr{C}^\alpha(Q)$, and $u(x,t)=0$ for $x_n=0$, then 
\[
\frac{u(x,t)}{x_n}=\int_0^1 \partial_{x_n} u(x', sx_n,t)\,\ud s.
\] 
Therefore, one can check by Lemma \ref{lem:equivalentnorms} that if $\alpha<\frac{2}{p-1}$ then $\frac{u(x,t)}{x_n}\in\mathscr{C}^\alpha(Q)$, and 
\[
\left[\frac{u(x,t)}{x_n}\right]_{\mathscr{C}^\alpha(Q)}\le C(n,p,\alpha,Q) [D_xu]_{\mathscr{C}^\alpha(Q)}.
\]
\end{rem}

\subsection{Schauder estimates} 
In this section, we will prove Schauder estimates for solutions of \eqref{eq:general}. We will assume that the coefficients of the equation \eqref{eq:linear-eq} are sufficiently regular, and by Theorem \ref{thm:global-smooth}, its solutions are also regular. Then, we only need to derive the a priori estimates, and pass to a limit in the end. The justification of this process is given in Theorem \ref{thm:global-schauder}.

We will use Campanato's method to prove Schauder estimates. The Campanato method has been used to prove Schauder estimates for uniformly parabolic equations in Lieberman \cite{L}, and uniformly parabolic systems in Schlag \cite{Sch} and Dong-Zhang \cite{DZ}. Therefore, for equation  \eqref{eq:linear-eq}, we only need to take care of the estimates near $\pa\Omega$. The Campanato norm $\mathscr{C}^{\al}$ defined in the previous section is to serve this purpose. We will follow the steps of the Campanato's method in Schlag \cite{Sch}, where the only difference is that we will need to obtain the estimates for $\pa_t u$ first (with a price that we will assume $A$ and $b_1$ are independent of $t$ and have sufficient regularity in $x$).

Starting from Proposition \ref{prop:interior} to Corollary \ref{cor:space-schauder}, we always  assume 
\begin{itemize}
\item[(i).] $A$ and $b_1$ are independent of $t$ and belong to  $C^3(\overline B_1^+)$,  while $a , b_2$ belong to $C^{3}(\overline Q_1^+)$ and $f\in C^6(\overline Q_1^+)$;
\item[(ii).] The ellipticity condition \eqref{eq:ellip} holds, and the operator $
-\mathrm{div} (A \nabla ) +b_1$ is coercive, which means there exists a constant $\bar \lda>0$ such that
\be \label{eq:coer}
\int_{B_1^+} A\nabla \phi \nabla \phi +b_1 \phi^2 \ge \bar \lda \int_{B_1^+} \phi^2 \quad \quad \forall \phi\in H^1_0(B_1^+);
\ee
\item[(iii).] Without loss of generality,
\[
b_2\ge \lambda \quad\mbox{in } Q_1^+,
\]
since otherwise one can consider the equation for $e^{\Lambda t}u$ instead.
\end{itemize}
and assume $u\in C^{2+\gamma}(\overline Q_1^+)$ ($\forall\gamma\in(0,1)$) is a classical solution of \eqref{eq:linear-eq}-\eqref{eq:linear-eq-D}.

First, we study the interior estimates of such $u$, for which we can use the theory of linear uniformly parabolic equations. 

\begin{prop} \label{prop:interior} Let $\bar x\in B_{1/2}^+$ and $\bar t\in [-1/4,0]$.  For any $0<\rho\le R\le 1/4$, we have 
\begin{align*}
&\frac{1}{|\bar x_n\rho|^{2\al}}\dashint_{\mathcal{D}_{\rho}(\bar x,\bar t)} |u_t-(u_t)_{\mathcal{D}_{\rho}(\bar x,\bar t)}|^2 \ud x\ud t \\& \le C \frac{1}{|\bar x_nR|^{2\al}}\dashint_{\mathcal{D}_{R} (\bar x,\bar t)} |u_t-(u_t)_{\mathcal{D}_{R} (\bar x,\bar t)}|^2 \ud x\ud t \\
&\quad+C\big[\|\pa_tu\|_{L^\infty (Q_1^+)}+\|u\|_{L^\infty (Q_1^+)} +\|f\|_{\mathscr{C}^{\al} (Q_1^+)}\big]^2
\end{align*}
and 
\begin{align*}
&\frac{1}{|\bar x_n\rho|^{2\al}}\dashint_{\mathcal{D}_{\rho}(\bar x,\bar t)} |\nabla u-(\nabla u)_{\mathcal{D}_{\rho}(\bar x,\bar t)}|^2 \ud x\ud t \\& \le C \frac{1}{|\bar x_nR|^{2\al}}\dashint_{\mathcal{D}_{R} (\bar x,\bar t)} |\nabla u-(\nabla u)_{\mathcal{D} _{R} (\bar x,\bar t)}|^2 \ud x\ud t \\
&\quad+C\big[\|\pa_tu\|_{L^\infty (Q_1^+)}+\|u\|_{L^\infty (Q_1^+)} +\|f\|_{\mathscr{C}^{\al} (Q_1^+)}\big]^2, 
\end{align*} 
where  $C>0$ is a  constant depending only on $n,p,\lda,\Lda,\alpha$, the $C^3(\overline B_1^+)$ norms of $A,b_1,$ and the $\mathscr{C}^{\al} (Q_1^+)$ norms of $a, b_2$. 
\end{prop}

\begin{proof} Set $\ell=\bar x_n$ and

\begin{align*}
\tilde u(x,t)&= u(\ell x+\bar x, \ell^{p+1}t+\bar t),\quad \tilde a(x,t)= a(\ell x+\bar x, \ell^{p+1}t+\bar t),\\ 
\tilde A(x)&= A(\ell x+\bar x),\quad \tilde b_1(x)= b_1(\ell x+\bar x),\\ 
\tilde b_2(x,t)&= b_2(\ell x+\bar x, \ell^{p+1}t+\bar t),\quad \tilde f(x,t)= f(\ell x+\bar x, \ell^{p+1}t+\bar t).
\end{align*}
Then all $\tilde a,\tilde b_2, \tilde f$ are in the standard parabolic H\"older space $C^{\alpha,\alpha/2}(Q_{3/4})$, 
\begin{align*}
[\tilde a]_{C^{\alpha,\alpha/2}(Q_{3/4})}&\le \ell^\alpha  [a]_{\mathscr{C}^{\al} (Q_1^+)},\  [\tilde b_2]_{C^{\alpha,\alpha/2}(Q_{3/4})}\le \ell^\alpha  [b_2]_{\mathscr{C}^{\al} (Q_1^+)}, \\
[\tilde a_{ij}]_{C^{\alpha}(B_{3/4})}&\le \ell^\alpha  [a_{ij}]_{C^{\al} (B_1^+)},\  [\tilde b_1]_{C^{\alpha}(B_{3/4})}\le \ell^\alpha  [b_1]_{C^{\al} (B_1^+)},\\
[\tilde f]_{C^{\alpha,\alpha/2}(Q_{3/4})}&\le \ell^\alpha  [f]_{\mathscr{C}^{\al} (Q_1^+)},
\end{align*}
and 
\begin{align*}
&\tilde a(x,t)\cdot  (x_n+1)^{p-1}  \pa_t \tilde u- \mathrm{div}(\tilde A \nabla \tilde u) +\ell^2 \tilde b_1 \tilde u + \ell^{p+1} \tilde b_2(x,t)\cdot (x_{n}+1)^{p-1} \tilde u \\
&\quad= \ell^{p+1}(x_n+1)^{p-1} \tilde f \quad \mbox{in }Q_1,
\end{align*}
which is a uniform parabolic equation in $Q_{3/4}$.  
The coercivity \eqref{eq:coer} implies  that  $\int_{B_1^+} A\nabla \phi \nabla \phi +b_1 \phi^2 \ge \tilde \lda \int_{B_1^+} |\nabla\phi|^2,  \ \forall \phi\in H^1_0(B_1^+)$, with another constant $\tilde \lda>0$, from which one can see that the operator $- \mathrm{div}(\tilde A \nabla \cdot) +\ell^2 \tilde b_1\cdot $ is also coercive in $Q_1$ with a uniform constant independent of $\ell$.  Applying Proposition \ref{prop:appendixCvariablecoefficientfort} and Proposition \ref{prop:appendixCvariablecoefficientforx} to $\tilde u$, and then scaling back to $u$, we have
\begin{align*}
&\frac{1}{|\ell \rho|^{2\al}}\dashint_{\mathcal{D}_{\rho}(\bar x,\bar t)} |u_t-(u_t)_{\mathcal{D}_{\rho}(\bar x,\bar t)}|^2 \ud x\ud t \\& \le C \frac{1}{|\ell R|^{2\al}}\dashint_{\mathcal{D}_{R} (\bar x,\bar t)} |u_t-(u_t)_{\mathcal{D}_{R} (\bar x,\bar t)}|^2 \ud x\ud t \\
&\quad+C\big[\|\pa_tu\|_{L^\infty (Q_{3/4}^+)}+\|u\|_{L^\infty (Q_{3/4}^+)} +\|f\|_{\mathscr{C}^{\al} (Q_{3/4}^+)}\big]^2
\end{align*}
and 
\begin{align*}
&\frac{1}{|\ell\rho|^{2\al}}\dashint_{\mathcal{D}_{\rho}(\bar x,\bar t)} |\nabla u-(\nabla u)_{\mathcal{D}_{\rho}(\bar x,\bar t)}|^2 \ud x\ud t \\& \le C \frac{1}{|\ell R|^{2\al}}\dashint_{\mathcal{D}_{R} (\bar x,\bar t)} |\nabla u-(\nabla u)_{\mathcal{D} _{R} (\bar x,\bar t)}|^2 \ud x\ud t \\
&\quad + C\big[\|\pa_tu\|_{L^\infty (Q_{3/4}^+)}+\|D_x u\|_{L^\infty (Q_{3/4}^+)}+\|u\|_{L^\infty (Q_{3/4}^+)} +\|f\|_{\mathscr{C}^{\al} (Q_{3/4}^+)}\big]^2.
\end{align*} 
Applying elliptic estimates on each time slice, we have
\[
\|D_x u\|_{L^\infty (Q_{3/4}^+)}\le C\big[\|\pa_tu\|_{L^\infty (Q_{1}^+)}+\|u\|_{L^\infty (Q_{1}^+)} +\|f\|_{L^{\infty} (Q_{1}^+)}\big]. 
\]
This finishes the proof.
\end{proof}

Next, we study the boundary estimates.    We start with equations having regular coefficients.

 \begin{lem} \label{lem:cc-2} Let $0<\rho\le R/4$.  
 
 If $a\equiv \bar a$, $b_2\equiv \bar b$ and $f\equiv \bar f$ for some constants $\bar a, \bar b$ and $\bar f$,  we have
 \be \label{2.6-1}
 \int_{\mathcal{Q}_\rho^+}x_n^{p-1}|\pa_t u|^2 \, \ud x\ud t\le C\left(\frac{\rho}{R}\right)^{n+2p+2} \int_{\mathcal{Q}_R^+} x_n^{p-1} |\pa_tu|^2\, \ud x\ud t,
 \ee
where $C>0$ is a constant depending only on $n,p,\lda,\Lda,\bar a, \bar b$, and the $C^{3} (\overline B_1^+)$ norms of $A,b_1$. 

If in addition that $a_{ij}$ and $b_1$ are also constants, then we have 
\begin{align}
 \int_{\mathcal{Q}_\rho^+}|\nabla_{x'} u|^2\, \ud x\ud t &\le C\left(\frac{\rho}{R}\right)^{n+p+3} \int_{\mathcal{Q}_R^+} |\nabla_{x'} u|^2\, \ud x\ud t, \label{2.6-2}\\
 \rho^2  \int_{\mathcal{Q}_\rho^+}  |\nabla \nabla_{x'}u|^2\, \ud x\ud t  &\le C \int_{\mathcal{Q}_{2\rho}^+}  |\nabla_{x'}u|^2 \, \ud x\ud t,\label{2.6-3}
 \end{align}
 and, for any vector $X$, 
 \be
 \begin{split}
 &\rho^{p+1}\int_{\mathcal{Q}_\rho^+} x_n^{p-1} |\pa_t u|^2\, \ud x\ud t  \\
 &\le C\Big(\frac{\rho}{R}\Big)^{n+3p+1}  \int_{\mathcal{Q}_{R}^+} ( |\nabla u-X|^2 + R^2u^2) \, \ud x\ud t+CR^{n+3p+1} \bar f^2 ,
 \end{split}
 \ee
 where $C>0$ is a constant depending only on $n,p,\lda,\Lda,\bar a, \bar b, b_1$.
 \end{lem}

 \begin{proof} Note that $v:=\pa_t u$ satisfies 
 \be \label{eq:const-co}
 \bar a x_n^{p-1} \pa_t v-\mathrm{div}(A\nabla v)+b_1 v+x_n^{p-1}\bar b v=0 \quad \mbox{in }Q_1^+
 \ee
 and $v=0$ on $\pa' B_1^+ \times (-1,0]$.   By Proposition \ref{prop:smooth-co}, we have 
 \[
 \|D_x v\|_{L^\infty(\mathcal{Q}_{\rho}^+ )} ^2 \le  \frac{C}{R^{n+2(p+1)}} \int_{\mathcal{Q}_R^+} x_n^{p-1} v^2\, \ud x\ud t
 \]
 for any $0<\rho\le R/2$. It follows that 
 \begin{align*}
  \int_{\mathcal{Q}_\rho^+}x_n^{p-1}|\pa_t u|^2\, \ud x\ud t& = \int_{\mathcal{Q}_\rho^+}x_n^{p-1}|\pa_t u(x',x_n,t)- \pa_t u(x',0, t)|^2 \, \ud x\ud t\\&
  \le C \rho^{p+1} \int_{\mathcal{Q}_\rho^+}|D_{x_n} \pa_t u(x,t)|^2 \, \ud x\ud t\\&
  \le C \rho^{n+2(p+1)}\sup_{\mathcal{Q}_\rho^+ }|D_{x_n} \pa_t u|^2 \\&
  \le C\left(\frac{\rho}{R}\right)^{n+2p+2} \int_{\mathcal{Q}_R^+} x_n^{p-1} |\pa_tu|^2\, \ud x\ud t.
 \end{align*}
This proves the first inequality. 

Now let us suppose in addition that $a_{ij}$ and $b_1$ are also constants.
 
 If we let $w=D_{x_i} u$ for any $1\le i\le n-1$, then $w$ also satisfies \eqref{eq:const-co} and $w=0$  on $\pa' B_1^+ \times (-1,0]$.  By Proposition \ref{prop:smooth-co}, we have 
 \[
 \|\nabla w\|_{L^\infty(\mathcal{Q}_{\rho}^+)} ^2 \le  \frac{C}{R^{n+p+3}} \int_{\mathcal{Q}_R^+} w^2\, \ud x\ud t
 \]
 for any $0<\rho\le R/2$. It follows that 
 \begin{align*}
  \int_{\mathcal{Q}_\rho^+}|D_{x_i} u|^2\, \ud x\ud t& = \int_{Q_\rho^+}|D_{x_i} u(x',x_n,t)- D_{x_i} u(x',0, t)|^2\, \ud x\ud t \\&
  \le C \rho^{2} \int_{\mathcal{Q}_\rho^+}|D_{x_n}D_{x_i} u(x,t)|^2\, \ud x\ud t \\&
  \le C \rho^{n+p+3}\sup_{\mathcal{Q}_\rho ^+}|D_{x_n}D_{x_i} u|^2 \\&
  \le C\left(\frac{\rho}{R}\right)^{n+p+3} \int_{\mathcal{Q}_R^+} |D_{x_i} u|^2 \, \ud x\ud t.
 \end{align*}
 This proves the second inequality.  While for \eqref{2.6-3}, it follows from the local estimates of $v$ using Proposition \ref{prop:local-bound}. 
 
 Finally, by Proposition \ref{prop:local-bound} and H\"older inequality, using \eqref{eq:const-co} with $v=\pa_t u$ we have, for $\rho<R/4$,
 \[
 \sup_{\mathcal{Q}_\rho^+ }|\pa_t u|^2 \le  C\frac{1}{R^{n+3p+1}} \int_{\mathcal{Q}_{R/2}^+} x_n^{2p} |\pa_tu|^2\, \ud x\ud t.
 \]
 Hence, 
 \[
 \rho^{p+1}\int_{\mathcal{Q}_\rho^+} x_n^{p-1} |\pa_t u|^2 \, \ud x\ud t  \le C\Big(\frac{\rho}{R}\Big)^{n+3p+1}  \int_{\mathcal{Q}_{R/2}^+} x_n^{2p} |\pa_tu|^2 \, \ud x\ud t. 
 \]
 
Let $a_{ij}$ and $b_1$ be constants  and $X$ be an arbitrary constant vector.  Then we can rewrite the equation $u$, which is \eqref{eq:linear-eq}, as 
 \[
\bar a x_n^{p-1} \pa_t u -\mathrm{div}(A(\nabla u-X)) +b_1 u+ \bar b x_n^{p-1} u=x_n^{p-1} \bar f. 
 \]
Using $ x_n^{p+1}\eta^2 \pa_t u $ as a test function, where $\eta$ is a smooth cutoff function vanishing in $Q_1^+\setminus\mathcal{Q}_{3/4}^+$, we have  
 \[
  \int_{\mathcal{Q}_{1/2}^+} x_n^{2p} |\pa_tu|^2 \, \ud x\ud t \le C \int_{\mathcal{Q}_{1}^+} ( |\nabla u-X|^2 + (|b_1|^2 +|b_2|^2)u^2) \, \ud x\ud t +C \bar f^2,
 \]
 where $C>0 $ is a  constant independent of $X$. By scaling, we have 
 \begin{align*}
  \int_{\mathcal{Q}_{R/2}^+} x_n^{2p} |\pa_tu|^2 \, \ud x\ud t 
  \le C \int_{\mathcal{Q}_{R}^+} ( |\nabla u-X|^2 + R^2u^2)\, \ud x\ud t + C R^{n+3p+1} \bar f^2.
 \end{align*}
   Therefore, we have 
 \begin{align*}
& \rho^{p+1}\int_{\mathcal{Q}_\rho^+} x_n^{p-1} |\pa_t u|^2\, \ud x\ud t  \\
& \le C\Big(\frac{\rho}{R}\Big)^{n+3p+1}  \int_{\mathcal{Q}_{R}^+} ( |\nabla u-X|^2 + R^2u^2) \, \ud x\ud t + CR^{n+3p+1} \bar f^2.
 \end{align*}
 This proves the last inequality.
 \end{proof}
 
We now use the freezing coefficients method to prove Schauder estimates.  
 
 \begin{prop} \label{prop:campanato}  $\forall\ 0<\rho<R$ (small), we have
 \begin{align*}
&\frac{1}{\rho^{n+2p+2\al} }\int_{\mathcal{Q}_\rho^+} x_n^{p-1} |\pa_t u|^2  \\
&\le  \frac{C}{R^{n+2p+2\al}}\int_{\mathcal{Q}_R^+} x_n^{p-1} |\pa_t u|^2+C\big[\|u\|_{L^\infty (Q_1^+)} +\|f\|_{\mathscr{C}^{\al} (Q_1^+)}\big]^2,
 \end{align*}
where $C>0$ is a  constant depending only on $n,p,\lda,\bar\lambda, \Lda,\alpha$, the $C^3(\overline B_1^+)$ norms of $A,b_1,$ and the $\mathscr{C}^{\al} (Q_1^+)$ norms of $a, b_2$. 
 \end{prop}

 \begin{proof} We denote $f_R^\mu= (f)^\mu_{\mathcal{Q}_R^+}$ for short. Let $u_1$ be the unique weak solution of
 \begin{align*}
 a_R^\mu x_n^{p-1} \pa_t u_1 -\mathrm{div}(A\nabla u_1)+b_1 u_1+x_n^{p-1}(b_2)^\mu_R u_1&=x_n^{p-1}  f_R^\mu \quad \mbox{in }\mathcal{Q}_R^+, \\ u_1&=u \quad \mbox{on }  \pa_{pa} \mathcal{Q}_R^+.
 \end{align*}
Theorem \ref{thm:existenceofweaksolution} guarantees the existence and uniqueness of $u_1$. Let $u_2=u-u_1$. Then $u_2$ is a weak solution  of
 \begin{align*}
 &a_R^\mu x_n^{p-1} \pa_t u_2  -\mathrm{div}(A\nabla u_2)+b_1 u_2+x_n^{p-1}(b_2)_R^\mu u_2\\&=x_n^{p-1}  (f-f_R^\mu) -x_n^{p-1}(b_2-(b_2)_R^\mu) u-x_n^{p-1}(a-a_R^\mu) \pa_t u \quad \mbox{in }\mathcal{Q}_R^+
 \end{align*} 
 with $
 u_2=0 \mbox{ on }  \pa_{pa} \mathcal{Q}_R^+$. By Lemma \ref{lem:cc-2},  we have for $0<\rho<R/4$ that
 \begin{align*}
 \int_{\mathcal{Q}_\rho^+} x_n^{p-1} |\pa_t u|^2& \le 2 \int_{\mathcal{Q}_\rho^+} x_n^{p-1} |\pa_t u_1|^2 +2 \int_{\mathcal{Q}_\rho^+} x_n^{p-1} |\pa_t u_2|^2 \\&
 \le C\frac{\rho^{n+2p+2} }{R^{n+2p+2} } \int_{\mathcal{Q}_R^+} x_n^{p-1} |\pa_tu_1|^2 + 2 \int_{\mathcal{Q}_R^+} x_n^{p-1} |\pa_t u_2|^2\\&
 \le C\frac{\rho^{n+2p+2} }{R^{n+2p+2} } \int_{\mathcal{Q}_R^+} x_n^{p-1} |\pa_tu|^2 + C\int_{\mathcal{Q}_R^+} x_n^{p-1} |\pa_t u_2|^2.
 \end{align*}
Since we assumed that $-\mathrm{div}(A \nabla )+b_1$ is coercive and $b_2>0$, by Theorem \ref{thm:energyestimateut}, we have
 \begin{align*}
 &\int_{\mathcal{Q}_R^+}  x_n^{p-1} |\pa_t u_2|^2 \\
 & \le  C \int_{\mathcal{Q}_R^+} x_n^{p-1}  \{(f-f_R^\mu)^2+(b_2- (b_2)_R^\mu)^2 u^2+(a-a_R^\mu)^2 |\pa_tu|^2)\} \\&
 \le CF^2 R^{n+2p+2\al } +CR^{2\al}   \int_{\mathcal{Q}_R^+} x_n^{p-1} |\pa_tu|^2,
 \end{align*}
where $$F= \|u\|_{L^\infty (Q_1^+)} +\|f\|_{\mathscr{C}^{\al} (Q_1^+)}.$$ Therefore, we have
\begin{align*}
 \int_{\mathcal{Q}_\rho^+} x_n^{p-1} |\pa_t u|^2 \le C\left\{ \left(\frac{\rho}{R}\right)^{n+2p+2} +R^{2\al}  \right\}   \int_{\mathcal{Q}_R^+} x_n^{p-1} |\pa_tu|^2  +CF^2 R^{n+2p+2\al }. 
\end{align*}
Also, the above inequality trivially holds for $\rho\in[R/4,R]$. By an iteration lemma, Lemma \ref{lem:lemma3.4inHL}, it follows that
 \[
\frac{1}{\rho^{n+2p+2\al} }\int_{\mathcal{Q}_\rho^+} x_n^{p-1} |\pa_t u|^2   \le  \frac{C}{R^{n+2p+2\al}}\int_{\mathcal{Q}_R^+} x_n^{p-1} |\pa_t u|^2+CF^2,
 \]
from which the proposition follows. 
 \end{proof}    

Since
 \[
\int_{\mathcal{Q}_\rho^+} x_n^{p-1} |\pa_t u-(\pa_t u)^\mu_{\mathcal{Q}_\rho^+}|^2   \le \int_{\mathcal{Q}_\rho^+} x_n^{p-1} |\pa_t u|^2,
 \] 
 it follows from Proposition \ref{prop:campanato} that
 \[
\frac{1}{\rho^{n+2p+2\al} }\int_{\mathcal{Q}_\rho^+} x_n^{p-1} |\pa_t u-(\pa_t u)^\mu_{\mathcal{Q}_\rho^+}|^2   \le  \frac{C}{R^{n+2p+2\al}}\int_{\mathcal{Q}_R^+} x_n^{p-1} |\pa_t u|^2+CF^2.
 \]

\begin{cor} \label{cor:time-schauder} We have 
\[
\|u\|_{C^1(\overline Q_{1/2}^+)}+\|\pa_t u\|_{ \mathscr{C}^{\al}(Q_{1/2}^+)} \le C \big[\|\pa_t u\|_{L^\infty (Q_1^+)}+\|u\|_{L^\infty (Q_1^+)} +\|f\|_{\mathscr{C}^{\al} (Q_1^+)}\big],
\]  
where $C>0$ is a  constant depending only on $n,p,\lda,\bar\lambda, \Lda,\alpha$, the $C^3(\overline B_1^+)$ norms of $A,b_1,$ and the $\mathscr{C}^{\al} (Q_1^+)$ norms of $a, b_2$.
\end{cor} 

\begin{proof}
By using Proposition \ref{prop:interior}, Proposition \ref{prop:campanato}, and \eqref{eq:connection}, we have 
\begin{equation}\label{eq:partialtinfinity}
\|\pa_t u\|_{ \mathscr{C}^{\al}(Q_{1/2}^+)} \le C \big[\|\pa_t u\|_{L^\infty (Q_1^+)}+\|u\|_{L^\infty (Q_1^+)} +\|f\|_{\mathscr{C}^{\al} (Q_1^+)}\big].
\end{equation}
Since $A$ is $C^2$, we can apply elliptic estimates on each time slice to obtain the estimate for the spatial derivative $D_x u$.
\end{proof}
 
After deriving the estimate for $\partial_t u$, we move to deriving the estimate for $D_x u$. 
 \begin{prop} \label{prop:space} 
For small $R$, we have for $0<\rho< R$,
 \begin{align*}
& \frac{1}{\rho^{n+p+1+2\al}}\int_{\mathcal{Q}_\rho^+} |\nabla_{x'} u|^2 + |D_{x_n} u-(D_{x_n} u)_{\mathcal{Q}_\rho^+}|^2 \\&
\le \frac{C}{R^{n+p+1+2\al}} \int_{\mathcal{Q}_R^+} |\nabla_{x'} u|^2+  |D_{x_n} u-(D_{x_n} u)_{\mathcal{Q}_R^+}|^2 \\
&\quad+C \big[\|\pa_t u\|_{L^\infty (Q_1^+)}+\|u\|_{L^\infty (Q_1^+)} +\|f\|_{\mathscr{C}^{\al} (Q_1^+)}\big]^2,  
\end{align*}
where $C>0$ is a  constant depending only on $n,p,\lda,\bar\lambda, \Lda,\alpha$, the $C^3(\overline B_1^+)$ norms of $A,b_1,$ and the $\mathscr{C}^{\al} (Q_1^+)$ norms of $a, b_2$.
 \end{prop} 
 
 \begin{proof} 
Let $u_1$ be the unique weak solution of
 \begin{align*}
 a_R^\mu x_n^{p-1} \pa_t u_1 -\mathrm{div}(A(0)\nabla u_1)+b_1(0) u_1+x_n^{p-1}(b_2)^\mu_R u_1&=x_n^{p-1}  f^\mu_R \quad \mbox{in }\mathcal{Q}_R^+, \\ u_1&=u \quad \mbox{on }  \pa_{pa} \mathcal{Q}_R^+.
 \end{align*}
Let $u_2=u-u_1$. Then $u_2$ is a weak solution of 
 \begin{align*}
 &a^\mu_Rx_n^{p-1} \pa_t u_2  -\mathrm{div}(A(0)\nabla u_2)+b_1(0) u_2+x_n^{p-1}(b_2)^\mu_R u_2\\
 &\quad=x_n^{p-1}  (f-f^\mu_R) -x_n^{p-1}(b_2-(b_2)^\mu_R) u-x_n^{p-1}(a-a^\mu_R) \pa_t u\\
 &\quad\quad+\mathrm{div}((A-A(0))\nabla u) -(b_1-b_1(0))u  \quad \mbox{in }\mathcal{Q}_R^+
 \end{align*} 
 with $
 u_2=0 \mbox{ on }  \pa_{pa} \mathcal{Q}_R^+$. Then 
 \begin{align*}
& \int_{\mathcal{Q}_\rho^+} |\nabla_{x'} u|^2 + |D_{x_n} u-(D_{x_n} u)_{\mathcal{Q}_\rho^+}|^2 \\&
\le 2\int_{\mathcal{Q}_\rho^+} |\nabla_{x'} u_1|^2 + |D_{x_n} u-(D_{x_n} u_1)_{\mathcal{Q}_\rho^+}|^2\\
&\quad+2\int_{\mathcal{Q}_\rho^+} |\nabla_{x'} u_2|^2 + |D_{x_n} u_2-(D_{x_n} u_2)_{\mathcal{Q}_\rho^+}|^2.
 \end{align*}
Since $b_1$ is $C^1$, then by the ellipticity of $A$ and rescaling, we will have for all small $R$ that
\[
\int_{B_R^+} A(0)\nabla \phi \nabla \phi +b_1(0) \phi^2 \ge \bar \lda\int_{B_R^+} \phi^2 \quad \quad \forall \phi\in H^1_0(B_R^+).
\]   
By the Poincar\'e inequality and using the equation of $u_1$, we have  
 \begin{align*}
 &\int_{\mathcal{Q}_\rho^+} (D_{x_n} u_1-(D_{x_n} u_1)_{\mathcal{Q}_\rho^+})^2 \\&\le C\Big(\rho^2 \int_{\mathcal{Q}_\rho^+}|D_x D_{x_n}u_1|^2 +\rho^{2(p+1)} \int_{\mathcal{Q}_\rho^+}  |\pa_t D_{x_n} u_1|^2  \Big)  \\&
 \le C \Big(\rho^2  \int_{\mathcal{Q}_\rho^+} \{ |D_x D_{x'}u_1|^2  + u_1^2+x_n^{2(p-1)}( |\pa_t u_1|^2+(f_R^\mu)^2 )\}\\
 &\quad\quad\quad+\rho^{2(p+1)} \int_{\mathcal{Q}_\rho^+}  |\pa_t D_{x_n} u_1|^2   \Big) \\& 
 \le C \Big(\rho^2  \int_{\mathcal{Q}_\rho^+} \{ |D_x D_{x'}u_1|^2 + \rho^{p-1}x_n^{p-1} |\pa_t u_1|^2\}+\rho^{n+3(p+1)} \sup_{\mathcal{Q}_\rho^+}  |\pa_t D_{x_n} u_1|^2 \\
 &\quad\quad\quad+ C F^2 \rho^{n+p+3}  \Big)\\&
\le C \Big(\rho^2  \int_{\mathcal{Q}_\rho^+} \{ |D_x D_{x'}u_1|^2 + \rho^{p-1}x_n^{p-1} |\pa_t u_1|^2\}+\rho^{p+1} \int_{\mathcal{Q}_{2\rho}^+} x_n^{p-1} |\pa_t u_1|^2\Big)  \\
&\quad\quad\quad+ C F^2 \rho^{n+p+3},
 \end{align*}
 where $F= \|u\|_{L^\infty (Q_1^+)} +\|f\|_{\mathscr{C}^{\al} (Q_1^+)},$ and we have used Proposition \ref{prop:smooth-co} in the last inequality. By Lemma  \ref{lem:cc-2}, we have for $0<\rho\le R/8$ that
 \begin{align}
& \int_{\mathcal{Q}_\rho^+} |\nabla_{x'} u_1|^2+  |D_{x_n} u_1-(D_{x_n} u_1)_{\mathcal{Q}_\rho^+}|^2 \nonumber \\&\le C F^2 R^{n+p+3}+ C (\frac{\rho}{R})^{n+p+3} \int_{\mathcal{Q}_R^+} |\nabla_{x'} u_1|^2+  |D_{x_n} u_1-(D_{x_n} u_1)_{\mathcal{Q}_R^+}|^2.
\label{eq:good-iterate}
 \end{align}
By the energy estimates \eqref{eq:energyestimateunoL2} of $u_2$ and Corollary \ref{cor:time-schauder}, we have 
 \begin{align*}
 \int_{\mathcal{Q}_R^+} |\nabla u_2|^2 &\le  C M^2 R^{n+p+1+2\al} + CR^{2\al} \int_{\mathcal{Q}_R^+} |\nabla u|^2 \\& 
  \le C M^2 R^{n+p+1+2\al} + CR^{n+p+1+2\al } \|\nabla u\|_{L^\infty(\mathcal{Q}_R^+)}^2 \\&
 \le C M^2 R^{n+p+1+2\al},
 \end{align*}
 where
 \[
M=\|\pa_t u\|_{L^\infty (Q_1^+)}+\|u\|_{L^\infty (Q_1^+)} +\|f\|_{\mathscr{C}^{\al} (Q_1^+)}.
 \]
 Therefore, we obtain 
  \begin{align*}
& \int_{\mathcal{Q}_\rho^+} |\nabla_{x'} u|^2 + |D_{x_n} u-(D_{x_n} u)_{\mathcal{Q}_\rho^+}|^2 \\&
\le C \left(\frac{\rho}{R}\right)^{n+p+3} \int_{\mathcal{Q}_R^+} (|\nabla_{x'} u|^2+  |D_{x_n} u-(D_{x_n} u)_{\mathcal{Q}_R^+}|^2) +C M^2 R^{n+p+1+2\al} .
 \end{align*} 
Note that the above inequality trivially holds for $\rho\in[R/8,R]$. By an iteration lemma, Lemma \ref{lem:lemma3.4inHL}, it follows that
  \begin{align*}
& \int_{\mathcal{Q}_\rho^+} |\nabla_{x'} u|^2 + |D_{x_n} u-(D_{x_n} u)_{\mathcal{Q}_\rho^+}|^2 \\&
\le C \left(\frac{\rho}{R}\right)^{n+p+1+2\al} \int_{\mathcal{Q}_R^+} (|\nabla_{x'} u|^2+  |D_{x_n} u-(D_{x_n} u)_{\mathcal{Q}_R^+}|^2) +C M^2 \rho^{n+p+1+2\al} ,
\end{align*}
from which we complete the proof.  
 \end{proof}

\begin{cor} \label{cor:space-schauder} 
We have 
\begin{align*}
&\|D_x u\|_{\mathscr{C}^{\al} (Q_{1/2}^+)} +\sup_{t\in (-\frac14,0]}\|D_x^2 u(\cdot,t)\|_{C^\beta (B_{1/2}^+)}  \\
&\le C\big[\|\pa_t u\|_{L^\infty (Q_1^+)} +\|u\|_{L^\infty (Q_1^+)} +\|f\|_{\mathscr{C}^{\al} (Q_1^+)}\big],
\end{align*}
where $\beta=\min(\al,p-1)$, and $C>0$ is a  constant depending only on $n,p,\lda,\bar\lambda, \Lda,\alpha$, the $C^3(\overline B_1^+)$ norms of $A,b_1,$ and the $\mathscr{C}^{\al} (Q_1^+)$ norms of $a, b_2$.
\end{cor} 
\begin{proof}
We first notice that
\begin{align*}
\dashint_{\mathcal{Q}_\rho^+} |\nabla u-(\nabla u)^\mu_{\mathcal{Q}_\rho^+}|^2\,\ud \mu&\le \dashint_{\mathcal{Q}_\rho^+} |\nabla u-(\nabla u)_{\mathcal{Q}_\rho^+}|^2\,\ud \mu\\
&\le \frac{C}{\rho^{n+2p}}\int_{\mathcal{Q}_\rho^+} |\nabla u-(\nabla u)_{\mathcal{Q}_\rho^+}|^2x_n^{p-1} \ud x\ud t\\
&\le C \dashint_{\mathcal{Q}_\rho^+} |\nabla u-(\nabla u)_{\mathcal{Q}_\rho^+}|^2\,\ud x\ud t.
\end{align*}
Hence, it follows from Proposition \ref{prop:interior}, Proposition \ref{prop:space} and \eqref{eq:connection} that
\[
\|D_x u\|_{\mathscr{C}^{\al} (Q_{1/2}^+)}\le C \|\nabla u\|_{L^2(Q_{3/4}^+)}+CM\le CM,
\]
where $M= \|\pa_t u\|_{L^\infty (Q_1^+)} +\|u\|_{L^\infty (Q_1^+)} +\|f\|_{\mathscr{C}^{\al} (Q_1^+)}.$ This proves the  estimate for $D_x u$. The  estimate for $D_x^2 u$ can be obtained by using Corollary \ref{cor:time-schauder} and elliptic estimates for second order derivatives on each time slice. 
\end{proof}

Using Lemma 3.1 on page 78 in \cite{LSU}, which is stated in Lemma \ref{lem:lemma3.1inLSU}, the above corollary implies that $D_x^2 u\in C^\gamma(Q_{1/2}^+)$ for some $\gamma>0$.

\begin{thm}\label{thm:global-schauder1} Suppose $u$ is a classical solution of \eqref{eq:linear-eq} satisfying \eqref{eq:linear-eq-D} and $u(x,-1)=0$ for $x\in B_1^+$.    Assume $A, b_1$ are independent of $t$ and belong to  $C^3(\overline B_1^+)$,  while $a , b_2$ belong to $C^{3}(\overline Q_1^+)$, \eqref{eq:ellip} holds, $f\in C^{6}(\overline Q_1^+)$ and $f(x,-1)=0$ for all $x\in\partial' B_1^+$.    Moreover, we assume the operator $
-\mathrm{div} (A \nabla ) +b_1$ satisfies \eqref{eq:coer}.   Then
\begin{align*}
&\|\pa_t u\|_{\mathscr{C}^{\al} (B_{1/2}^+ \times [-1,0])}+\|D_x u\|_{\mathscr{C}^{\al} (B_{1/2}^+ \times [-1,0])} +\sup_{t\in(-1,0)}\|D_x^2 u(\cdot, t)\|_{C^{\beta} (B_{1/2}^+)}   \\
&\le C (\|\pa_t u\|_{L^\infty(Q_1^+)} +\|u\|_{L^\infty(Q_1^+)} +\|f\|_{\mathscr{C}^{\al} (Q_1^+)}),
\end{align*}
where  $\beta=\min(\alpha,p-1)$ and $C>0$ is a constant depending only on $n,p,\lda,\bar\lambda, \Lda,\alpha$, the $C^3(\overline B_1^+)$ norms of $A,b_1,$ and the $\mathscr{C}^{\al} (Q_1^+)$ norms of $a, b_2$. 
\end{thm}
\begin{proof} By Corollary \ref{cor:time-schauder}  and Corollary \ref{cor:space-schauder}, one only needs to prove the Schauder estimates up to the bottom. For the estimate at $\{x_n=0,t=0\}$, the force term of the freezing coefficient equations would be chosen to be identically zero (including letting $\bar f=0$ in Lemma \ref{lem:cc-2}), due to the compatibility condition assumption. Remark \ref{rem:interiorbottom} will also be used. The proof is similar  and we leave the details to readers.
\end{proof}

For $\al\in (0,1)$, $Q:=\Omega \times[-1,0]$ and $u\in L^2(Q)$, define $\ud \nu=d(x)^{p-1} \ud x\ud t$ and   
\begin{align*}
[u]_{\mathscr{C}^{\al}(Q)}= \Big(& \sup_{(\bar x,\bar t)\in Q, 0<\rho<1 /2} \frac{1}{|d(\bar x)\rho|^{2\al}}\dashint_{\mathcal{G}_{\rho}(\bar x,\bar t)\cap Q} |u-(u)_{\mathcal{G}_{\rho}(\bar x,\bar t)\cap Q }|^2 \ud x\ud t\Big)^{1/2} \\&+\Big(\sup_{(\bar x,\bar t)\in\overline Q, \bar x\in\partial\Omega, R>0 } \frac{1}{R^{2\al}}\dashint_{\mathcal{\widetilde Q}_R (\bar x,\bar t)\cap Q} |u-(u)^\nu_{\mathcal{\widetilde Q}_R (\bar x,\bar t)\cap Q}|^2 \ud \nu \Big)^{1/2},
\end{align*}
where
\[
\mathcal{G}_\rho(\bar x,\bar t)=B_{\rho d(\bar x)}(\bar x)\times (\bar t-d(\bar x)^{p+1} \rho^2,\bar t+d(\bar x)^{p+1} \rho^2).
\]

 Denote $$\|u\|_{\mathscr{C}^{\al}(Q)}= \|u\|_{L^2(Q)}+[u]_{\mathscr{C}^{\al}(Q)}.$$ 

\begin{thm}\label{thm:global-schauder} 
Assume $A$ and $b_1$ are independent of $t$, belong to  $C^3(\overline \Omega)$ and satisfy \eqref{eq:coeromega},  while $a , b_2$ and $f$ belong to $\mathscr{C}^{\al}(\Omega \times[-1,0])$, \eqref{eq:ellip} holds, and $f(x,-1)=0$ for all $x\in\partial\Omega$.  Then there exists a unique classical solution $u$ of \eqref{eq:general} satisfying $u=0$ on $\partial_{pa}(\Omega \times(-1,0])$. Moreover,
\begin{align*}
&\|\pa_t u\|_{\mathscr{C}^{\al} (\Omega \times [-1,0])}+\|D_x u\|_{\mathscr{C}^{\al} (\Omega \times [-1,0])} +\sup_{t\in(-1,0)}\|D_x^2 u(\cdot, t)\|_{C^{\beta} (\Omega)} \\
&\le  C\|f\|_{\mathscr{C}^{\al} (\Omega \times [-1,0])},
\end{align*}
where  $\beta=\min(\alpha,p-1)$ and $C>0$ depends only on $n,p,\lda,\bar\lambda, \Lda,\alpha,\Omega$, the $C^3(\overline \Omega)$ norms of $A,b_1,$ and the $\mathscr{C}^{\al} (\Omega\times[-1,0])$ norms of $a, b_2$. 
\end{thm}

\begin{proof} 
Let $a^{(j)}, b_2^{(j)}$ be sufficiently smooth functions (say $C^6$) such that
$ a^{(j)} \to a$ and $b_2^{(j)}\to b_2$  uniformly in $\overline\Omega \times [-1,0]$ and
\begin{align*}
\|a^{(j)}\|_{\mathscr{C}^{\al} (\Omega \times [-1,0])}&\le C \|a\|_{\mathscr{C}^{\al} (\Omega \times [-1,0])}, \\\|b_2^{(j)}\|_{\mathscr{C}^{\al} (\Omega \times [-1,0])}&\le C \|b_2\|_{\mathscr{C}^{\al} (\Omega \times [-1,0])}.
\end{align*}
Let $f^{(j)}$ be sufficiently smooth functions (say $C^6$) such that $f^{(j)}$ vanishes near $\partial\Omega\times\{t=-1\}$, $f^{(j)} \to f$ uniformly in $\Omega \times [-1,0]$ and
\[
\|f^{(j)}\|_{\mathscr{C}^{\al} (\Omega \times [-1,0])}\le C \|f\|_{\mathscr{C}^{\al} (\Omega \times [-1,0])}.
\]
Such approximations can be achieved by standard mollifiers as in the usual H\"older space (recall that the space $\mathscr{C}^{\al}$ has a characterization of a weighted H\"older space in Lemma \ref{lem:equivalentnorms}).

By Theorem \ref{thm:global-smooth}, there exists a unique classical solution  $u^{(j)}$ of \eqref{eq:general} with $a,b_2,f$ replaced by $a^{(j)}, b_2^{(j)}, f^{(j)}$, respectively, satisfying $u^{(j)}=0$ on $\partial_{pa}(\Omega \times(-1,0])$. By Theorem \ref{thm:global-schauder1},
\begin{align*}
&\|\pa_t u^{(j)}\|_{\mathscr{C}^{\al} (\Omega \times [-1,0])}+\|D_x u^{(j)}\|_{\mathscr{C}^{\al} (\Omega \times [-1,0])} +\sup_{t\in(-1,0)}\|D_x^2 u^{(j)}\|_{C^{\beta} (\Omega)}\\
&\le C (\|\pa_t u^{(j)}\|_{L^\infty(\Omega \times [-1,0])}+\|u^{(j)}\|_{L^\infty(\Omega \times [-1,0])}+\|f\|_{\mathscr{C}^{\al} (\Omega \times [-1,0])}).
\end{align*}
Using the interpolation inequality:  for every $ \va>0$, 
\begin{align*}
&\|\pa_t u^{(j)}\|_{L^\infty(\Omega \times [-1,0])}+\|u^{(j)}\|_{L^\infty(\Omega \times [-1,0])}\\
&\le \va [\pa_t u^{(j)}]_{\mathscr{C}^{\al} (\Omega \times [-1,0])}+ \va [D_x u^{(j)}]_{\mathscr{C}^{\al} (\Omega \times [-1,0])}+C(n,\alpha, \va) \|u^{(j)}\|_{L^2(\Omega \times [-1,0])}
\end{align*}
with $C(n,\alpha,\va)>0$ depending only on $n,\alpha$ and $\va$, we obtain
\begin{align*}
&\|\pa_t u^{(j)}\|_{\mathscr{C}^{\al} (\Omega \times [-1,0])}+\|D_x u^{(j)}\|_{\mathscr{C}^{\al} (\Omega \times [-1,0])} +\sup_{t\in(-1,0)}\|D_x^2 u^{(j)}\|_{C^{\beta} (\Omega)}\\
&\le C (\|u^{(j)}\|_{L^2(\Omega \times [-1,0])}+\|f\|_{\mathscr{C}^{\al} (\Omega \times [-1,0])})\\
& \le C \|f\|_{\mathscr{C}^{\al} (\Omega \times [-1,0])}.
\end{align*}
where we used Theorem \ref{thm:energyestimateut} in the last inequality.
Thus, by using Ascoli-Arzela's theorem, there is a subsequence of $\{u^{(j)}\}$ converging uniformly to some function $u$. Moreover,  
\begin{align*}
&\|\pa_t u\|_{\mathscr{C}^{\al} (\Omega \times [-1,0])}+\|D_x u\|_{\mathscr{C}^{\al} (\Omega \times [-1,0])} +\sup_{t\in(-1,0)}\|D_x^2 u\|_{C^{\beta} (\Omega)} \\
&\le  C\|f\|_{\mathscr{C}^{\al} (\Omega \times [-1,0])},
\end{align*}
and $u$ is a solution of \eqref{eq:general}.

Now let us prove the uniqueness. Let $f\equiv 0$. Let $u, \pa_t u, D_x u, D^2_x u\in C(\overline\Omega\times[-1,0])$ and $u$ be a solution. For small $h>0$, we multiply $u^h(x,t)=\frac{u(x,t+h)-u(x,t)}{h}$ to \eqref{eq:general} and send $h\to 0$. By a similar argument to the derivation of \eqref{eq:L2control}, we obtain
\[
\int_{Q_1^+} u^2\,\ud x\ud s=0.
\]
So $u\equiv 0$ in $Q_1^+$. This proves the uniqueness.
\end{proof}

\section{Short time smooth solutions}  
\label{sec:short-time}

Let $\om \subset \R^n$, $n\ge 1$, be a smooth bounded  domain. Let 
\[
\mathcal{L}=- \Delta -b,
\] where $b<\lda_1$ is a constant and $\lda_1$ is the first Dirichlet  eigenvalue of $-\Delta $ on $\om$. Recall that $d(x)=dist(x,\pa \om)$.

Define
\begin{align} \label{eq:good-initial}
\mathcal{S}=\Big\{u\in C^\infty( \Omega) \cap C^{2+\alpha}_0(\overline \om):  \inf_{\om}\frac{u}{d}>0, &u^{1-p}\mathcal{L}u \in C_0^{2+\alpha}(\overline \om), \\
&(u^{1-p}\mathcal{L} )^2 u \in C_0^\alpha(\overline \om) \Big\}\nonumber,
\end{align}
where $$0<\al <\min \{1, p-1,\frac{2}{p-1}\},$$ and the subscript $0$ in $C_0^{k+\al}$  means that every function belonging to this set vanishes on $\partial\om$. The set $\mathcal{S}$ is not empty, because each positive eigenfunction of the problem
\[
\mathcal{L}u=\mu\omega(x)^q u\quad\mbox{in }\Omega,\quad u=0\mbox{ on }\partial\Omega,
\]
belongs to the set $\mathcal{S}$,  where $\omega(x)\ge 0$ is the first eigenfunction of $-\Delta$ in $\Omega$ with zero boundary condition and $q$ is a large positive constant.  Moreover, for every $u\in \mathcal{S}$ and every nonnegative $\tilde u\in C^\infty_c(\Omega)$ (smooth functions with compact support), then $a u + (1-a)\tilde u\in\mathcal{S}$ for all $a\in (0,1]$. Therefore, the set $\mathcal{S}$ is dense in the sets of nonnegative functions in $H^1_0(\Omega)$, $C_0(\overline\Omega)$ and $H^1_0(\Omega)\cap C_0(\overline\Omega)$, respectively.
Also, every positive solution of  $\mathcal{L} u=u^p$ in $\om$ and $u=0$ on $\pa \om$ belongs to the set $\mathcal{S}$.

\begin{prop} \label{prop:short-1} Let $p>1$ and $u_0\in \mathcal{S}$. Then there exists $T>0$ depending only on $n$, $\om$,  $b$, $\inf_{\Omega}u_0/d$, and $\|u_0^{1-p}\mathcal{L} u_0 \|_{C^{1}(\overline\om)}$ such that the Cauchy-Dirichlet problem of the linear parabolic equation
\[
p u_0^{p-1} \pa_tw =-\mathcal{L} w-t\cdot \frac{p-1}{p}u_0^{p-2} (u_0^{1-p}\mathcal{L}u_0)^2 \quad \mbox{in }\om \times [0,T],
\]
\[
w\big|_{t=0}= u_0, \quad w=0 \quad \mbox{on }\pa \om \times [0,T]
\]
has a unique classical nonnegative solution satisfying
\[
C^{-1}\le \frac{w(x,t)}{d(x)}\le C \quad \mbox{for }(x,t)\in \om \times [0,T]
\]
and 
\begin{align*}
&\sum_{i=0}^1\left[\|\partial_t^{i+1} w\|_{\mathscr{C}^{\al}(\om\times[0,T])} +\|D_x \partial_t^{i} w\|_{\mathscr{C}^{\al}(\om\times[0,T])} + \sup_{t\in[0,T]}\|D_x^2 \partial_t^{i} w(\cdot,t)\|_{C^\alpha(\Omega)} \right] \\
&\quad+ \sup_{t\in[0,T]}\|D_x^3 w(\cdot,t)\|_{C^\alpha(\Omega)}
\le C,
\end{align*}
where $C>0$ is a constant depending only on $n,p$, $\alpha$, $\om$, $b$, $\inf_{\Omega}\frac{u_0}{d}$, $\|u_0\|_{C^{2+\alpha}(\overline\Omega)}$, $  \|u_0^{1-p}\mathcal{L} u_0 \|_{C^{2+\alpha}(\overline\om)} $   and $  \| (u_0^{1-p}\mathcal{L} )^2 u_0 \|_{C^\alpha(\om)} $.  

\end{prop}

\begin{proof}
For $u_0\in\mathcal{S}$, one can check that $u_0^{1-p}\mathcal{L} u_0 $ and $(u_0^{1-p}\mathcal{L}u_0 )^2/u_0 $ are in $C_0^\alpha(\overline\Omega)$, and thus, the existence and uniqueness of classical solutions follows from Theorem \ref{thm:global-schauder}.

Let $\psi$ be the normalized nonnegative first eigenfunction of $-\Delta$ in the $L^2$ norm, i.e., $\int_{\om}\psi^2=1$,
\[
-\Delta \psi=\lda_1 \psi \quad \mbox{in }\om, \quad \mbox{and}\quad \psi=0 \quad \mbox{on }\pa \om.
\]
Thus $0<\inf_{\om} \frac{\psi}{d} \le \sup _{\om} \frac{\psi}{d} <\infty$. Note that
\[
pu_0^{p-1} \pa_t \psi=0\ge -(\lda_1-b) \psi =\Delta \psi +b\psi=-\mathcal{L}\psi.
\]
Since $p u_0^{p-1} \pa_tw \le-\mathcal{L} w$, by the comparison principle, $w \le C\psi$ for some $C>0$ depending only on $n$, $\om$,  $T$,  $b$ and $\sup_\Omega u_0/d$.

Let $\tilde\psi$ be the normalized nonnegative weighted first eigenfunction of $\mathcal{L}$ in the $L^2$ norm, i.e., $\int_{\om}\tilde\psi^2=1$, satisfying
\[
\mathcal{L} \tilde\psi=\tilde \lda_1 d^{p-1}\tilde\psi \quad \mbox{in }\om, \quad \mbox{and}\quad \tilde\psi=0 \quad \mbox{on }\pa \om.
\]
Thus $0<\inf_{\om} \frac{\tilde\psi}{d} \le \sup _{\om} \frac{\tilde\psi}{d} <\infty$, and $\inf_{\om} \frac{\tilde\psi}{d}$ depends only on $n,b,\Omega$ by the Hopf lemma. Therefore, there exists $M>0$ depending only on $\inf_\Omega u_0/d$ and $\|u_0^{1-p}\mathcal{L} u_0 \|_{C^{1}(\overline\om)} $ such that $|u_0^{1-p}\mathcal{L} u_0|\le Mu_0$ and $|u_0^{1-p}\mathcal{L} u_0|\le M\tilde\psi$ in $\overline\Omega$, and thus, 
\begin{align*}
p u_0^{p-1} \pa_tw \ge-\mathcal{L} w-tM^2u_0^{p-1}\tilde\psi\quad \mbox{in }\om \times [0,T].
\end{align*}
Note that for a positive smooth function $\eta$ of $t$,
\begin{align*}
pu_0^{p-1} \pa_t [\eta(t)\tilde\psi(x)]+\mathcal{L} [\eta(t)\tilde\psi(x)] &= pu_0^{p-1} \tilde\psi\eta'+\eta\tilde\lambda_1d^{p-1}\tilde\psi\\
&\le u_0^{p-1} \tilde\psi (p\eta'+m\lambda_1\eta),
\end{align*}
where $m>0$ depends only on  $\inf_\Omega u_0/d$. Let
\[
\eta(t)=\frac{pM^2}{m^2\tilde\lambda_1^2}-\frac{tM^2}{m\tilde\lambda_1}-\mu e^{-\frac{m\tilde\lambda_1 t}{p}},
\]
where $\mu>0$ is chosen such that
\[
\inf_\Omega\frac{u_0}{d}\ge 2 \left(\frac{pM^2}{m^2\tilde\lambda_1^2}-\mu\right)\sup_\Omega \frac{\tilde \psi}{d}.
\]
Therefore, there exists $T>0$ (whose dependence is clear) such that $\eta>0$ in $[0,T]$, $\eta(0)\tilde\psi\le u_0$ in $\Omega$ and
\[
pu_0^{p-1} \pa_t [\eta(t)\tilde\psi(x)]+\mathcal{L} [\eta(t)\tilde\psi(x)] \le -tM^2u_0^{p-1}.
\]
By the comparison principle, $w \ge \tilde \psi/C$ in $\overline\Omega\times[0,T]$ for some $C>0$ depending only on $n$, $\om$,  $b$, $\inf_\Omega u_0/d$ and $\|u_0^{1-p}\mathcal{L} u_0 \|_{C^{1}(\overline\om)}$.

Applying the comparison principle again to the equation of $w_t$, we have
\[
|\pa_t w|\le C \psi
\]
for some $C>0$ depending only on $n$, $\om$,  $T$,  $b$,   and $  \| u_0^{1-p}\mathcal{L} u_0 \|_{C^1(\overline\om)}$.

If we denote $v=w_t-u_0^{1-p}\mathcal{L} u_0$, then
\[
p u_0^{p-1} \pa_tv =-\mathcal{L} v- u_0^{p-1}[(u_0^{1-p}\mathcal{L} )^2 u_0]-\frac{p-1}{p}u_0^{p-2} (u_0^{1-p}\mathcal{L}u_0)^2 \quad \mbox{in }\om \times [0,T],
\]
\[
v=0 \quad  \mbox{on }\pa_{pa} (\om \times [0,T]).
\]
Applying Theorem \ref{thm:global-schauder} to the equations of $w$ and $v$, we obtain
\begin{align*}
&\sum_{i=0}^1\left[\|\partial_t^{i+1} w\|_{\mathscr{C}^{\al}(\om\times[0,T])} +\|D_x \partial_t^{i} w\|_{\mathscr{C}^{\al}(\om\times[0,T])} + \sup_{t\in[0,T]}\|D_x^2 \partial_t^{i} w(\cdot,t)\|_{C^\alpha(\Omega)} \right]\\
&\quad\le C.
\end{align*}
Applying elliptic estimates to the equation of $w$ on each time slides (and differentiating it once in the $x$-variables), we obtain
\[
 \sup_{t\in(\delta,T)}\|D_x^3 w(\cdot,t)\|_{C^\alpha(\Omega)}\le C
\]
for all $\delta\in(0,T)$ (note that $C$ does not depends on $\delta$). Since $u_0^{1-p}\mathcal{L}u_0 \in C_0^{2+\alpha}(\overline \om)$, we have $\|\nabla^3 u_0\|_{C^\alpha(\Omega)}\le C$, and thus,
\[
 \sup_{t\in[0,T]}\|D_x^3 w(\cdot,t)\|_{C^\alpha(\Omega)}\le C.
\]
This finishes the proof.
\end{proof}
Using Lemma 3.1 on page 78 in \cite{LSU} (see Lemma \ref{lem:lemma3.1inLSU}), the estimates in the above proposition imply that all $D_x^2 w,D_x^2 \partial_t w$ and $D_x^3 w$ are H\"older continuous in the time variable as well.

Denote
\begin{align*}
C^{3, 2}(\overline \om\times [0,T])=\{&u\in C(\overline \om\times [0,T]): D_x u, D_x^2 u, D_x^3 u,\\ &\quad\quad\quad\quad\partial_t u, D_x \partial_t u, D_x^2 \partial_t u, \partial_t^2 u \in C(\overline \om\times [0,T])\}.
\end{align*}
\begin{thm} \label{thm:short-time} Let $p>1$ and $u_0\in \mathcal{S}$.   Then there exist  $T>0$ and a unique nonnegative $u\in C^{3, 2}(\overline \om\times [0,T])$ satisfying that
\[
\pa_t u^p= -\mathcal{L} u \quad \mbox{in }\om\times [0,T],
\]
\[
u(0)=u_0, \quad u=0 \quad \mbox{on }\pa \om \times [0,T].
\]
Moreover, $u(\cdot,t)\in\mathcal{S}$ for all $t\in [0,T]$. 
\end{thm}

\begin{proof}  We shall use the implicit function theorem. Let $w$ be the solutions obtained in Proposition \ref{prop:short-1}. For any small $\va_0>0$, we have $\|w(\cdot, t)-u_0\|_{C^{2+\alpha}(\overline\om)} \le \va_0$ provided $t\le T_{\va_0}$, where $0<T_{\va_0} \le 1$ is a constant depending on $\va_0$. Note that
\[
1/C \le \frac{w}{d} \le C.
\]
Let
\begin{align*}
\mathcal{Y}= \Big\{ d^{p-1} f: f,\partial_t f\in \mathscr{C}^{\al}(\overline\om\times [0, T_{\va_0}]), \ &f(\cdot,0)\equiv 0 \mbox{ on }\overline\Omega, \\
&\partial_tf(x,0)= 0\mbox{ for }x\in\partial\Omega \Big\}
\end{align*}
with
\[
\|\phi\|_{\mathcal{Y}}= \|d^{1-p} \phi\|_{\mathscr{C}^{\al}(\overline\om\times [0, T_{\va_0}])}+\|d^{1-p} \partial_t\phi\|_{\mathscr{C}^{\al}(\overline\om\times [0, T_{\va_0}])},
\]
and let 
\begin{align*}
\mathcal{X}= \Big\{&\varphi\in C^{3, 2}(\overline\om\times [0, T_{\va_0}]):\ \pa_t \varphi, D_x \varphi, \pa_{tt} \varphi, D_x \partial_t\varphi \in  \mathscr{C}^{\al}(\overline\om\times [0, T_{\va_0}]),\\
&\quad\quad\quad  D_x^2 \varphi(\cdot, t),  D_x^2 \partial_t\varphi(\cdot, t), D_x^3\varphi(\cdot, t) \in C^\al(\overline\om) ~\forall ~0\le t\le T_{\va_0}, \\ 
& \quad\quad\quad \mathcal{L} \varphi \in \mathcal{Y}, \mbox{ and } \partial_t\varphi=\varphi= 0 \mbox{ on }\pa_{pa} (\overline\om \times [0, T_{\va_0}]) \Big\}
\end{align*}
with  
\begin{align*}
&\|\varphi \|_{\mathcal{X}}= \\
&\sum_{i=0}^1\Big[\|\partial_t^{i+1} \varphi\|_{\mathscr{C}^{\al}(\om\times[0,T])} +\|D_x \partial_t^{i} \varphi\|_{\mathscr{C}^{\al}(\om\times[0,T])} + \sup_{t\in[0,T]}\|D_x^2 \partial_t^{i} \varphi(\cdot,t)\|_{C^\alpha(\Omega)} \Big] \\
&\quad+ \sup_{t\in[0,T]}\|D_x^3 \varphi(\cdot,t)\|_{C^\alpha(\Omega)}+ \|\mathcal{L} \varphi\|_{\mathcal{Y}}+\|\varphi \|_{L^{\infty}(\overline\om\times [0, T_{\va_0}]) } ,
\end{align*}
where $\pa_{pa} (\overline\om \times [0, T_{\va_0}])  $  stands for the parabolic boundary of $\overline\om \times [0, T_{\va_0}] $. 
It is easy to check that both $\mathcal{X}$ and $\mathcal{Y}$ are Banach spaces, and $w\in\mathcal{X}$. 

Define $F(v):= p |v|^{p-1} \frac{\pa v}{\pa t} +\mathcal{L} v$ for $v \in \mathcal{X}$, and for all $\varphi\in \mathcal{X}$, 
\begin{align*}
\mathscr{P}(\varphi)&:= F(w+\varphi)-F(w)\\
&=p(|w+\varphi|^{p-1} -w^{p-1}) \pa_t w +p |w+\varphi|^{p-1} \pa_t \varphi + \mathcal{L}  \varphi.
\end{align*}
By the definition of $\mathcal{X}$, the properties of $w$ and Remark \ref{rem:dividecheck}, we have $\frac{\varphi}{w},\frac{\partial_t\varphi}{w},\frac{\partial_t w}{w}\in \mathscr{C}^{\alpha}(\overline\om\times [0, T_{\va_0}])$. Hence, $\mathscr{P}(\varphi)\in \mathcal{Y}$ if $\varphi\in \mathcal{X}$ and $\|\varphi\|_{\mathcal{X}}$ is small (so that $|\nabla\varphi|<|\nabla w|/2$ in $\{x\in\overline\Omega: d(x)<\delta_1\}\times[0,T_{\va_0}]$ and $w+\varphi>\delta_2$ in $\{x\in\overline\Omega: d(x)\ge \delta_1\}\times[0,T_{\va_0}])$.

Note that $\mathscr{P}(0)=0$ and
\[
\mathscr{P}'(0) \varphi =p w^{p-1} \pa_t \varphi + \mathcal{L}  \varphi +p(p-1) w^{p-2} \pa_t w \varphi \quad \mbox{for every }\varphi \in \mathcal{X},
\]
where $\mathscr{P}'(0)$ is the Fr\'echet derivative of $\mathscr{P}$ at $0$. 
It follows from Theorem \ref{thm:global-schauder} (our definitions on $\mathcal{X}$ and $\mathcal{Y}$ will ensure the compatibility conditions even after differentiating the linearized equation in the time variable) and elliptic estimates that  $\mathscr{P}'(0) $ is invertible when $\va_0$ is chosen sufficiently small. By the implicit function theorem, there exists a positive constant $\delta>0$ such that for any $\phi \in \mathcal{Y} $ with $\|\phi\|_{\mathcal{Y}} < \delta$ there exists a unique solution $\varphi \in \mathcal{X}$ of the equation $\mathscr{P}(\varphi)=\phi$.  Let $\overline T$ be sufficiently small. Pick a cutoff function $0\le \eta\le1 $ satisfying $\eta(s)=1$ for $s\le \overline T$ and $\eta(s)=0 $ for $s\ge 2\overline T$.   By the equation of $w$, we see that $\eta F(w)\in\mathcal{Y}$, and we claim that $\|\eta F(w) \|_{\mathcal{Y}} <\delta$ provided $\overline T$ is small. Indeed, we have that $F(w)|_{t=0}=0$ and 
\begin{align*}
\partial_t F(w)&=pw^{p-1}w_{tt}+Lw_t+p(p-1)w^{p-2}w_t^2\\
&=p(w^{p-1}-u_0^{p-1})w_{tt}+p(p-1)u_0^{p-2}(w_t^2(x,t)-w_t^2(x,0))\\
&=tp(p-1)w_{tt} \int_0^1 w^{p-2}(x,st)w_t(x,st)\,\ud s\\
&\quad+tp(p-1)u_0^{p-2}\Big(w_t(x,t)+w_t(x,0)\Big)\int_0^1w_{tt}(x,st)\,\ud s\\
&=O(t) d(x)^{p-1},
\end{align*}
from which the claim follows.

Therefore, there exists a function $\varphi\in \mathcal{X}$ such that
\[
\mathscr{P}(\varphi)=-\eta F(w).
\]
Thus $u=w+\varphi$ is a solution, since $u$ will be nonnegative using the estimates of $\varphi$ if $\overline T$ is sufficiently small.  

One can also verify that $u(\cdot,t)\in\mathcal{S}$ for all $t\in[0,T]$.

Finally, the uniqueness of $C^{3,2}(\overline\Omega\times[0,T])$ solutions follows from comparison principles (see, e.g., \cite{DaK, LiP}).  Therefore, we complete the proof. 
\end{proof}

\section{Time derivative estimates}  
\label{sec:time-derivative}

Consider nonnegative solutions of the equation
\be\label{eq:FDE-1}
\begin{cases}
\pa_t u^p =-\mathcal{L} u +u^p \quad \mbox{in } \om \times (0,T],\\
u(x,t)=0\quad \mbox{on } \partial\om \times (0,T],
\end{cases}
\ee
where $\Omega\subset\R^n$ is a smooth bounded domain, $1<p<\infty$ if $n=1,2$ and $1<p\le \frac{n+2}{n-2}$ if $n\ge 3$. Suppose that
\be \label{eq:pre-1}
\frac{1}{c_0}d(x) \le u(x,t)\le c_0 d(x)
\ee
for some constant $c_0\ge 1$, where $d(x)=dist(x,\partial\Omega)$.   Let 
\[
J[u(t)]= \int_{\om} \Big( |\nabla u(\cdot,t)|^2- bu(\cdot,t)^2 -\frac{2}{p+1} |u(\cdot,t)|^{p+1}\Big)\,\ud x.
\]

\subsection{Upper bounds}
Here, we prove the upper bounds of $\|u(\cdot,t)\|_{H^1_0(\Omega)}$ and $\partial_t u$.

 \begin{prop} \label{prop:H10} Let $u \in C^{3,2}(\overline \om \times [0,T])$  be a positive solution of  the fast diffusion equation \eqref{eq:FDE-1} satisfying \eqref{eq:pre-1}. Then there exists $C>0$ depending only on $n,\Omega$, $p,b,c_0$ and $ \int_\om |\nabla u(x,0)|^2$ (but not on $T$) such that
 \[
\|u(\cdot,t)\|_{H^1_0(\Omega)} \le C \quad \mbox{for all }t\in [0,T].
 \]
 \end{prop}

 \begin{proof}   Since $\pa_t u=0$ on $\partial\Omega\times(0,T)$, we have 
\begin{equation}\label{eq:energyderivative} 
\begin{split}
\frac{\ud }{\ud t}J[(u(t))]    &=2 \int_{\om} \nabla u \nabla \pa_t u-b u \pa_tu -u^{p} \pa_tu \\& = -2\int_\Omega (\Delta u+bu +u^p) \pa_tu\\&=-2p \int_\Omega u^{p-1}(\pa _t u)^2.
\end{split}
\end{equation}
It follows that $J[u]$ is non-increasing in $[0,T]$. The conclusion follows by using \eqref{eq:pre-1}.
\end{proof}

 \begin{prop} \label{prop:time-up} Let $u \in C^{3,2}(\overline \om \times [0,T])$  be a positive solution of  the fast diffusion equation \eqref{eq:FDE-1} satisfying \eqref{eq:pre-1}. Then there exists $C>0$ depending only on $n,\Omega$, $T, p,b,c_0$ and $ \int_\om |\nabla u(x,0)|^2$ such that
 \[
 \pa_t u(x,t) \le C \quad \mbox{for all }(x,t)\in \overline\Omega\times[T/2,T].
 \]
 \end{prop}

 \begin{proof}     
 By integrating \eqref{eq:energyderivative} in $t$, we have
\be \label{eq:2-bound0}
\int_{0}^T \int_\om u^{p-1}(\pa _t u)^2 \,\ud x\ud t \le   \frac{1}{2p}  \int_\om |\nabla u(x,0)|^2\,\ud x +C,
\ee
where $C>0$ depends only on $n,\om,b$, $p$ and $c_0$ in \eqref{eq:pre-1}. 


Differentiating the equation \eqref{eq:FDE-1} in $t$ variable and denoting $v= \pa_t u$,  we have
\[
pu^{p-1} \pa_t v +p(p-1) u^{p-2} v^2= -\mathcal{L} v +pu^{p-1} v \quad \mbox{in }\om \times (0,T]
\]
and
\[
v=0 \quad \mbox{on }\pa \om \times (0,T].
\]
Let $0<T_1<T_2 <\frac{T}{2}$, $\eta_1(t)$ be a smooth  function so that $\eta_1(t)=0$ for $t\le T_1$,  $\eta=1$ for $t>T_2$ and $|\eta'_1|\le \frac{2}{T_2-T_1}$. For $x_0\in \pa \om$ and $R>0$ small,   let $\eta_2\in C_c^2( B_{R}(x_0)) $ be a nonnegative cutoff function, and $\eta=\eta_1 \eta_2$.   For $k>0$,  let $w= (v-k)^+$. Using $\eta^2 w$ as a test function, we have, for $t\in (T_1,T)$, 
\begin{align}
&\frac{p}{2} \int_{\om}  u^{p-1} \eta^2 w^2\,\ud x\Big|_t + \frac12\int_{T_1}^t\int_\om |\nabla (\eta w)|^2 \nonumber \\&  \le  C \int_{T_1}^t \int_\om (u^{p-1} \eta |\pa_t\eta|  +|\nabla \eta|^2)w^2 + \int_{T_1}^t \int_\om(|b| +pu^{p-1}) \eta^2 (k+w)w \nonumber  \\&\quad   - p(p-1) \int_{T_1}^t \int_\om u^{p-2} vw(v-w) \nonumber \\& 
\le C\int_{T_1}^t \int_\om (u^{p-1} \eta |\pa_t\eta|  +|\nabla \eta|^2)w^2 + \int_{T_1}^t \int_\om(|b| +pu^{p-1}) \eta^2 (k+w)w. 
\label{eq:energy-inequality}
\end{align} 
Given  \eqref{eq:pre-1}  and \eqref{eq:energy-inequality}, by the proof of Proposition \ref{prop:local-bound} and estimates for uniform elliptic equations, it follows that, for any $\gamma>0$
\[
\sup_{\Omega \times [T/2, T] }  v^+ \le C \|v\|_{L^\gamma(\om\times [T/4, T] ) },
\]
where $C >0$ depends only on $n, \om, T, b, c_0,p$ and $\gamma$.  Using the H\"older inequality and \eqref{eq:2-bound0}, the proposition follows.
 \end{proof}

\begin{rem}
Note that we cannot obtain a lower bound of $\pa_t u$ via the above proofs, since \eqref{eq:energy-inequality} does not hold for $v=-\pa_t u$. In fact, this  is the main difficulty.
\end{rem}

The upper bound in Proposition \ref{prop:time-up} can be improved.

 \begin{prop} \label{prop:time-up-2} Let $u \in C^{3,2}(\overline \om \times [0,T])$  be a positive solution of  \eqref{eq:FDE-1} satisfying \eqref{eq:pre-1}. Then there exists $C>0$ depending only on $n,\Omega,b$, $T, p,c_0$ and $ \int_\om |\nabla u(x,0)|^2$  such that
 \[
 \pa_t u(x,t) \le C d(x) \quad \mbox{for all }(x,t)\in \overline\Omega\times[T/2,T].
 \]
 \end{prop}

 \begin{proof}
 This follows from a slight modification of the proof of Theorem 4.1 in \cite{DKV} with $m=1$ (see Remark 4.2 there). Note that for $v=\partial_t u$, we know from Proposition \ref{prop:time-up} that $v\le M$ for in $\overline\Omega\times[T/4,T]$. Due to the condition \eqref{eq:pre-1}, the same barrier function in the proof of Theorem 4.1 of \cite{DKV} with $m=1$ will be a supersolution of $pu^{p-1} \pa_t \Psi =-\mathcal{L} \Psi$, from which we obtain the desired upper bound using the comparison principle.
 \end{proof}

As pointed out by the referee, the above estimate can follow quickly from the B\'enilan-Crandall \cite{BC} estimate, which implies $u_t\le \frac{p}{(p-1)^2}\big[1-e^{-\frac{pt}{p-1}}\big]^{-1} u$ in our setting.
 
\subsection{Two lemmas}

We will use the following two lemmas. The first one is a Sobolev type inequality.  The following weighted Sobolev type inequalities are well known and well studied in, e.g., Kufner-Opic \cite{KO1,KO2}, Davies \cite{Davies} and Maz'ya \cite{Mazya}. Here, we include a proof for completeness. 

\begin{lem} \label{lem:sob}  Let $d(x)=\mbox{dist}(x,\pa \om)$ as before. There exists a constant $C>0$ depending only on $n,p$ and $\Omega$ such that
\be \label{eq:weight-sob}
\int_{\om} |\nabla \phi|^2 d^2 + \phi^2 d^{p+1} \,\ud x \ge C\left( \int_{\om} \phi^{p+1} d^{p+1} \,\ud x\right)^{\frac{2}{p+1}}
\ee
for any function $\phi\in H^{1}(\om)$, where $1<p<\infty$ if $n=1,2$ and $1<p\le \frac{n+2}{n-2}$ if $n\ge 3$. 
\end{lem}

\begin{proof}  Assume $n\ge 3$ first. By the variational method there exists a $C^2$ positive solution of
\[
-\Delta v(x)= \omega(x) v(x)^{\frac{n+2}{n-2}-\frac{1}{2(n-2)}} \quad  \mbox{in }\om \quad \mbox{and}\quad v=0\quad \mbox{on }\pa \om,
\]
where $\omega(x)\ge 0$ is the first eigenfunction of $-\Delta$ in $\Omega$ with zero boundary condition. Moreover, there exists $C>0$ such that $\frac 1C d(x)\le v(x) \le C d(x)$, and thus,  $W:=-v^{-p} \Delta v$ is bounded. Then we have
\begingroup
\allowdisplaybreaks
\begin{align*}
\int_{\om} |\nabla(\phi v)|^2\,\ud x&= \int_{\om}  \phi^2 |\nabla v|^2 +v^2 |\nabla \phi|^2 +2v\phi \nabla v \nabla \phi\,\ud x\\&
=\int_{\om} (-v \Delta v \phi^2+v^2 |\nabla \phi|^2) \\&
=\int_{\om} v^2 |\nabla \phi|^2 +W\phi^2 v^{p+1}\\&
\le C\int_{\om} |\nabla \phi|^2 d^2 + \phi^2 d^{p+1},
\end{align*}
\endgroup
where we used the integration by parts in the second equality and the fact $v=0$ on $\pa \om$.
Since $v\phi=0$ on $\pa \om$, by the Poincar\'e-Sobolev inequality there exists $C>0$ such that
\[
C\int_{\om} |\nabla(\phi v)|^2\,\ud x \ge  \left(\int_{\om} |v\phi|^{p+1}\,\ud x\right)^{\frac{2}{p+1}}.
\]
If $n=1,2$, one can show \eqref{eq:weight-sob} similarly. 
\end{proof}

The second one is an ODE argument. 

 \begin{lem} \label{lem:ODE} Let $\zeta(\cdot):[0,T) \to [0,\infty)$ be a $C^1$ solution of
\be \label{eq:ode-ineq}
\frac{\ud \zeta}{\ud t} \le \al  \zeta^{\mu_1} (\zeta^{\mu_2}+\zeta^{\mu_3}),
\ee
where $\al>0$, $\mu_1>0$, $\mu_2\in [0,1)$ and $\mu_3 \in [0,1]$ are constants. If $\int_0^T \zeta^{\mu_1}\,\ud t <\infty$,  then
\[
\zeta(t) \le C,
\]
where the constant $C$ depends only on $\al, \mu_2, \mu_3$, $\zeta(0)$ and $\int_0^T \zeta^{\mu_1}\,\ud t$.
\end{lem}
\begin{proof} Let
\[
H(\zeta)=\int_0^\zeta (s^{\mu_2}+s^{\mu_3})^{-1} \,\ud s, \quad \zeta\ge 0.
\]
Integrating the equation from $0$ to $t$,  we have
\[
0\le H(\zeta(t))\le H(\zeta(0))+  \al \int_{0}^{T} \zeta^{\mu_1}(t)\,\ud t.
\]
If $\mu_3=1$, we have
\[
\zeta(t) \le C_1\exp\left\{C_1H(\zeta(0))+  C_1\al \int_{0}^{T} \zeta^{\mu_1}(t)\,\ud t\right\},
\]
where $C_1$ depends only on $\mu_2$. 
If $\mu_3<1$,  we have
\[
\zeta(t) \le \left(C_2+C_2H(\zeta(0))+  C_2\al \int_{0}^{T} \zeta^{\mu_1}(t)\,\ud t\right)^{\frac{1}{1-\max\{u_2,\mu_3\}}},
\]
where $C_2$ depends only on $\mu_2$ and $\mu_3$.
Therefore, we complete the proof.
\end{proof} 

\subsection{Integral bounds}
Let $u\in C^{3,2}(\overline\Omega\times[0,T])$ be a solution of  the fast diffusion equation \eqref{eq:FDE-1} satisfying \eqref{eq:pre-1}. Let
\begin{equation}\label{eq:defnQ}
\mathcal{R}=u^{-p} \mathcal{L} u=1-\frac{p\partial_t u}{u}
\end{equation}
and
\begin{equation}\label{eq:defnMq}
M_q(t)= \int_\om |\mathcal{R} (x,t)-1|^q u(x,t)^{p+1}\,\ud x.
\end{equation}
We will show in this section that $M_q$ is bounded in $[T/2,T]$ for all $1\le q<\infty$.

\begin{prop} \label{prop:integral-bound} Let $u \in C^{3,2}(\overline \om \times [0,T])$  be a positive solution of  \eqref{eq:FDE-1} satisfying \eqref{eq:pre-1}. Let $\mathcal{R} $ be defined in \eqref{eq:defnQ}, and $M_q$ in \eqref{eq:defnMq} for every $q>1$. Then there exists a constant $C>0$, depending only on $n, \om, T, p$, $b, c_0$, $ q$ and $\|u(\cdot,0)\|_{H_0^1(\om)}$, such that
\[
M_{q} \le C  \quad \mbox{for }t>T/4.
\]
\end{prop}

\begin{proof} 
First of all, it follows from \eqref{eq:2-bound0} that
\be \label{eq:2-bound}
\begin{split}
\int_{0}^T M_2(t)\,\ud t \le  \frac{p}{2}  \int_\Omega |\nabla u(x,0)|^2 \,\ud x +C,
\end{split}
\ee
where $C$ depends only on $n,\om, p, b$ and $c_0$. 
By H\"older inequality and \eqref{eq:pre-1},  we have, for any $1\le q<2$,
\be \label{eq:1-bound}
\begin{split}
\int_0^T\int_{\om} |\mathcal{R}|^q u^{p+1}\,\ud x\ud t &\le CT^{1-q/2}  \left(\int_{0}^T M_2(t)\,\ud t\right)^{\frac{q}{2}} \\
&\le C\left(\int_\Omega |\nabla u(x,0)|^2\,\ud x+1 \right)^{\frac{q}{2}} .
\end{split}
\ee

Next, we derive some evolution equations.  By \eqref{eq:FDE-1}, we have
\be \label{eq:vol}
\pa_t u^{p+1}= -\frac{p+1}{p} (\mathcal{R}-1) u^{p+1}
\ee
and
\begin{align}
\pa_t \mathcal{R} &= -pu^{-p-1} \pa_t u  \mathcal{L} u+ u^{-p} \mathcal{L}  \pa_t u  \nonumber 
\\&=- u^{-2p} (-\mathcal{L} u+ u^{p}) \mathcal{L}u + \frac{1}{p}u^{-p} \mathcal{L}(u^{-p+1}(- \mathcal{L}u+ u^p))\nonumber\\
&=-\frac{1}{p}u^{-p} \mathcal{L} (u(\mathcal{R}-1))-  \mathcal{R} +\mathcal{R}^2.  
\label{eq:Q-evolve}
\end{align}
Since $u \in C^{3,2}(\overline \om \times [0,T])$, $u=\pa_t u=0$ on $\pa \om \times [0,T]$, and  $\mathcal{R} =1-pu_t/u$, using \eqref{eq:pre-1} we can show that  
\be \label{eq:Q-regular}
\sup_{\overline \om \times[0,T]} (|\mathcal{R} |+  |\nabla \mathcal{R}|)<\infty.
\ee  
For $q>1$, we have
\begin{align*}
\frac{\ud M_q}{\ud t}&= \int_{\om} q |\mathcal{R}-1|^{q-2}(\mathcal{R}-1) u^{p+1}\frac{\pa }{\pa t} \mathcal{R} \,\ud x+ \int_{\om} |\mathcal{R}-1|^{q} \frac{\pa }{\pa t} u^{p+1}\,\ud x  \\&
=\frac{q}{p} \int_{\om} |\mathcal{R}-1|^{q-2} (\mathcal{R}-1) u \Delta (u( \mathcal{R}-1)) \,\ud x+ \frac{q}{p}\int_{\om} b |\mathcal{R}-1|^{q}u^2\,\ud x    \\&
\quad + \int_{\om} \left( \Big(q-\frac{p+1}{p}\Big) |\mathcal{R}-1|^q (\mathcal{R} -1) +q|\mathcal{R}-1|^q\right) u^{p+1}\,\ud x .  
\end{align*}
Integrating by parts and using \eqref{eq:Q-regular}, we have
\begin{align*}
&\int_{\om} |\mathcal{R}-1|^{q-2} (\mathcal{R}-1) u \Delta (u( \mathcal{R}-1)) \,\ud x\\
&=-\int_{\om} \nabla[|\mathcal{R}-1|^{q-2} (\mathcal{R}-1) u] \, \nabla [u( \mathcal{R}-1)] \,\ud x\\
&=-(q-1) \int_{\om}  |\mathcal{R}-1|^{q-2}|\nabla (\mathcal{R}-1)|^2u^2  \,\ud x - \int_{\om} \nabla[|R-1|^qu]\nabla u\\
&=\int_{\om} |\mathcal{R}-1|^{q} (-bu^2-u^{p+1}\mathcal{R})\,\ud x-(q-1) \int_{\om}  |\mathcal{R}-1|^{q-2}|\nabla (\mathcal{R}-1)|^2u^2  \,\ud x. 
\end{align*} 
Hence,
\be \label{eq:evolution}
\begin{split}
&\frac{\ud M_q}{\ud t} = -\frac{q(q-1)}{p} \int_{\om}  |\mathcal{R} -1|^{q-2}|\nabla (\mathcal{R}-1)|^2u^2  \,\ud x \\&\quad + \int_{\om} \left( \frac{p-1}{p}\Big(q-\frac{p+1}{p-1}\Big) |\mathcal{R}-1|^q (\mathcal{R}-1) +\Big(q-\frac{q}{p}\Big) |\mathcal{R}-1|^{q} \right) u^{p+1} \,\ud x.
\end{split}
\ee 
Note that
\[
\int_{\om}  |\mathcal{R}-1|^q (\mathcal{R}-1) u^{p+1} \,\ud x = \int_{\om} |\mathcal{R}-1|^{q+1} u^{p+1}-2 |\mathcal{R}-1|^{q} (\mathcal{R}-1)^{-} u^{p+1}\,\ud x,
\]
where $(\mathcal{R}-1)^-=\max\{0,-(\mathcal{R}-1)\}$. By Proposition \ref{prop:time-up-2}, we have $(\mathcal{R}-1)^- \le C$ and thus
\be \label{eq:diff-ineq-1}
\begin{split}
&\frac{\ud }{\ud t} M_q(t) +\frac{p-1}{p}\left(\frac{p+1}{p-1}-q\right) M_{q+1} \\
&\quad\le C M_q- \frac{q(q-1)}{p}\int_{\om} |\mathcal{R}-1|^{q-2} |\nabla (\mathcal{R}-1)|^2 u^2 \,\ud x.
\end{split}
\ee
By Lemma  \ref{lem:sob}, we have
\begin{align*}
&\frac{q(q-1)}{p}\int_{\om} |\mathcal{R} -1|^{q-2} |\nabla (\mathcal{R}-1)|^2 u^2 \,\ud x \\ &= \frac{4(q-1)}{pq} \int_{\om} |\nabla |\mathcal{R}-1|^{q/2}| ^2u^2\,\ud x \\&
\ge \frac{1}{C} M_{\frac{q(p+1)}{2}}^{\frac{2}{p+1}} - CM_{q},
\end{align*}
where $C>1$ depends only on $\om$, $p$, $q$ and $c_0$.
Hence, we have
\be \label{eq:diff-ineq-2}
\frac{\ud }{\ud t} M_q(t)  +\frac{p-1}{p}\left(\frac{p+1}{p-1}-q\right) M_{q+1}  +\beta M_{\frac{q(p+1)}{2}}^{\frac{2}{p+1}} \le CM_q
\ee
for some $\beta>0$. (For $q<2$, one just need to consider approximations $M_{q,\va}=\int_{\om}[(\mathcal{R}-1)^2+\va^2]^{\frac{q}{2}}u^{p+1}$, and derive an inequality for $M_{q,\va}$ first, and finally send $\va\to 0$ to derive \eqref{eq:diff-ineq-2}.)

If $q>\frac{p+1}{p-1}$, by the interpolation inequality and Young inequality we have
\begin{align*}
M_{q+1} \le M_{q(p+1)/2}^{\frac{2}{q(p-1)}} M_{q}^{\frac{(p-1)(q+1)-(p+1)}{(p-1)q}} \le \va M_{q(p+1)/2}^{\frac{2}{p+1}} +C(\va) M_{q}^{\frac{(p-1)(q+1)-(p+1)}{q(p-1)-(p+1)}}.
\end{align*}
 Hence,
\be \label{eq:diff-ineq-3}
\begin{split}
\frac{\ud }{\ud t} M_q(t)  +\beta M_{\frac{q(p+1)}{2}}^{\frac{2}{p+1}} &\le C(M_q+  M_{q}^{\frac{(p-1)(q+1)-(p+1)}{q(p-1)-(p+1)}})\\
&=CM_q^{\frac{2}{p+1}}(M_q^{\frac{p-1}{p+1}}+  M_{q}^{\frac{p-1}{q(p-1)-(p+1)}+\frac{p-1}{p+1}}).
\end{split}
\ee

Lastly, we want to use the differential inequality of $M_q$ to prove its estimate. For $q\le \min\{2, \frac{p+1}{p-1}\}$,  by \eqref{eq:2-bound} and \eqref{eq:1-bound} we know $\int_{0}^T M_q\,\ud t\le C$. We can find $t_0<T/100$ such that $M_q(t_0)\le C$.  By \eqref{eq:diff-ineq-2}, we have $\frac{\ud }{\ud t} M_q\le C M_q$. Using  Lemma \ref{lem:ODE}, we have $M_q \le C$ for $t_0\le t\le T$. Then from integrating \eqref{eq:diff-ineq-2} it follows that $\int_{t_0}^T M_{q+1}\,\ud t \le C$ for $q<\frac{p+1}{p-1}$ and $q\le 2$. If $\frac{p+1}{p-1}>2$, we can keep using \eqref{eq:diff-ineq-2} and conclude that $M_{q} \le C$ in $[t_1, T]$ for some $t_1\le T/50$ and for all $q\le \frac{p+1}{p-1}$, and that $\int_{t_1}^T M_{q+1}(t)\,\ud t\le C$ for any $q<\frac{p+1}{p-1}$.  Then integrating \eqref{eq:diff-ineq-2} with $q=\frac{p+1}{p-1}$, we have
\begin{equation}\label{eq:Mqboard}
\int_{t_1}^T M_{\frac{(p+1)^2}{2(p-1)}}^{\frac{2}{p+1}} \,\ud t \le C.
\end{equation}
Note that
\begin{align*}
\frac{p-1}{q(p-1)-(p+1)}+\frac{p-1}{p+1} =1\quad\mbox{if}\quad q=q_0:=\frac{(p+1)^2}{2(p-1)}.
\end{align*}
Hence from \eqref{eq:diff-ineq-3} we have that
\[
\frac{\ud }{\ud t} M_{q_0}(t) \le CM_{q_0}^{\frac{2}{p+1}}(M_{q_0}^{\frac{p-1}{p+1}}+  M_{q_0}).
\]
By \eqref{eq:Mqboard} and Lemma \ref{lem:ODE}, we have $M_{q_0} \le C$ in $[t_2, T]$ for some $t_2\le T/25$. Let $q_k=q_0\left(\frac{p+1}{2}\right)^k$. Integrating \eqref{eq:diff-ineq-3} we have $\int_{t_2}^T M_{q_1}\le C$. Notice that since $q_k>q_0$,
\[
0<\frac{p-1}{q(p-1)-(p+1)}+\frac{p-1}{p+1} <1.
\]
Therefore, we can recursively use \eqref{eq:diff-ineq-3} and Lemma \ref{lem:ODE} to show that $M_{q_k} \le C$ in $[T/4, T]$ for all $k\ge 1$. We then complete the proof.
\end{proof}

By Proposition \ref{prop:integral-bound},  using H\"older's inequality we find that, for any $1\le q<\infty$, 
 \be \label{eq:almost}
 \int_{\om} \Big|\frac{\pa_t u}{u}\Big|^q\,\ud x \le C \quad \mbox{for }T/2\le t\le T,
 \ee
where  $C>0$ depending only on $n, \om, T, p$, $b, c_0$, $ q$ and $\|u(\cdot,0)\|_{H_0^1(\om)}$.
 
\subsection{Derivative bounds} 

By the integral bound \eqref{eq:almost}, we will be able to prove the derivative bound.

\begin{thm}\label{thm:uttbound} Let $u \in C^{3,2}(\overline \om \times [0,T])$  be a positive solution of  \eqref{eq:FDE-1} satisfying \eqref{eq:pre-1}. Then 
 $\pa_t^\ell u\in C^{0}(\overline \om\times [T/2, T])$ for every $\ell=0,1,2,\cdots$, and there holds 
\[
\|d^{-1} \pa_t^\ell u\|_{L^\infty(\om\times [T/2, T])} \le C,
\] 
where $C>0$ depends only on  $n,\om, T, p, b, c_0,\ell$ and $\|u(\cdot,0)\|_{H_0^1(\om)}$.   

Moreover, if $p$ is an integer, then $u\in C^\infty(\overline\om\times [T/2, T]))$, and 
\[
\|D_x^k\partial_t^\ell u\|_{L^\infty(\om\times [T/2, T]) }\le C,
\]
where $C>0$ depends only on  $n,\om, T, p, b, c_0,\ell,k$ and $\|u(\cdot,0)\|_{H_0^1(\om)}$.

If $p$ is not an integer, then $\pa_t^\ell u(\cdot,t)\in C^{2+p}(\overline \om)$ for every $\ell=0,1,2,\cdots$, and all $t\in[T/2,T]$, and there holds
\[
\sup_{t\in[T/2,T]}\|\partial_t^\ell u(\cdot,t)\|_{C^{2+p}(\overline\om) }\le C,
\]
where $C>0$ depends only on  $n,\om, T, p, b, c_0,\ell$ and $\|u(\cdot,0)\|_{H_0^1(\om)}$.
\end{thm}

To prove the above estimate, we first prove the estimate of $u_t$, and then  bootstrap to other high order derivatives.
\begin{prop} \label{prop:L2}
Let $u \in C^{3,2}(\overline \om \times [0,T])$  be a positive solution of  \eqref{eq:FDE-1} satisfying \eqref{eq:pre-1}. For $v=\partial_t u$, we have
\[
\|v\|_{L^\infty(\om \times [T/2, T])} \le C
\]
and 
\[
\int_{T/2}^T\int_{\om} (|\nabla v|^2 + |\pa_t v |^2 u^{p-1})\,\ud x\ud t \le C,
\]
where $C>0$ depending only on $n, \om, T, p$, $b, c_0$, $ q$ and $\|u(\cdot,0)\|_{H_0^1(\om)}$. 
\end{prop}

\begin{proof}  Recall that 
\be \label{eq:derivative}
pu^{p-1} \pa_t v +p(p-1) u^{p-2} v^2= -\mathcal{L} v+pu^{p-1} v \quad \mbox{in }\om \times (0,T]
\ee
and $v=0$ on $\pa \om \times (0,T]$.
 Using \eqref{eq:almost}, Proposition \ref{prop:local-bound} and the  De Giorgi-Nash-Moser estimates for uniformly parabolic equation, we have  
\[
\|v\|_{L^\infty(\om \times [T/2, T]) }\le C \left( \int_{T/4}^T \int_\om v^2 u^{p-1}\,\ud x\ud t\right)^{1/2}\le C,
\]
where $C>0$ depends only on $n, \om, T, p,b, c_0$. 

Using $v$ as a test function and using Proposition \ref{prop:integral-bound}, we have
\[
\int_{T/4}^{T} \int_\om |\nabla v |^2 -bv^2 \le C(M_2+M_3) \le C.
\]
Let $\eta(t)$ be a cutoff function satisfying $\eta=0$ for $t \le T/4$ and $\eta=1$ for $t\ge T/2$.
Multiplying  \eqref{eq:derivative} by $\pa _t v \eta^2$, integrating by parts and using  Proposition \ref{prop:integral-bound}, we have
\begin{align*}
\int_{T/2}^T \int_{\om} u^{p-1} (\pa_t v)^2 &\le C \left[M_3(T)+\int_{T/4}^TM_4\,\ud t+\int_{T/4}^{T} \int_\om (|\nabla v |^2-bv^2)   \ud x\ud t  \right]\\
& \le C.
\end{align*}
\end{proof}

\begin{proof}[Proof of Theorem \ref{thm:uttbound}]   
We will first prove 
\be \label{eq:tt-estimate}
|\pa_{tt} u| \le C \quad \mbox{in } \om\times [T/2, T] 
\ee
for some $C>0 $ depends only on $n,\om, T, p, b, c_0$ and $\int_{\om}|\nabla u(\cdot,0)|^2\,\ud x$.   

Since $u>0$ in $\Omega\times(0,T]$, the equation of $u$ is parabolic in $\Omega\times(0,T]$, and thus,  $u\in C^\infty(\Omega\times(0,T])$. Let $w=\pa_{tt} u=\pa_t v$. Since $u\in C^{3,2}(\overline \om \times [0,T])$, we have that $D_xv$ is $C^1$ in $x$ and $v$ is $C^1$ in $t$. By Lemma \ref{lem:lemma3.1inLSU} and Remark \ref{rem:dividecheck},  $\frac{v}{u} \in \mathscr{C}^{\al} (\overline\Omega \times [0,T])$ for some $\al>0$. Let $\eta(t)$ be a cutoff function satisfying $\eta=0$ for $t \le T/10$ and $\eta=1$ for $t\ge T/8$, and let $\tilde v=\eta v$. Then 
\[
pu^{p-1} \pa_t  \tilde v+p(p-1) u^{p-1}\frac{v}{u} \tilde v= -\mathcal{L} \tilde v+pu^{p-1} \tilde v+pu^{p-1}v\pa_t \eta \quad \mbox{in }\om \times (0,T],
\]
and $\tilde v=0$ on $\pa_{pa}(\om \times (0,T])$. By the Schauder estimates in Theorem \ref{thm:global-schauder}, we know that $\pa_tv  \in \mathscr{C}^{\al} (\overline\Omega \times [T/8,T])$ (without uniform estimates yet). Applying the difference quotient technique in the time variable to the equation \eqref{eq:derivative} for $v$, one can further obtain for $w=\pa_t v$ that $\pa_t w\in C(\overline\Omega\times[T/4,T])$ and $D^2_x w\in C(\overline\Omega\times[T/4,T])$.

Denote $p_i= p(p-1)\cdots (p-i)$, $i=1,2,3,\dots$. Then by differentiating \eqref{eq:derivative} in $t$, we have
\be  \label{eq:2nd-derivative}
p u^{p-1} \pa_t w +3p_1 u^{p-2} v w + \mathcal{L}w-pu^{p-1} w=-p_2 u^{p-3} v^3 +p_1 u^{p-2} v^2,
\ee 
which is now a linear equation. By \eqref{eq:almost}, Proposition \ref{prop:local-bound} and the  De Giorgi-Nash-Moser estimates for uniformly parabolic equation,   we have 
\begin{align*}
&\|w\|_{L^\infty(\Omega\times(T/2,T))} \\
&\le C \Big (\int_{\Omega\times[T/4,T]} u^{p-1} w^2\,\ud x\ud t\Big)^{1/2} + C \left\| \frac{|v|^3}{u^3}+\frac{v^2}{u^2}\right\|_{L^{s}(\Omega\times[T/4,T])}  
\end{align*}  
for some $s>\max\{\frac{\chi}{\chi-1},\frac{n+p+1}{2}\}$, where $\chi>1$ be the constant in Lemma \ref{lem:weightedsobolev} and $C>0 $ depends only on $n,\om, T, p, b, c_0$ and $\|u^{p-2} v \|_{L^{s}(\Omega\times[0,T])}$. Note that  $ v/u=\partial_t u/u=\frac{1}{p}(1-\mathcal{R}) $. By  Proposition \ref{prop:integral-bound} and Proposition \ref{prop:L2},  we conclude \eqref{eq:tt-estimate}.  

Using Proposition \ref{prop:L2} and applying elliptic estimates on each time slice to \eqref{eq:FDE-1}, we have for any $\beta\in (0,1)$
\[
\sup_{t\in[T/2,T]} \|D_x u(\cdot,t)\|_{C^\beta(\overline\Omega)}\le C.
\]
Using \eqref{eq:almost}, \eqref{eq:tt-estimate} and applying elliptic estimates on each time slice to \eqref{eq:derivative}, we have 
\[
\sup_{t\in[T/2,T]} \|D_x v(\cdot,t)\|_{C^\beta(\overline\Omega)}\le C.
\]
Using Lemma 3.1 on page 78 in \cite{LSU}, we have
\[
\|D_x u\|_{C_{x,t}^{\beta,\beta'}(\overline\Omega\times[T/2,T])}+\|D_x v\|_{C_{x,t}^{\beta,\beta'}(\overline\Omega\times[T/2,T])}\le C.
\]
Therefore, we can choose some $0<\alpha<\min\{1,p-1,\frac{2}{p-1}\}$ that
\be\label{eq:quotientholder1}
\|v/u\|_{\mathscr{C}^{\al} (\overline\Omega \times [3T/4,T])}\le C.
\ee
Then we can apply the Schauder estimates  in Theorem \ref{thm:global-schauder} to \eqref{eq:FDE-1}, \eqref{eq:derivative} and \eqref{eq:2nd-derivative}, and obtain  that
\begin{align*}
\|\pa_t u\|_{\mathscr{C}^{\al} (\overline\Omega \times [3T/4,T])}+\|D_x u\|_{\mathscr{C}^{\al} (\overline\Omega \times [3T/4,T])} +\sup_{t\in(3T/4,T)}\|D_x^2 u\|_{C^{\alpha} (\Omega)} &\le C,\\
\|\pa_t v\|_{\mathscr{C}^{\al} (\overline\Omega \times [3T/4,T])}+\|D_x v\|_{\mathscr{C}^{\al} (\overline\Omega \times [3T/4,T])} +\sup_{t\in(3T/4,T)}\|D_x^2 v\|_{C^{\alpha} (\Omega)} &\le C,\\
\|\pa_t w\|_{\mathscr{C}^{\al} (\overline\Omega \times [3T/4,T])}+\|D_x w\|_{\mathscr{C}^{\al} (\overline\Omega \times [3T/4,T])} +\sup_{t\in(3T/4,T)}\|D_x^2 w\|_{C^{\alpha} (\Omega)} &\le C.
\end{align*}
Hence
\be\label{eq:quotientholder2}
\|w/u\|_{\mathscr{C}^{\al} (\overline\Omega \times [3T/4,T])}\le C.
\ee
Applying elliptic estimates on each time slice to \eqref{eq:FDE-1}, we have
\[
\sup_{t\in(3T/4,T)}\|D_x^3 u(\cdot,t)\|_{C^{\alpha} (\Omega)} \le C.
\]
Moreover, for all $t\in [T/3,T]$, we have
\be\label{eq:estcontinue}
\begin{split}
&\|u^{1-p}\mathcal{L}u \|_{C^{2+\alpha}(\overline \om)}+ \|(u^{1-p}\mathcal{L} )^2 u\|_{C^\alpha(\overline \om)}\\
&\le C (\|u \|_{C^{2+\alpha}(\overline \om)}+\|u_t \|_{C^{2+\alpha}(\overline \om)}+\|u^{-1}u_t^2\|_{C^\alpha(\overline \om)}+\|u_{tt}\|_{C^\alpha(\overline \om)})\\
&\le C.
\end{split}
\ee

Now we can keep differentiating \eqref{eq:2nd-derivative},  using \eqref{eq:quotientholder1} and \eqref{eq:quotientholder2}, and applying the Schauder estimates obtained in Theorem \ref{thm:global-schauder} to show that for all $\ell=2,3,4,\cdots$,
\be\label{eq:quotientholderk}
\|\pa_t^{\ell-1} u/u\|_{\mathscr{C}^{\al} (\overline\Omega \times [7T/8,T])}\le C,
\ee
and
\be\label{eq:schauderk}
\begin{split}
&\|\pa_t^{\ell+1} u\|_{\mathscr{C}^{\al} (\overline\Omega \times [7T/8,T])}+\|D_x \pa_t^{\ell}u\|_{\mathscr{C}^{\al} (\overline\Omega \times [7T/8,T])} +\sup_{t\in(7T/8,T)}\|D_x^2 \pa_t^{\ell}u\|_{C^{\alpha} (\Omega)} \\
&\le C.
\end{split}
\ee
Applying elliptic estimates on each time slice to the equation of $\partial_t^{k-1}u$ (and differentiating it in $x$ once since $p>1$), one has 
\[
\sup_{t\in(7T/8,T)}\|D_x^3 \partial_t^{\ell-1}u(\cdot,t)\|_{C^{\alpha} (\Omega)} \le C.
\]
If $p$ is an integer, then for every $k\ge 0$, one can keep applying elliptic estimates to the equations of $u, \partial_t u,\cdots,\partial_t^{k+\ell}u$ on each time slice and using standard bootstrap arguments for elliptic equations to show that $\partial_t^{\ell}u(\cdot,t)\in C^{k+3+\alpha}(\overline\om)$ for all $t\in[T/2,T]$, and 
\[
\|D_x^{k+3}\partial_t^\ell u\|_{L^\infty(\overline\om\times [T/2, T])) }\le C.
\]
If $p$ is not an integer, then we can only different these equations in $x$ only $[p]$ (the integer part of $p$) times, and thus, obtain
\[
\sup_{t\in [T/2, T]}\|\partial_t^\ell u(\cdot,t)\|_{C^{2+p}(\overline\om) }\le C.
\]
\end{proof}

\section{Optimal boundary regularity} 
\label{sec:last}

Let $\lda_1>0$ be the first Dirichlet eigenvalue of $-\Delta$ in $\Omega$, and $b<\lda_1$ is a constant. Let $u$ be a bounded weak solution of 
\be \label{eq:main-new} 
\begin{cases}
\pa_t u^p =\Delta u +bu& \quad \mbox{in }\om \times (0,\infty),\\
u=0& \quad \mbox{on }\pa \om \times (0,\infty),
\end{cases}
\ee 
i.e.,  $u\in L^2_{loc}([0,\infty); H^1_0(\Omega))\cap L^\infty(\om \times (0,\infty))$, $u\ge 0$, $u^p\in C([0,\infty);L^{1}(\Omega))$,  and satisfies 
\[
\int_{\Omega\times(0,\infty)} \Big(u^p\pa_t\varphi -\nabla u\cdot\nabla\varphi+bu\varphi\Big)\,\ud x\ud t=0\quad\forall\,\varphi\in C_c^1(\Omega\times(0,\infty)).
\]
Note that $u(\cdot,0)$ is well defined, and is considered as the initial data if needed. Let $T^*$ be the extinction time of $u$. Consider the rescaled solution
\begin{align}\label{eq:changeutoU}
U(x,t)= \Big (\frac{p}{(p-1)(T^*-\tau)}\Big)^{\frac{1}{p-1}}u(x,\tau), \quad t =\frac{(p-1)}{p}  \ln \Big(\frac{T^*}{T^*-\tau}\Big).
\end{align}
Then 
\be \label{eq:main-scl-1} 
\begin{cases}
\pa_t U^p =\Delta U +bU+ U^{p}& \quad \mbox{in }\om \times (0,\infty),\\
U=0& \quad \mbox{on }\pa \om \times (0,\infty). 
\end{cases}
\ee   
Let $\delta>0$ and $T>2\delta$ be constants. 

If $1<p<\infty$ when $n=1,2$; or $1<p<\frac{n+2}{n-2}$ when $n\ge 3$, by Theorem 1.1 of DiBenedetto-Kwong-Vespri \cite{DKV} (and a similar proof for $b\neq 0$) we have 
\be \label{eq:pre-2}
\frac{1}{c_0}d(x) \le U(x,t)\le c_0 d(x), \quad t\in (\delta, \infty),
\ee 
for some $c_0\ge 1$ depending  on $\om$, $n, b, p, \delta$, $\|U(\cdot,0)\|_{H^{1}_0(\om)}$ and $\|U(\cdot,0)\|_{L^{p+1}(\om)}$ only. 

If $p=\frac{n+2}{n-2}$ and $n\ge 3$, from Lemma 5.2 and Proposition 6.2 in \cite{DKV}, we have 
\be \label{eq:pre-3}
\frac{1}{c_0}d(x) \le U(x,t)\le c_0 d(x), \quad t\in (\delta, T),
\ee 
with $c_0$ depending  on $\om$, $n,b,p,\delta$, and also $\|U\|_{L^{\infty}(\om\times(0,T))}$. Note that when $b=0$, $\|U\|_{L^{\infty}(\om\times(0,T))}$ may blow up as $T\to\infty$, see, e.g., Galaktionov-King \cite{GKing} and Sire-Wei-Zheng \cite{SWZ}. In \cite{JX20}, we proved that if $n\ge 4$ and $0<b<\lambda_1$, then $\|U\|_{L^{\infty}(\om\times(0,T))}$ is uniformly bounded independent of $T$.

It is known that for all $0<\delta<T<\infty$, $U\in C^\infty(\Omega\times[\delta,T])\cap C^\alpha(\overline\Omega\times[\delta,T])$ for some $\alpha\in(0,1)$.  We will prove its regularity up to the boundary $\pa\Omega$.
  
\begin{thm} \label{thm:last}   Let $U\in L^2_{loc}([0,\infty); H^1_0(\Omega))\cap L_{loc}^\infty(\overline\Omega\times[0,\infty))$, $U\ge 0$, $U^p\in C([0,\infty);L^{1}(\Omega))$  be a weak solution of \eqref{eq:main-scl-1}, and satisfy \eqref{eq:pre-2} or \eqref{eq:pre-3} (depending on the choices of $p$ and $n$ as above). 

 If $p$ is an integer, then $U\in C^\infty(\overline \om\times (0,\infty))$. Moreover, for any $2\delta<t<T$ and $x\in\overline\Omega$, 
\[
d(x)^{-1}|\pa_t^l  U(x,t)| +\|D_x^k\pa_t^l   U(\cdot,t)\|_{L^{\infty}(\overline\Omega)}  \le C, 
\]   
where $k,l\ge 0$ are integers,  and $C>0$ depends only on $n,\om, b, c_0, p,\delta, k$ and $l$.  
 
If $p$ is not an integer, then $ \pa_t^l U(\cdot, t) \in C^{2
+p}(\overline \om)$ for all $l\ge 0$ and all $t\in(0,\infty)$.  Moreover, for any $2\delta<t<T$ and $x\in\overline\Omega$,
\[
d(x)^{-1}|\pa_t^l  U(x,t)| +\|\pa_t^l   U(\cdot,t)\|_{C^{2+p}(\overline\Omega)}  \le C, 
\]   
where  $C>0$ depends only on $n,\om, c_0, b,p,\delta$ and $l$.  


\end{thm} 

\begin{proof} 
Using the standard energy estimates, we have that there exists $t_0\in[\delta/2,\delta]$ such that $U(\cdot,t_0)\in H^1_0(\Omega)$, and $\|U(\cdot,t_0)\|_{H^1_0(\Omega)}\le C$ for some $C>0$ depending only on $n,\om, c_0, b,p,\delta$. By a shifting in the time variable, we may assume that 
\[
t_0=0\quad\mbox{and denote }U_0=U(\cdot,0).
\]
 Let
\begin{equation}\label{eq:scaleforU}
w(x,s)=\left[\frac{p-(p-1)s}{p}\right]^{\frac{1}{p-1}}U(x,t(s)),\quad t(s)= \frac{p}{p-1}\log\Big(\frac{p}{p-(p-1)s}\Big).
\end{equation}
Then
\begin{equation}\label{eq:scaleforUinw}
\begin{cases}
\pa_s w^p=\Delta w+bw\quad\mbox{in }\Omega\times(0,\infty).\\
w=0\quad\mbox{on }\pa\Omega\times(0,\infty).
\end{cases}
\end{equation}

Step 1: If $U_0\in\mathcal{S}$ (defined in \eqref{eq:good-initial}), then the theorem follows from repeatedly using  Theorem \ref{thm:short-time}, the assumption \eqref{eq:pre-2} or \eqref{eq:pre-3}, Proposition \ref{prop:H10} and Theorem \ref{thm:uttbound}.

Step 2: Otherwise, since $U_0\in H^1_0(\Omega)\cap C_0(\overline\Omega)$,  one can take a sequence of functions $\{U_0^{(j)}\}$ in $\mathcal{S}$ to approximate $U_0$ in $H_0^1(\om)\cap C_0(\overline\Omega)$. Let $w^{j}$ be the weak solution to \eqref{eq:scaleforUinw} with initial data $w^{(j)}(\cdot, 0)=U_0^{(j)}$. It follows from the comparison principle that there exists $C>0$ independent of $j$ such that 
\[
w^{(j)}\le C \quad\mbox{in }\Omega\times(0,\infty).
\]
By the H\"older continuity estimate up to $\pa \om $ (see Chen-DiBenedetto \cite{CDi}), there exists $\alpha>0$ such that for every $\delta>0$, there exists $C>0$ independent of $j$ such that
\[
\|w^{(j)}\|_{C^\alpha(\overline\Omega\times[\delta,1])}\le C \quad\mbox{in }\overline\Omega\times[\delta,1].
\]
By the Ascoli-Arzela theorem and the uniqueness of weak solutions to \eqref{eq:scaleforUinw}, there exists a subsequence of $\{w^{(j)}\}$, which is still denoted as $\{w^{(j)}\}$, such that it converges uniformly to $w$ on $\overline\Omega\times[\delta,1]$. Since $w>0$ in  $\Omega\times[\delta,1]$, there exists $c>0$ independent of $j$ such that 
\[
\int_{\Omega} |w^{(j)}(x,t)|^{p+1}\,\ud x\ge c\quad\mbox{in }[\delta,1] \mbox{ for all large }j.
\]
Then, by Proposition 6.2 of DiBenedetto-Kwong-Vespri \cite{DKV}, there exists $C$ independent of $j$ such that 
\begin{equation}\label{eq:appequniformbound}
\frac{d(x)}{C}\le w^{(j)}\le Cd(x)\quad\mbox{in }\Omega\times[3\delta/2, 1] \mbox{ for all large }j.
\end{equation}
Meanwhile,  since $U^{(j)}_0\in\mathcal{S}$, then from Step 1, we know $w^{(j)}\in C^{3,2}(\overline\Omega\times[0,1])$ for all large $j$. Then, it follows from \eqref{eq:appequniformbound}, Proposition \ref{prop:H10} and Theorem \ref{thm:uttbound} that there exists a subsequence of $\{w^{(j)}\}$ converging to $w$ in $C^{3,2}(\Omega\times[2\delta,1])$. In particular, $w(\cdot,2\delta)\in\mathcal{S}$, and thus, the theorem follows from Step 1. 
\end{proof} 

\begin{proof}[Proof of Theorem \ref{thm:main}]
 It follows from Theorem \ref{thm:last} by taking $b=0$, together with \eqref{eq:dkvsubcritical} or \eqref{eq:dkvcritical}. 
\end{proof} 

\begin{proof}[Proof of Theorem \ref{thm:main2}]
 It follows from Theorem \ref{thm:last} by taking $b=0$, together with \eqref{eq:dkvsubcritical}  and Lemma 5.2 of DiBenedetto-Kwong-Vespri \cite{DKV} about the lower and upper bounds on $T^*$. 
\end{proof} 

\begin{proof}[Proof of Theorem \ref{thm:main3}]
 It follows from Theorem \ref{thm:last} by taking $b=0$, together with \eqref{eq:dkvcritical} and Lemma 5.2 of DiBenedetto-Kwong-Vespri \cite{DKV} about the lower and upper bounds on $T^*$. 
\end{proof} 

\begin{proof}[Proof of Corollary \ref{cor:main}]  
Let
\[
w=v-S.
\]
Then
\begin{align*}
pv^{p-1} w_t&=\Delta w + c(x,t) d(x)^{p-1}w \quad\mbox{in }\Omega\times(0,\infty),\\
w&=0 \quad\mbox{on }\pa\Omega\times(0,\infty),
\end{align*}
where
\[
c=\int_0^1 \left[\frac{\theta v+(1-\theta)S}{d(x)}\right]^{p-1}\,\ud\theta.
\]
Note that the function $c$ shares the same regularity and same estimates as $v$. 
Then it follows from Proposition \ref{prop:local-bound} and H\"older inequality that
\begin{align}\label{eq:finaldifferenceinfinity}
\|w\|_{L^\infty(\overline \Omega\times[T-2,T])}\le C \left(\int_{\Omega\times[T-3,T])}w^2 S^{p-1}\,\ud x\right)^{1/2}.
\end{align}
By Theorem \ref{thm:global-schauder} and elliptic Schauder estimates on each time slice, we have
\begin{align*}
\|D_x^k\partial_t^\ell w\|_{L^\infty(\om\times [T-1, T]) }\le C (k,\ell)\|w\|_{L^\infty(\overline \Omega\times[T-2,T])},
\end{align*}
where $1\le k<\infty$ if $p$ is an integer, and
 \begin{align*}
\sup_{t\in[T-1,T]}\|\partial_t^\ell w(\cdot,t)\|_{C^{2+p}(\overline\om) }\le C(\ell)\|w\|_{L^\infty(\overline \Omega\times[T-2,T])}
\end{align*}
 if $p$ is not an integer.
 In particular, combining with \eqref{eq:finaldifferenceinfinity}, we have
\begin{align}\label{eq:finaldifferencek}
\|D_x^k w(\cdot,T)\|_{L^\infty(\om)}\le C(k) \left(\int_{\Omega\times[T-3,T])}w^2 S^{p-1}\,\ud x\right)^{1/2},
\end{align}
where $1\le k<\infty$ if $p$ is an integer, and 
 \begin{align}\label{eq:finaldifferencek2}
\| w(\cdot,T)\|_{C^{2+p}(\overline\om) }\le  \left(\int_{\Omega\times[T-3,T])}w^2 S^{p-1}\,\ud x\right)^{1/2}
\end{align}
 if $p$ is not an integer.

Since $v(\cdot,t)=S(\cdot)=0$ on $\pa\Omega\times[1,\infty)$, we have 
\[
\left\|\frac{v(\cdot,t)-S}{S}\right\|_{C^{k}(\overline \om)}\le C \|v(\cdot,t)-S\|_{C^{k+1}(\overline \om)}=C \|w(\cdot,t)\|_{C^{k+1}(\overline \om)},
\] 
where $k$ is any positive integer if $p$ is an integer, and 
\[
\left\|\frac{v(\cdot,t)-S}{S}\right\|_{C^{1+p}(\overline \om)}\le C \|v(\cdot,t)-S\|_{C^{2+p}(\overline \om)}=C \|w(\cdot,t)\|_{C^{2+p}(\overline \om)}
\] 
if $p$ is not an integer. 
Then the corollary follows from combining \eqref{eq:finaldifferencek}, \eqref{eq:finaldifferencek2}, \eqref{eq:relativeerror} and \eqref{eq:sharpgammap}.
\end{proof}

\appendix

\section{Weak solutions to a singular linear equation}

Suppose that $a$ and $A$ satisfy  \eqref{eq:ellip}, $\pa_ta, b\in L^q(Q_1^+) $ for some $q>\frac{\chi}{\chi-1}$,  where $\chi>1$ is the constant in Lemma \ref{lem:weightedsobolev}, and $f_0,f_1,\cdots,f_n\in L^2(Q_1^+)$. Let $F=(f_1,\cdots,f_n)$.
\begin{defn}\label{defn:weaksolutionp}
We say $u$ is a weak solution of 
\be\label{eq:degiorgilinearappendix}
a(x,t) x_n^{p-1} \pa_t u -\mathrm{div}(A(x,t) \nabla u)+b(x,t) u=f_0(x,t) + \mathrm{div} F \quad \mbox{in }Q_1^+
\ee
with the partial boundary condition 
\be \label{eq:linear-eq-D-app}
u=0 \quad \mbox{on }\pa' B_1^+\times[-1,0],
\ee
if $u\in C ((-1,0]; L^2(B_1^+,x_n^{p-1}\ud x)) \cap  L^2((-1,0];H_{0,L}^1(B_1^+)) $ and satisfies
\begin{equation}\label{eq:definitionweaksolution}
\begin{split}
&\int_{B^+_1}a(x,s) x_n^{p-1}  u(x,s) \varphi(x,s)\,\ud x-\int_{-1}^s\int_{B_1^+} x_n^{p-1}(\varphi\partial_t a+a\partial_t \varphi)u\,\ud x\ud t\\
&+ \int_{-1}^s\int_{B_1^+} [(A(x,t) \nabla u)\cdot\nabla\varphi+b u\varphi]\,\ud x\ud t=\int_{-1}^s\int_{B_1^+} (f_0\varphi+F\cdot\nabla\varphi)\,\ud x\ud t
\end{split}
\end{equation}
a.e. $s\in (-1,0]$ for every $\varphi$ in the space 
\begin{equation}\label{eq:V110}
\begin{split}
V^{1,1}_0(Q_1^+):=\{\varphi\in L^2(Q_1^+): \ &\partial_t\varphi\in L^2(Q_1^+,x_n^{p-1}\ud x\ud t), D_x\varphi\in L^2(Q_1^+), \\
&\varphi\equiv 0\mbox{ on }\partial B_1^+\times(-1,0]\} 
\end{split}
\end{equation}
and $\varphi (\cdot,0)\equiv 0$ in $B_1^+$ (in the trace sense).
\end{defn}
 Let 
\[
\|\varphi\|_{V^{1,1}_0(Q_1^+)}=\|\partial_t\varphi\|_{L^2(Q_1^+,x_n^{p-1}\ud x\ud t)}+\|D_x \varphi\|_{L^2(Q_1^+)}.
\] 
Then $V^{1,1}_0(Q_1^+)$ is a Banach space embedded in $V_0^1(Q_1^+)$ (cf. \eqref{eq:V10norm}). Using Lemma \ref{lem:weightedsobolev}, one can verify that each integral in \eqref{eq:definitionweaksolution} is finite. Let $W^{1,1}_2(Q_1^+)=\{g\in L^2(Q_1^+): \pa_tg\in L^2(Q_1^+), D_x g \in L^2(Q_1^+)\}$ be the standard Sobolev space.  Then for $\varphi\in V^{1,1}_0(Q_1^+)$, one can check that $\forall\,\va>0$, $\varphi_\va(x,t):=e^{-\va /x_n}\varphi(x,t)\in W^{1,1}_2(Q_1^+)$ and $\|\varphi_\va-\varphi\|_{V^{1,1}_0(Q_1^+)}\to 0$ as $\va\to 0^+$.

\begin{defn}\label{defn:weaksolutionwithinitialtime}
We say that $u$ is a weak solution of \eqref{eq:degiorgilinearappendix} with the partial boundary condition \eqref{eq:linear-eq-D-app} and the initial condition $u(\cdot,-1)\equiv 0$, if  $u\in C ([-1,0]; L^2(B_1^+,x_n^{p-1}\ud x)) \cap  L^2((-1,0];H_{0,L}^1(B_1^+)) $, $u(\cdot,-1)\equiv 0$, and satisfies \eqref{eq:definitionweaksolution} for all $\varphi\in V^{1,1}_0(Q_1^+)$. 
\end{defn}

\begin{defn}\label{defn:weaksolutionglobal}
We say that $u$ is a weak solution of \eqref{eq:degiorgilinearappendix} with the full boundary condition $u\equiv 0$ on $\pa_{pa} Q_1^+$, if  $u\in C ([-1,0]; L^2(B_1^+,x_n^{p-1}\ud x)) \cap  L^2((-1,0];H_{0}^1(B_1^+)) $, $u(\cdot,-1)\equiv 0$, and satisfies \eqref{eq:definitionweaksolution} for all $\varphi\in V^{1,1}_0(Q_1^+)$.
\end{defn}

Note that $C ([-1,0]; L^2(B_1^+))\subset W^{1,1}_2(Q_1^+)$.

\begin{defn}\label{defn:weaksolutionglobalin}
Let $g\in W^{1,1}_2(Q_1^+)$. We say that $u$ is a weak solution of \eqref{eq:degiorgilinearappendix} with the inhomogeneous boundary condition $u\equiv g$ on $\pa_{pa} Q_1^+$, if  $u\in C ([-1,0]; L^2(B_1^+,x_n^{p-1}\ud x)) \cap  L^2((-1,0];H^1(B_1^+)) $, $u= g$ on $\pa_{pa} Q_1^+$, and $v:=u-g$ is a weak solution of
\begin{align*}
&a x_n^{p-1} \pa_t v -\mathrm{div}(A \nabla v)+b(x,t) v\\
&\quad=-a x_n^{p-1} \pa_t g -b(x,t) g+ f_0(x,t) + \mathrm{div} (A(x,t) \nabla g+F)
\end{align*}
with homogeneous boundary condition $v\equiv 0$ on $\pa_{pa} Q_1^+$.
\end{defn}

We have the following uniqueness of weak solutions.

\begin{thm}\label{thm:uniquenessofweaksolution}
Suppose that $a$ and $A$ satisfy  \eqref{eq:ellip}, $\pa_ta, b\in L^q(Q_1^+) $ for some $q>\frac{\chi}{\chi-1}$,  where $\chi>1$ is the constant in Lemma \ref{lem:weightedsobolev}, and $f_0, f_1,\cdots, f_n\in L^2(Q_1^+)$. Then there exists at most one weak solution of \eqref{eq:degiorgilinearappendix} with the full boundary condition $u\equiv 0$ on $\pa_{pa} Q_1^+$.
\end{thm}

\begin{proof}
Using the Steklov average of $u$ as a test function, and taking a limit in the end, and also using H\"older inequality and Lemma \ref{lem:weightedsobolev}, we have
\begin{equation}\label{eq:energyestimateu}
\|u\|^2_{V^1_0(B_1^+\times(-1,t])} \le C\|u\|^2_{L^2(B_1^+\times(-1,t])}+ C \sum_{j=0}^n\|f_j\|^2_{L^2(Q_1^+)},
\end{equation}
where $C>0$ depends only on $n$, $\lda, \Lda$, $p$, $\|\pa_t a\|_{L^q(Q_1^+)}$ and  $\|b\|_{L^q(Q_1^+)}$. 

Now let $f_j\equiv 0$ for all $j=0,1,\cdots,n$. By Lemma \ref{lem:weightedsobolev} and H\"older's inequality, we have from \eqref{eq:energyestimateu} that
\[
\|u\|^2_{L^{2\chi}(B_1^+\times(-1,t])}\le C\|u\|^2_{L^2(B_1^+\times(-1,t])}\le C (t+1)^{\frac{\chi-1}{2\chi}} \|u\|^2_{L^{2\chi}(B_1^+\times(-1,t])}.
\]
Therefore, $u\equiv 0$ in $B_1^+\times(-1,t])$ if $t-(-1)$ is small. Repeating this argument for finitely many times, we have $u\equiv 0$ in $B_1^+\times(-1,0]$.
\end{proof}

Next, we want to study the weak solutions of the following specially designed equation
\begin{equation}\label{eq:linear-eq-app}
a x_n^{p-1} \pa_t u -\mathrm{div}(A\nabla u)+b_1 u+b_2x_n^{p-1}u=x_n^{p-1}f_0 + \mathrm{div} F \quad \mbox{in }Q_1^+.
\end{equation}
We suppose that $-\mathrm{div} (A \nabla ) +b_1$ is coercive, i.e., there exists a constant $\bar \lda>0$ such that
\be \label{eq:coer-app}
\int_{B_1^+} A\nabla \phi \nabla \phi +b_1 \phi^2 \ge \bar \lda \int_{B_1^+} \phi^2 \quad \quad \forall \phi\in H^1_0(B_1^+), a.e.\ t\in[-1,1].
\ee
\begin{thm}\label{thm:existenceofweaksolution}
Suppose that $a$ and $A$ satisfy  \eqref{eq:ellip}, $a\in C^1(\overline Q_1^+)$, $b_2\in L^q(Q_1^+) $ for some $q>\frac{\chi}{\chi-1}$,  where $\chi>1$ is the constant in Lemma \ref{lem:weightedsobolev}, $b_1\in L^\infty(Q_1^+) $, and $f_0, f_1,\cdots, f_n\in L^2(Q_1^+)$, $F=(f_1,\cdots,f_n)$. Suppose that \eqref{eq:coer-app} holds. Then there exists a unique weak solution of \eqref{eq:linear-eq-app} with the full boundary condition $u\equiv 0$ on $\pa_{pa} Q_1^+$. Moreover, there exists $C>0$, which depends only on $n$, $\bar\lambda, \lda$, $\|\pa_t a\|_{L^q(Q_1^+)}$ and  $\|b_2\|_{L^q(Q_1^+)}$, such that 
\begin{equation}\label{eq:energyestimateunoL2}
\|u\|^2_{V^1_0(Q_1^+)} \le  C \sum_{j=0}^n\|f_j\|^2_{L^2(Q_1^+)}.
\end{equation}
\end{thm}

\begin{proof} 
For all $\va>0$, there exists a unique (standard energy) weak solution $u_\va\in C ([-1,0]; L^2(B_1^+)) \cap  L^2((-1,0];H_{0}^1(B_1^+))$ to the uniformly parabolic equation
\be\label{eq:degiorgilinearappappendix}
\begin{split}
&a\cdot  (x_n+\va)^{p-1} \pa_t u_\va -\mathrm{div}(A\nabla u_\va)+b_1 u_\va+b_2\cdot (x_n+\va)^{p-1}u_\va\\
&\quad=(x_n+\va)^{p-1}f_0 + \mathrm{div} F \quad \mbox{in }Q_1^+.
\end{split}
\ee
with  $u_\va\equiv 0$ on $\pa_{pa} Q_1^+$. Using \eqref{eq:coer-app}, the energy estimate for $u_\va$ implies that
\begin{align*}
&\sup_{t\in[-1,s]}\int_{B_1^+}u_\va^2(x_n+\va)^{p-1}\,\ud x+\|\nabla u_\va\|^2_{L^2(B_1^+\times[-1,s])}\\
&\quad\le C \|(x_n+\va)^{p-1} u_\va^2\|_{L^{\frac{q}{q-1}}(B_1^+\times[-1,s])}+ C \sum_{j=0}^n\|f_j\|^2_{L^2(Q_1^+)} ,
\end{align*}
where $C>0$ depends only on $n$, $\bar\lambda, \lda$, $\|\pa_t a\|_{L^q(Q_1^+)}$ and  $\|b_2\|_{L^q(Q_1^+)}$. Since $q>\frac{\chi}{\chi-1}$, one can use H\"older's inequality, Young's inequality and Lemma  \ref{lem:weightedsobolev} to obtain  
\begin{align*}
 &\|(x_n+\va)^{p-1} u_\va^2\|_{L^{\frac{q}{q-1}}(B_1^+\times[-1,s])}\\
 &\le \delta \|u_\va\|^2_{V^1_0(B_1^+\times[-1,s])}+C(\delta)  \|(x_n+\va)^{p-1} u_\va^2\|_{L^{1}(B_1^+\times[-1,s])}.
\end{align*}
This implies that
\begin{align*}
&\sup_{t\in[-1,s]}\int_{B_1^+}u_\va^2(x_n+\va)^{p-1}\,\ud x+\|\nabla u_\va\|^2_{L^2(B_1^+\times[-1,s])}\\
&\quad\le C \|(x_n+\va)^{p-1} u_\va^2\|_{L^{1}(B_1^+\times[-1,s])}+ C \sum_{j=0}^n\|f_j\|^2_{L^2(Q_1^+)}.
\end{align*}
In particular,
\[
\int_{B_1^+}u_\va^2(x_n+\va)^{p-1}\,\ud x|_s\le \|(x_n+\va)^{p-1} u_\va^2\|_{L^{1}(B_1^+\times[-1,s])}+C \sum_{j=0}^n\|f_j\|^2_{L^2(Q_1^+)}.
\]
By Gronwall's inequality, 
\[
\|(x_n+\va)^{p-1} u_\va^2\|_{L^{1}(Q_1^+)}\le C \sum_{j=0}^n\|f_j\|^2_{L^2(Q_1^+)}.
\]
Thus, 
\begin{equation}\label{eq:energyestimatewithoutu}
\|u_\va\|^2_{V^1_0(Q_1^+)}\le C \sum_{j=0}^n\|f_j\|^2_{L^2(Q_1^+)}.
\end{equation}
Therefore, there exists $u\in  V_0^1(Q_1^+)$ and  a subsequence $\{u_{\va_j}\}$,  such that $u_{\va_j} \rightharpoonup u$ weakly  in (weak $*$ topology of) $V_0^1(Q_1^+)$. Then one can verify that this $u$ satisfies \eqref{eq:energyestimateunoL2} and \eqref{eq:definitionweaksolution}. Using that $a\in C(\overline Q_1^+)$, one can further verify that $u\in C ([-1,0]; L^2(B_1^+,x_n^{p-1}\ud x)) $.

The uniqueness follows from Theorem \ref{thm:uniquenessofweaksolution}.
\end{proof}

\begin{thm}\label{thm:energyestimateut}
Assume all the assumptions in Theorem \ref{thm:existenceofweaksolution}. We also  suppose that  the coefficients $A,b_1$ are independent of $t$ and are of $C^2(\overline B_1^+)$, and $a, b_2$ are of $C^1(\overline Q_1^+)$, and $F\equiv 0$. Let $u$ be the weak solution of \eqref{eq:linear-eq-app} with the full boundary condition $u\equiv 0$ on $\pa_{pa} Q_1^+$. Then
\begin{equation}\label{eq:energyestimateut}
\int_{Q_1^+} u^2\,\ud x\ud t+\int_{Q_1^+} x_n^{p-1}|\pa_t u|^2\,\ud x\ud t\le C \int_{Q_1^+} x_n^{p-1}f_0^2\,\ud x\ud t,
\end{equation}
where $C>0$ depends only on  $\lda, \bar\lambda$ and $\|b_2\|_{L^\infty(Q_1^+)}$.
\end{thm}
\begin{proof}
For $\va>0$, let $f_\va\in C_c^\infty(Q_1^+)$ such that $f_\va\to f_0$ in $L^2(Q_1^+)$ as $\va\to 0$. 

Let $u_\va\in C ([-1,0]; L^2(B_1^+)) \cap  L^2((-1,0];H_{0}^1(B_1^+))$ be the unique weak solution of
\be\label{eq:degiorgilinearappappendix2}
\begin{split}
&a\cdot  (x_n+\va)^{p-1} \pa_t u_\va -\mathrm{div}(A\nabla u_\va)+b_1 u_\va+b_2\cdot (x_n+\va)^{p-1}u_\va\\
&\quad=(x_n+\va)^{p-1}f_\va \quad \mbox{in }Q_1^+.
\end{split}
\ee
with  $u_\va\equiv 0$ on $\pa_{pa} Q_1^+$. By the Schauder regularity theory, we know that $D_{x} u_\va, \pa_t u_\va\in C(\overline Q_1^+)$.

For small $h>0$, denote $$u_\va^h(x,t)=\frac{u_\va(x,t+h)-u_\va(x,t)}{h}$$ for all $-1\le t\le -h$, and denote the left hand side of \eqref{eq:degiorgilinearappappendix2} as $I(x,t)$. Then we have for  all $-1<t<-h$,
\begin{equation}\label{eq:partialttestfunction}
\begin{split}
&\int_{B_1^+\times(-1,t]} (I(x,s+h)+I(x,s))u_\va^h(x,s)\,\ud x\ud s\\
&=\int_{B_1^+\times(-1,t]}(x_n+\va)^{p-1}(f_\va (x,s+h)+f_\va (x,s))u_\va^h(x,s)\,\ud x\ud s.
\end{split}
\end{equation}
Using the symmetry of $A$, we have
\begingroup
\allowdisplaybreaks
\begin{align*}
&-\int_{B_1^+}\int_{-1}^t [\mathrm{div}(A\nabla u_\va(x,s+h))+ \mathrm{div}(A\nabla u_\va(x,s))]u_\va^h(x,s)\,\ud s\ud x\\
&=\frac{1}{h}\int_{B_1^+}\int_{-1}^t (A\nabla u_\va(x,s+h))\cdot \nabla u_\va(x,s+h)\,\ud s\ud x\\
&\quad -\frac{1}{h}\int_{B_1^+}\int_{-1}^t  (A\nabla u_\va(x,s))\cdot \nabla u_\va(x,s)]\,\ud s\ud x\\
&=\frac{1}{h}\int_{B_1^+}\int_{t}^{t+h} (A\nabla u_\va(x,s+h))\cdot \nabla u_\va(x,s+h)\,\ud s\ud x\\
&\quad - \frac{1}{h}\int_{B_1^+}\int_{-1}^{-1+h} (A\nabla u_\va(x,s))\cdot \nabla u_\va(x,s)\,\ud s\ud x\\
& \to \int_{B_1^+} (A\nabla u_\va(x,t))\cdot \nabla u_\va(x,t)\,\ud x \quad\mbox{as }h\to 0,
\end{align*}
\endgroup
where we used that $D_xu\in C^{0}(\overline B_1^+\times[-1,0])$ and $u(x,-1)\equiv 0$. 
Here, we used $u_\va^h(x,t)$ instead of $\pa_t u_\va$ to avoid involving $D_x\pa_t u_\va$ in the calculation.  Using similar arguments, by sending $h\to 0$ in \eqref{eq:partialttestfunction}, and using \eqref{eq:coer-app} and H\"older's inequality, we have
\begingroup
\allowdisplaybreaks
\begin{align*}
&\int_{B_1^+\times(-1,t]}a\cdot  (x_n+\va)^{p-1} |\pa_t u_\va|^2 \,\ud s\ud x +\bar\lambda \int_{B_1^+}  u_\va(x,t)^2\,\ud x \\
&\le  C \int_{B_1^+\times(-1,t]} (x_n+\va)^{p-1} (|u_\va|+|f_\va|)|\pa_tu_\va|\,\ud s\ud x\\
&\le \int_{B_1^+\times(-1,t]}(x_n+\va)^{p-1} \left[\frac{\lambda}{2}|\pa_t u_\va|^2  + C (|u_\va|^2+|f_\va|^2)\right]\,\ud s\ud x.
\end{align*}
\endgroup
Using $a\ge\lambda$, we have
\begin{equation}\label{eq:energyestimateutsupL2}
\begin{split}
&\int_{B_1^+\times(-1,t]} x_n^{p-1} |\pa_t u_\va|^2 \,\ud s\ud x+\int_{B_1^+}  u_\va(x,t)^2\,\ud x\\
&\le C\int_{B_1^+\times(-1,t]} u_\va^2\,\ud x\ud s+C\int_{Q_1^+} (x_n+\va)^{p-1}f_\va^2\,\ud x\ud s. 
\end{split}
\end{equation}
In particular, 
\[
\frac{\ud}{\ud t}\int_{B_1^+\times(-1,t]} u_\va^2\,\ud x\ud s\le C\int_{B_1^+\times(-1,t]} u_\va^2\,\ud x\ud s+C\int_{Q_1^+} (x_n+\va)^{p-1}f_\va^2\,\ud x\ud s.
\]
Since $u_\va(\cdot,0)\equiv 0$, by Gronwall's inequality, we have
\begin{align}\label{eq:L2control}
\int_{Q_1^+} u_\va^2\,\ud x\ud s\le C\int_{Q_1^+} (x_n+\va)^{p-1}f_\va^2\,\ud x\ud s.
\end{align}
Thus, we have
\begin{align}\label{eq:weakderivativeint}
&\int_{Q_1^+} x_n^{p-1} |\pa_t u_\va|^2 \,\ud s\ud x \le C\int_{Q_1^+} (x_n+\va)^{p-1}f_\va^2\,\ud x\ud s. 
\end{align}
Therefore, $\int_{(B_1^+\cap\{x_n>\delta\})\times(-1,0]} |\pa_t u_\va|^2 \le C(\delta)$ for every $\delta>0$. This implies the existence of weak derivative $\pa_t u$, and that $\pa_t u_\va$ weakly converges to $\pa_t u$ in $L^2((B_1^+\cap\{x_n>\delta\})\times(-1,0])$ for every $\delta$. Then, \eqref{eq:energyestimateut} follows from sending $\va \to 0$ in \eqref{eq:L2control} and \eqref{eq:weakderivativeint}.
\end{proof}

\section{Some useful lemma}
The following two iteration lemmas are used in the paper.
\begin{lem}[Lemma 4.3 in \cite{HL}]\label{lem:lemma4.3inHL}
Let $h(t)\ge 0$ be bounded in $[\tau_0,\tau_1]$ with $\tau_0\ge 0$. Suppose for $\tau_0\le t<s\le\tau_1$ we have
\[
h(t)\le \theta h(s)+\frac{A}{(s-t)^\alpha}+B
\]
for some $\theta\in [0,1), \alpha>0, A\ge 0, B\ge 0$. Then for any $\tau_0\le t<s\le\tau_1$, there holds
\[
h(t)\le C(\alpha,\theta)\left(\frac{A}{(s-t)^\alpha}+B\right).
\]
\end{lem}
\begin{lem}[Lemma 3.4 in \cite{HL}]\label{lem:lemma3.4inHL}
Let $h(t)$ be a nonnegative and nondecreasing function on $[0,R]$. Suppose that
\[
h(\rho)\le A\left[\Big(\frac{\rho}{r}\Big)^\alpha+\va\right] h(r)+Br^\beta
\]
for any $0<\rho\le r\le R$, with $A,B,\alpha,\beta$ nonnegative constants and $\beta<\alpha$. Then for any $\gamma\in [\beta,\alpha)$, there exists a constant $\va_0=\va_0(A,\alpha,\beta,\gamma)$ such that if $\va<\va_0$, we have for all $0<\rho\le r\le R$
\[
h(\rho)\le C\left[\Big(\frac{\rho}{r}\Big)^\gamma h(r)+B\rho^\beta \right].
\]
\end{lem}

The following lemma is Lemma 3.1 on page 78 in \cite{LSU}.

\begin{lem}[Lemma 3.1 on page 78 in \cite{LSU}]\label{lem:lemma3.1inLSU}
Let $\Omega\subset\R^n$ be a domain satisfying a cone condition, i.e., there exists a fixed spherical cone of some altitude $d$ such that, no matter at what point of $\overline\Omega$ its vertex is placed, the cone itself can be swung so that all of it is located in $\Omega.$ Let $T>0$ and $Q_T=\Omega\times(0,T]$. Let $u\in C(\overline Q_T)$ be differentiable in $x$, and satisfy
\begin{align*}
\sup_{x\in \Omega} \sup_{t,s\in[0,T],s\neq t}\frac{|u(x,s)-u(x,t)|}{|s-t|^\alpha}&\le\mu_1,\\
\sup_{t\in[0,T]} \sup_{x,y\in \Omega,x\neq y }\frac{|\nabla u(x,t)-\nabla u(y,t)|}{|x-y|^\beta}&\le\mu_2
\end{align*}
for some $\alpha,\beta\in (0,1)$, $\mu_1>0,\mu_2>0$. Then there exists $\mu_3>0$, depending only on $\alpha,\beta,\mu_1,\mu_2,d$ and the solid angle of the cone such that
\[
\sup_{x\in \Omega} \sup_{t,s\in[0,T],s\neq t}\frac{|\nabla u(x,s)-\nabla u(x,t)|}{|s-t|^\delta}\le\mu_3,
\]
where $\delta=\frac{\alpha\beta}{1+\beta}$.
\end{lem}

\section{Campanato's method for a uniformly parabolic equation}
This appendix is for the proof of Proposition \ref{prop:interior}, deriving certain estimates of solutions to a uniformly parabolic equation (see \eqref{eq:uniformlyparabolicafterscaling}) with a parameter $\ell$. We shall obtain estimates with explicit dependence of $\ell$.

Let $p>1$, $\ell\in(0,1)$ be constants, and $u\in C^{2+\gamma}(\overline Q_1^+)$ (for every $\gamma\in(0,1)$) be a solution of 
\begin{equation}\label{eq:uniformlyparabolicafterscaling}
\begin{split}
& a(x,t) \phi(x_n) \pa_t  u- \mathrm{div}( A(x) \nabla  u) +\ell^2  b_1(x)  u + \ell^{p+1}  b_2(x,t)\phi(x_n)  u \\
&\quad= \ell^{p+1}\phi(x_n)  f(x,t) \quad \mbox{in }Q_{3/4}, 
\end{split}
\end{equation}
where $\phi\in C^\infty((-1,1))$, $\lambda\le\phi\le \Lambda$ on $[-3/4,3/4]$, $\|\phi\|_{C^2(-3/4,3/4)}\le\Lambda$, $A=(a_{ij})_{n\times n}, a_{ij}\in C^3(\overline B_1), b_1\in C^3(\overline B_1)$, $A$ and $a$ satisfy the ellipticity condition \eqref{eq:ellip}, $b_2\ge \lambda$, 
\begin{align*}
&[a]_{C^{\alpha,\alpha/2}(Q_{3/4})}+ [b_2]_{C^{\alpha,\alpha/2}(Q_{3/4})}+[a_{ij}]_{C^{\alpha}(B_{3/4})}+ [b_1]_{C^{\alpha}(B_{3/4})}\le \ell^\alpha  \Lambda_1,\\
&[f]_{C^{\alpha,\alpha/2}(Q_{3/4})}\le \ell^\alpha  \Lambda_2
\end{align*}
for some $\alpha\in (0,1), \Lambda_1>0,\Lambda_2>0$, and there holds the coercivity: there exists $\bar\lambda>0$ such that
\[
\int_{B_1} [(A\nabla \varphi)\cdot \nabla \varphi +\ell^2b_1 \varphi^2]\,\ud x \ge \bar \lda \int_{\Omega} \varphi^2 \quad \quad \forall \varphi\in H^1_0(B_1).
\]
We denote
\[
(u)_R=\frac{1}{|Q_R|}\int_{Q_R} u(x,t)\,\ud x\ud t.
\]
Then we have the following two propositions.
\begin{prop}\label{prop:appendixCvariablecoefficientfort}
We have for $0<\rho\le R<1/2$ that
\begin{align*}
&\frac{1}{ (\rho\ell)^{2\al} }\dashint_{Q_\rho} |\pa _t u- (\pa _t u)_\rho|^2 \le \frac{C}{(R\ell)^{2\al} }  \dashint_{Q_R} |\pa _t u-(\pa _t u)_R|^2\\
&\quad\quad\quad\quad\quad+C\sup_{Q_{1/2}} |\pa_t u|^2  + C \ell^{2p+2}(\|u\|_{L^\infty(Q_{1/2})}+ \Lambda_2)^2, 
\end{align*} 
where $C>0$ depends only on $n,p,\alpha,\lambda,\Lambda,\bar\lambda, \Lambda_1, \|b_2\|_{L^\infty(Q_{3/4})}$ and the $C^2(\overline B_1)$ norms of $A$ and $b_1$ (but not on $\ell$).
\end{prop}

\begin{proof}
We only need to consider $0<\rho<R/4$, since otherwise the inequalities hold trivially.  

Let $u_1$ be the solution of 
\begin{align*}
(a)_R \phi\pa _t u_1 - \mathrm{div} (A \nabla u_1) +\ell^2 b_1 u_1 +\ell^{p+1} (b_2)_R\phi u_1 &= (f)_R\ell^{p+1}\phi\quad \mbox{in }Q_R\\
 u_1&=u \quad \mbox{on }\pa_{pa} Q_R,
\end{align*}
and $u=u_1+u_2$.  We have
\begingroup
\allowdisplaybreaks 
\begin{align*}
&\int_{Q_\rho} |\pa _t u- (\pa _t u)_\rho|^2 \\
& \le 2 \int_{Q_\rho} |\pa _t u_1- (\pa _t u_1)_\rho|^2  +2  \int_{Q_\rho} |\pa _t u_2|^2 \\
& \le C\left(\frac{\rho}{R}\right)^{n+4} \left[ \int_{Q_R} |\pa _t u_1-(\pa _t u_1)_R|^2  +\ell^4 R^2  \int_{Q_R} |\pa_t u_1|^2\right] + 2 \int_{Q_R} |\pa _t u_2|^2 \\&
\le C\left(\frac{\rho}{R}\right)^{n+4} \left[ \int_{Q_R} |\pa _t u-(\pa _t u)_R|^2  +\ell^4 R^2  \int_{Q_R} |\pa_t u|^2\right] + C \int_{Q_R} |\pa _t u_2|^2,  
\end{align*}
\endgroup
where we used Proposition \ref{prop:appendixCconstantcoefficient} in the second inequality. 
Note that 
\begin{align*}
&(a)_R \phi  \pa _t u_2 - \mathrm{div} (A \nabla u_2) +\ell^2 b_1 u_2 +\ell^{p+1} (b_2)_R \phi u_2 \\&=\ell^{p+1} (f-(f)_R)\phi-\ell^{p+1}(b_2-  (b_2)_R) u\phi- (a- (a)_R)\phi \pa_t u \quad \mbox{in }Q_R 
\end{align*}
and $u_2=0$ on $\pa_p Q_R$.  Similar to \eqref{eq:energyestimateut}, we have
\[
 \int_{Q_R} |\pa _t u_2|^2\le C\ell^{2\al} R^{2\al} \int_{Q_R } |\pa_t u|^2 +C \ell^{2p+2+2\al} R^{n+2+2\al} (\|u\|_{L^\infty(Q_{R})}^2+\Lambda_2^2). 
\]
Let 
\[
M= \sup_{Q_{1/2}} |\pa_t u|  + \ell^{p+1} \|u\|_{L^\infty(Q_{3/4})}+\ell^{p+1} \Lambda_2.
\] 
Using $\int_{Q_R} |\pa_t u|^2 \le C R^{n+2}  \sup_{Q_{1/2}} |\pa_t u|^2 $,  we obtain 
\begin{align*}
\int_{Q_\rho} |\pa _t u- (\pa _t u)_\rho|^2  
\le C\left(\frac{\rho}{R}\right)^{n+4}  \int_{Q_R} |\pa _t u-(\pa _t u)_R|^2  + C M^2 \ell^{2\al} R^{n+2+2\al}. 
\end{align*}
By Lemma \ref{lem:lemma3.4inHL}, we have 
\begin{align*}
\frac{1}{ (\rho\ell)^{2\al} }\dashint_{Q_\rho} |\pa _t u- (\pa _t u)_\rho|^2 
\le \frac{C}{(R\ell)^{2\al} }  \dashint_{Q_R} |\pa _t u-(\pa _t u)_R|^2  + C M^2. 
\end{align*}
This finishes the proof.
\end{proof}

\begin{prop}\label{prop:appendixCvariablecoefficientforx}
We have for $0<\rho\le R<1/2$ that
\begin{align*}
&\frac{1}{ (\rho\ell)^{2\al} }\dashint_{Q_\rho} |\nabla  u- (\nabla  u)_\rho|^2 
\le \frac{C}{(R\ell)^{2\al} }  \dashint_{Q_R} |\nabla  u-(\nabla  u)_R|^2\\
& +C\sup_{Q_{1/2}} |\pa_t u|^2 \ell^{-2\al}   +C\sup_{Q_{1/2}} |\nabla  u|^2 + C\ell^{2}\|u\|^2_{L^\infty(Q_{1/2})} +C\ell^{2}\|f\|_{C^{\alpha,\alpha/2}(Q_{1/2})}^2,
\end{align*} 
where $C>0$ depends only on $n,p,\alpha,\lambda,\Lambda,\bar\lambda, \Lambda_1, \|b_2\|_{L^\infty(Q_{3/4})}$ and the $C^2(\overline B_1)$ norms of $A$ and $b_1$ (but not on $\ell$).
\end{prop}
\begin{proof}
We only need to consider $0<\rho<R/4$, since otherwise the inequalities hold trivially.

Let $u_1$ be the solution of 
\begin{align*}
(a)_R \phi\pa _t u_1 - \mathrm{div} ((A)_R \nabla u_1) +\ell^2 (b_1)_R  u_1 &+\ell^{p+1} (b_2)_R \phi u_1 \\
&= (f)_R\ell^{p+1}\phi  \mbox{ in }Q_R
\end{align*}
satisfying  $u_1=u  \mbox{ on }\pa_p Q_R$ and let $u_2=u-u_1$. We have 
\begingroup
\allowdisplaybreaks
\begin{align*}
&\int_{Q_\rho} |\nabla u- (\nabla u)_\rho|^2 \\
& \le 2 \int_{Q_\rho} |\nabla  u_1- (\nabla u_1)_\rho|^2  +2  \int_{Q_\rho} |\nabla  u_2|^2 \\
& \le C\left(\frac{\rho}{R}\right)^{n+4} \left[ \int_{Q_R} |\nabla  u_1-(\nabla u_1)_R|^2  +\ell^4 R^2  \int_{Q_R} |\nabla  u_1|^2\right] +C R^2 \int_{Q_R} |\pa_t u|^2 \\
& \quad + C R^{n+4} (\ell^2\|u\|_{L^\infty(Q_{1/2})} +\ell^{p+1}\|f\|_{L^\infty(Q_{1/2})})^2 + 2 \int_{Q_R} |\nabla  u_2|^2 \\
&\le C\left(\frac{\rho}{R}\right)^{n+4} \left[ \int_{Q_R} |\nabla  u-(\nabla u)_R|^2  +\ell^4 R^2  \int_{Q_R} |\nabla  u|^2\right] +C R^2 \int_{Q_R} |\pa_t u|^2\\
& \quad  + C R^{n+4} (\ell^2\|u\|_{L^\infty(Q_{1/2})} +\ell^{p+1}\|f\|_{L^\infty(Q_{1/2})})^2 + C \int_{Q_R} |\nabla  u_2|^2,
\end{align*}
\endgroup
where we used Proposition \ref{prop:appendixCconstantcoefficient} in the second inequality. Note that 
\begin{align*}
&(a)_R \phi\pa _t u_2 -\mathrm{div} ((A)_R \nabla u_2) +\ell^2 (b_1)_R u_2 +\ell^{p+1} (b_2)_R\phi u_2 \\
&= \ell^{p+1}\phi (f-(f)_R)-\ell^{p+1}(b_2-  (b_2)_R) \phi u- (a- (a)_R)\phi \pa_t u \\
&+\mathrm{div}  ((A-(A)_R) \nabla u) -  \ell^2 (b_1-(b_1)_R) u  \quad \mbox{in }Q_R 
\end{align*}
and $u_2=0$ on $\pa_p Q_R$.  Similar to \eqref{eq:energyestimateunoL2},  it follows that 
\begin{align*}
 \int_{Q_R} |\nabla  u_2|^2&\le C\ell^{2\al} R^{2\al} \int_{Q_R }( |\pa_t u|^2 +|\nabla u|^2) \\
 &\quad+ C \ell^{2\al} R^{n+2+2\al} (\ell^{2}\|u\|_{L^\infty(Q_{1/2})}+\ell^{p+1}\Lambda_2)^2 . 
\end{align*}
Let 
\[
M_1= \sup_{Q_{1/2}} |\pa_t u| \ell^{-\al}   +\sup_{Q_{1/2}} |\nabla  u| + \ell^{2-\alpha}\|u\|_{L^\infty(Q_{1/2})} +\ell^{p+1-\alpha}\|f\|_{C^{\alpha,\alpha/2}(Q_{1/2})} .
\] 
Using $\int_{Q_R} |\pa_t u|^2 \le C R^{n+2}  \sup_{Q_{1/2}} |\pa_t u|^2 $, $\int_{Q_R} |\nabla  u|^2 \le C R^{n+2}  \sup_{Q_{1/2}} |\nabla u|^2$,  we obtain 
\begin{align*}
\int_{Q_\rho} |\nabla  u- (\nabla  u)_\rho|^2  
\le C\left(\frac{\rho}{R}\right)^{n+4}  \int_{Q_R} |\nabla  u-(\nabla  u)_R|^2  + C M_1^2 \ell^{2\al} R^{n+2+2\al}. 
\end{align*}
By Lemma \ref{lem:lemma3.4inHL}, we have 
\begin{align*}
\frac{1}{ (\rho\ell)^{2\al} }\dashint_{Q_\rho} |\nabla  u- (\nabla  u)_\rho|^2 
\le \frac{C}{(R\ell)^{2\al} }  \dashint_{Q_R} |\nabla  u-(\nabla  u)_R|^2  + C M_1^2. 
\end{align*} 
This finishes the proof.
\end{proof}

In the proof of the above two propositions, we used the following result when (part of) the coefficients of the equation are constants.
\begin{prop}\label{prop:appendixCconstantcoefficient}
Let $0<\rho\le R<1/2$, and suppose the equation \eqref{eq:uniformlyparabolicafterscaling} holds in $Q_R$.
 
(i) Suppose $a=\bar a$, $b_2=\bar b$ and $f=\bar f$ for some constants $\bar a,\bar b, \bar f$. Then we have 
\[
\int_{Q_\rho} |\pa _t u- (\pa _t u)_\rho|^2  \le C \left(\frac{\rho}{R}\right)^{n+4} \left[ \int_{Q_R} |\pa _t u-(\pa _t u)_R|^2  +\ell^4 R^2  \int_{Q_R} |\pa_t u|^2\right],
\] 
where $C>0$ depends only on $n,\lambda,\Lambda,|\bar b|$ and the $C^2(\overline B_1)$ norms of $A$ and $b_1$ (but not on $\ell$).

(ii) Suppose additionally that $(a_{ij})$  and $b_1$ are also constants. Then we have 
\begin{align*}
\int_{Q_\rho} |\nabla u- (\nabla u)_\rho|^2  &\le  C \left(\frac{\rho}{R}\right)^{n+4} \left[ \int_{Q_R} |\nabla u-(\nabla  u)_R|^2  +\ell^4 R^2  \int_{Q_R} |\nabla u|^2\right]\\&
+C R^2 \int_{Q_R} |\pa_t u|^2 + C \ell^4 R^2 \int_{Q_R} |u|^2 + C\ell^{2p+2} R^{n+4}|\bar f|^2,
\end{align*}
where $C>0$ depends only on $n,\lambda,\Lambda,|\bar b|$ and the $C^2(\overline B_1)$ norms of $A$ and $b_1$ (but not on $\ell$).
\end{prop}
\begin{proof}
We only need to consider $0<\rho<R/4$, since otherwise the inequalities hold trivially. 

(i) Let $v=\pa_t u$. Then
\[
\bar a \phi(x_n) \pa_t  v- \mathrm{div}( A(x) \nabla  v) +\ell^2  b_1(x)  v + \ell^{p+1} \bar b \phi(x_n)  v = 0 \quad \mbox{in }Q_{R}.
\]
By the Poincar\'e inequality
\begin{align}\label{eq:appendixCpoincare}
\int_{Q_\rho} |v- (v)_\rho|^2  &\le C \left(\rho^2 \int_{Q_\rho} |\nabla v|^2 +\rho^4\int_{Q_\rho} |\pa_t v|^2  \right).
\end{align}
Since
\begin{align*}
&\bar a \phi \pa_t\big( v- (v)_R\big)-\mathrm{div}\big( A\cdot\nabla (v-(v)_R)\big) +\ell^2 b_1 \big(v-(v)_R\big)\\
& + \ell^{p+1} \bar b \phi(v-(v)_R)= g \quad \mbox{in }Q_{R}, 
\end{align*}
where $g= -(v)_R\,\ell^2b_1-(v)_R\,\ell^{p+1} \bar b\phi$. By the local gradient estimates for uniformly parabolic equations,  we have
\begin{align*}
\sup_{Q_\rho} |\nabla v|^2 & \le \frac{C}{ R^{n+4}} \int_{Q_{R}} (v-(v)_R)^2 +  CR^2\|g\|_{L^\infty (Q_R)}^2\\& \le \frac{C}{ R^{n+4}} \left(\int_{Q_{R}} (v-(v)_R)^2 +  \ell^4 R^4 \int_{Q_R} v^2\right), 
\end{align*}
where we used $|(v)_R|^2\le \frac{C(n)}{R^{n+2}} \int_{Q_R} v^2 $ in the last inequality. Similarly,
\begin{align*}
\sup_{Q_\rho} |\pa_t v|^2 & \le \frac{C}{ R^{n+6}} \int_{Q_{R}} (v-(v)_R)^2 +  C\|g\|_{L^\infty (Q_R)}^2\\& \le \frac{C}{ R^{n+6}} \left(\int_{Q_{R}} (v-(v)_R)^2 +  \ell^4 R^2 \int_{Q_R} v^2\right).
\end{align*}
Plugging these two estimates into \eqref{eq:appendixCpoincare}, the conclusion of (i) follows.

(ii). Suppose additionally that $(a_{ij})$ and $b_1$ are constants. For $i=1, 2, \cdots, n-1$, $D_{x_i} u$ satisfies the same equation as $v$ in (i). Hence, the estimates for $D_{x_i} u$ follows in the same way as for $\pa_t u$ in (i). For $D_{x_n}u$, first, we have from the  Poincar\'e inequality that 
\[
\int_{Q_{\rho}}|D_{x_n}u- (D_{x_n}u)_\rho|^2 \le C (\rho^2 \int_{Q_\rho} |D_xD_{x_n}u|^2 + \rho^4 \int_{Q_\rho} |\pa_t D_{x_n}u|^2 ). 
\]
Second, by the equation of $u$, we have 
\[
|D_{x_n}^2 u|^2\le C(| \pa_t u|^2 +|\nabla \nabla_{x'} u|^2 + \ell^4 u^2 +  \ell^{2p+2} |\bar f|^2).
\]
Third, by the gradient estimates for the equation of $v$, we have 
\[
\rho^2 \int_{Q_\rho} |D_{x_n} \pa_t u|^2 \le C \int_{Q_{2\rho}} |\pa_t u|^2.
\]
Fourthly, by the gradient estimates for the equation of $D_{x_i} u$, we have 
\begin{align*}
\sup_{Q_\rho} |D_x D_{x_i} u|^2 &  \le \frac{C}{ R^{n+4}} \left(\int_{Q_{R}} (D_{x_i} u-(D_{x_i} u)_R)^2 +  \ell^4 R^4 \int_{Q_R} |D_{x_i} u|^2\right).
\end{align*}
Combining these four inequalities, the estimate in (ii) follows.
\end{proof}

\begin{rem}\label{rem:interiorbottom}
If one assumes that the equation \eqref{eq:uniformlyparabolicafterscaling} holds in $B_{3/4}\times[0,9/16]$ and $u(\cdot,0)\equiv 0$, then one can show similar estimates up to the bottom: 
\begin{align*}
&\frac{1}{ (\rho\ell)^{2\al} }\dashint_{Q_\rho} |\pa _t u- \pa_t u(0,0)|^2 \le \frac{C}{(R\ell)^{2\al} }  \dashint_{Q_R} |\pa _t u- \pa_t u(0,0)|^2\\
&\quad\quad\quad\quad\quad+C\sup_{Q_{1/2}} |\pa_t u|^2  + C \ell^{2p+2}(\|u\|_{L^\infty(Q_{1/2})}+ \Lambda_2)^2, \\
&\frac{1}{ (\rho\ell)^{2\al} }\dashint_{Q_\rho} |\nabla  u|^2 
\le \frac{C}{(R\ell)^{2\al} }  \dashint_{Q_R} |\nabla  u|^2\\
& +C\sup_{Q_{1/2}} |\pa_t u|^2 \ell^{-2\al}   +C\sup_{Q_{1/2}} |\nabla  u|^2 + C\ell^{2}\|u\|^2_{L^\infty(Q_{1/2})} +C\ell^{2}\|f\|_{C^{\alpha,\alpha/2}(Q_{1/2})}^2,
\end{align*} 
where $Q_R=B_R\times[0,R^2]$, $\pa_t u(0,0)=\ell^{p+1}f(0,0)/a(0,0)$. Consequently, if one has $|f(0,0)|\le \ell^\alpha\Lambda_2$, then we have
\begin{align*}
&\frac{1}{ (\rho\ell)^{2\al} }\dashint_{Q_\rho} |\pa _t u- (\pa_t u)_\rho|^2 \le \frac{2^{2\alpha}C}{\ell^{2\al} }  \dashint_{Q_{1/2}} |\pa _t u|^2\\
&\quad\quad\quad\quad\quad+C\sup_{Q_{1/2}} |\pa_t u|^2  + C \ell^{2p+2}(\|u\|_{L^\infty(Q_{1/2})}+ \Lambda_2)^2.
\end{align*} 
This remark will be used to obtain estimates up to the bottom in Theorem \ref{thm:global-schauder1} when the compatibility condition is enforced. We omit the details.
\end{rem}

\bigskip

\noindent T. Jin

\noindent Department of Mathematics, The Hong Kong University of Science and Technology\\
Clear Water Bay, Kowloon, Hong Kong\\[1mm]
Email: \textsf{tianlingjin@ust.hk}

\medskip

\noindent J. Xiong

\noindent School of Mathematical Sciences, Beijing Normal University\\
Beijing 100875, China\\[1mm]
Email: \textsf{jx@bnu.edu.cn}

\end{document}